\documentclass[10pt, a4]{amsart}
\usepackage{lmodern}
\usepackage[utf8]{inputenc}
\usepackage{amsmath}
\usepackage{graphicx}
\usepackage{amssymb}
\usepackage{esint}
\usepackage[dvipsnames]{xcolor}
\usepackage{tikz}
\usepackage{xxcolor}
\usepackage{floatrow}
\usepackage{color}
\usepackage{amsthm}
\usepackage{epsfig}
\usepackage[english]{babel}
\usepackage{hyperref}
\usepackage[normalem]{ulem}
\usepackage{tikz}
\usepackage{pgfplots}
\usepackage{subfigure}
\usepackage{lineno}
\usepackage{caption}
\usepackage{bbm}

\usepackage{etex}
\newcount\scratchcounter
\usepackage{amsmath}
\usepackage[utf8]{inputenc}
\DeclareUnicodeCharacter{2010}{-}

\usepackage[utf8]{inputenc}
\DeclareUnicodeCharacter{2011}{-}

\usepackage{stmaryrd}

\usepackage{float}

\hypersetup{
	colorlinks,
	linkcolor={red!80!black},
	citecolor={blue!50!black},
	urlcolor={blue!80!black}
}
\usepackage{diagbox}

\usepackage[a4paper,top=1 cm,bottom=3 cm,left=1 cm,right=1 cm]{geometry}

\theoremstyle{plain}
\newtheorem{lemma}{Lemma}[section]
\newtheorem{proposition}[lemma]{Proposition}
\newtheorem{theorem}[lemma]{Theorem}

\theoremstyle{definition}
\newtheorem{definition}[lemma]{Definition}

\theoremstyle{remark}
\newtheorem{remark}[lemma]{Remark}

\makeatletter
\renewenvironment{proof}[1][\proofname]{\par
  \pushQED{\qed}%
  \normalfont \topsep6\p@\@plus6\p@\relax
  \trivlist
  \item[\hskip\labelsep
        \itshape
    #1\ \ref{#1}.]\ignorespaces
}{%
  \popQED\endtrivlist\@endpefalse
}
\makeatother

\numberwithin{equation}{section}

\def\to{\rightarrow}

\usepackage{dsfont}

\newcommand{\cN}{\mathcal{N}}

\usepackage{mathtools}

\allowdisplaybreaks
\usepackage[colorinlistoftodos]{todonotes}
\usepackage{url}

\usepackage{graphicx,tikz} 

\title[Birth  control and Turnpike property]{Birth  control and turnpike property of Lotka-McKendrick models with diffusion}

\author{Marius Bargo}
\address[Marius Bargo]{Laboratoire LANIBIO, Université Joseph Ki-Zerbo, 01 BP 7021, Ouaga 01, Ouagadougou, Burkina Faso}
\email{bargomarius@gmail.com}

\author{Yacouba Simporé}
\address[Yacouba Simporé]{Chair for Dynamics, Control, Machine Learning and Numerics, Alexander Von Humboldt- Professorship, Department of Mathematics, Friedrich-Alexander-Universit\"{a}t at Erlangen-N\"{u}rnberg, Cauerstrasse 11, 91058 Erlangen, Germany, Université Yembila Abdoulaye TOGUYENI, Burkina Faso, Laboratoire LaST,  Laboratoire LANIBIO (UJKZ)}
\email{yacouba.simpore@fau.de}
\usepackage{thmbox}
\usepackage{dsfont} 
\usepackage{float}
\usepackage{appendix}
\begin{document}
\maketitle
{\bf Abstract.}  In this paper, we study the turnpike property in age-structured population dynamics with a birth control. These models describe the temporal evolution of one or more populations, incorporating age dependence and spatial structure. To this end, we first establish the null controllability of the system : we prove that, for any \(T>A\) and any initial datum in \(L^2(\Omega\times(0,A))\), the population can be driven to zero using control functions that are spatially localized in \(t\) but act only at age \(a=0\). We then show that, although this control is initially applied only at birth, it can be reformulated as a distributed control, and we demonstrate that the resulting control operator is admissible in the state space.  Thus, to prove the turnpike property we combine our null-controllability results with Phillips’ theorem on exponential stability to design a suitable dichotomy transformation based on solutions of the algebraic Riccati and Lyapunov equations. Finally, we present numerical examples that substantiate our analytical findings.
\\{\bf Keywords} : Null-controllability ;  exponential stability ;  Riccati operator ; turnpike property ; population dynamics.
\section{Introduction}
Many models of population dynamics have been proposed, including those developed by Pierre-François Verhulst and Lotka–McKendrick. One of the earliest mathematical models of population growth was presented in $1798$ by Thomas R. Malthus, who juxtaposed the exponential increase of populations with the limitations imposed by available resources. In recent decades, the field has witnessed significant advancements.

The growing interest in population dynamics is largely driven by global demographic growth, prompting policymakers to implement control measures. The baby boom, characterized by a surge in births after World War II, profoundly transformed global demographic dynamics. This rapid population increase led to challenges such as overpopulation, depletion of natural resources, and accelerated aging of populations. To address these issues, policies like China's one-child policy, inspired by Song Jian's research in 1976, were implemented.

These demographic challenges demand rigorous management to prevent crises related to the workforce, pension systems, and equitable access to resources. Studies in population dynamics aim not only to explain population evolution and address issues like extinction or stability but also to determine an optimal population level, often defined as an "ideal population," at minimal cost, within the framework of control theory.

The turnpike theory, developed by three preeminent \textsc{\romannumeral 20}\textsuperscript{th} century economists, Paul Samuelson, Robert Solow and Robert Dorfman, offers a pertinent tool in this context. Inspired by Von Neumann's maximal growth model, it optimizes dynamic trajectories by rapidly bringing the system close to an ideal stationary state before reaching the final objective. This theory establishes that, in the long term, optimal solutions to a control problem follow a trajectory close to that of an associated stationary problem.


The turnpike phenomenon refers to the tendency of optimal trajectories, over long but finite time horizons, to approach and remain close to a steady state for most of the time. This behavior was first observed by von Neumann in the context of optimal growth strategies toward economic equilibria and was later explicitly named in \cite{ref11}.

In modern macroeconomics, the turnpike property implies that when an economic agent aims to transition an economy from one capital level to another, the most efficient strategy, given a sufficiently long time horizon, is to rapidly adjust the capital stock to a level near a stationary optimal path before finally reaching the target. This concept, inherent to resource allocation strategies, has also been demonstrated in partial differential equations arising in mechanics, supporting the idea that for an optimal control problem defined over an extended time horizon, any optimal solution will remain close to the corresponding static optimal solution for most of the interval.

The turnpike phenomenon reveals a quasi-stationary structure in dynamic systems. In particular, when the time horizon \(T\) is sufficiently long, the optimal trajectory of the dynamic system remains near the steady-state solution of the associated static problem. This insight is particularly significant in population dynamics, where achieving an ideal population structure is crucial, and it provides a robust framework for characterizing population behavior and designing effective stabilization strategies.

We analyze a model of population dynamics to underscore the exponential decay of dynamic solutions toward the steady state. We consider the following system:
\begin{equation}\label{eq1.1}
\left\lbrace 
\begin{array}{ll}
\partial_{t}y(x,a,t)+\partial_{a}y(x,a,t)-\triangle y(x,a,t)+\mu(a) y(x,a,t)=0 & \text{ in }  Q,\\
 \\\partial_{\nu}y(x,a,t)=0 &\text{ on }  \Sigma,\\
 \\y(x,0,t)=\mathds{1}_{\omega}(x)v(x,t)+\displaystyle\int\limits_{0}^{A}\beta(x,a)y(x,a,t)da &\text{ in }  Q_{T},\\
 \\y(x,a,0)=y_{0}(x,a)&\text{ in } Q_A,
\end{array}
\right.
\end{equation}
with \( \mathds{1}_{\omega}(x) \) the characteristic function of the open subset \( \omega \subset \Omega \), where the control \( v \) is assumed to act.

In the above model, 
\begin{itemize}
    \item  \( Q = \Omega \times [0, A] \times [0, T] \), \( \Sigma = \partial\Omega \times [0, A] \times [0, T] \), \( Q_{T} = \Omega \times [0, T] \), and  \( Q_{A} = \Omega \times [0, A] \).
    \item  \( \Omega \subset \mathbb{R}^N \)  (\(N\in\mathbb{N}^*\)) the  bounded open set with a regular boundary representing the spatial domain occupied by the individuals. The operator \( \triangle \) denotes the Laplacian with respect to the spatial variable \( x \), modeling the diffusion or spatial displacement of the population.
    \item   \( y(x, a, t) \) represents the density of individuals located at \( x \in \Omega \), of age \( a \geq 0 \), at time \( t \geq 0 \).
    \item  \( y_0(x, a) \geq 0\) is the initial population distribution.
    \item  \( A \in [0, \infty) \) denotes the expected lifespan.
    \item  \( \beta(x, a) \) and \( \mu(a) \) are nonnegative functions representing the fertility and mortality rates, respectively, with \( \beta \) depending on both space and age.
    \item  The penultimate equation, which describes the birth process, is also known as the renewal equation; it provides a nonlocal condition that links the newborns to individuals of reproductive age: $\displaystyle\int\limits_{0}^{A}\beta(x,a)y(x,a,t)da$.
    \item  We impose homogeneous Neumann boundary conditions, reflecting an isolated population in which individuals remain confined to the domain \( \Omega \) : $\partial_{\nu}y(x,a,t)=0$.
\end{itemize}
In the sequel, we assume that the fertility rate $\beta$ and the mortality rate $\mu$ satisfy the demographic property :
\begin{align*}
(H_1):
\left\lbrace\begin{array}{l}
	\mu\geq 0 \text{ a.e. in } (0,A),\\
	\mu\in L^{1}_{loc}(0,A)\quad \displaystyle\int_{0}^{A}\mu(a)da=+\infty,
\end{array} \right.
\end{align*}

\begin{align*}
(H_2):\left\lbrace\begin{array}{l}
	\beta\in C^{1}(\Omega\times[0,A]) \text{ and }\beta(x,a)\geq 0\text{ }\forall (x,a)\in \Omega\times[0,A],
\end{array}
\right.
\end{align*}
and 
\begin{align*}
(\cN ull_{\beta}):\left\lbrace\begin{array}{l}
	\beta(x,a)=0 \text{ } \forall a\in (0,a_b) \text{ where } 0<a_0\leq a_b<A. \\
\end{array}
\right.
\end{align*}
We introduce also the following function 
\begin{equation*}
    \pi(a)=e^{-\displaystyle\int_0^a\mu(s)ds}
\end{equation*}
which is the probability of survival of an individual from age $0$ to $a.$  


In this paper, we investigate two key aspects of an infinite-dimensional linear system modeling age-structured population dynamics with spatial distribution, with a particular focus on control applied at birth : the null controllability and the turnpike property.  

 The model under consideration features a boundary control acting at birth (i.e., localized at age zero). As a first step, we study the well-posedness of the system and the admissibility of the control operator, by showing that the boundary control problem can be reformulated as an equivalent system with an internal source term, which admits a well-posed formulation.  

Next, we establish the null controllability of the system for any time \( T > A \). Finally, we prove that the optimal control problem associated with this system satisfies the exponential turnpike property in the context of infinite-dimensional continuous-time linear-quadratic control, without any constraint on the control input or the initial condition. A numerical simulation is presented to illustrate and support the theoretical results.
\section{well-posedness and the admissibility of control operators}\label{s2}
\paragraph{\bf Notation}
Let us introduce the notation used throughout this paper. We denote by \(\|\cdot\|\) (resp.\ \(\langle\cdot,\cdot\rangle\)) the norm (resp.\ inner product) on the Hilbert space 
\[
H = L^2\bigl(0,A;L^2(\Omega)\bigr).
\]
When we need to emphasize that a norm is taken in a smaller domain \(\omega\subset\Omega\), we write \(\|\cdot\|_{L^2(\omega)}\). The symbol \(\mathcal{L}(X,Y)\) designates the space of bounded linear operators from the Hilbert space \(X\) to the Hilbert space \(Y\). We also set
\[
Q_1 = \Omega\times[0,A],\quad
\Sigma_1 = \partial\Omega\times[0,A],\quad
\delta_0\text{ the Dirac mass at }0,\quad
I\text{ the identity on }H.
\]

In the context of functional analysis applied to dynamical systems and control theory, let \(\mathcal A\) be a (possibly unbounded) linear operator on \(H\). Its resolvent set \(\rho(\mathcal A)\subset\mathbb C\) is defined by
\[
\rho(\mathcal A)
=\bigl\{\lambda\in\mathbb C\mid\lambda I-\mathcal A
\text{ is invertible on }H\bigr\},
\]
and its resolvent is 
\[
R(\lambda;\mathcal A)=(\lambda I-\mathcal A)^{-1}\in\mathcal L(H).
\]
The spectrum \(\sigma(\mathcal A)\) is the complement of \(\rho(\mathcal A)\) in \(\mathbb C\).

Next, choosing any fixed \(\lambda\in\rho(\mathcal A)\), we define
\begin{itemize}
    \item the extended space \(H_{-1}\) as the completion of \(H\) under the norm
  \[
  \|x\|_{H_{-1}}:=\bigl\|R(\lambda;\mathcal A)\,x\bigr\|_{H},
  \quad x\in H,
  \]
\item  the restricted space \(H_{1}=D(\mathcal A)\) equipped with the norm
  \[
  \|z\|_{H_{1}}:=\bigl\|(\lambda I-\mathcal A)\,z\bigr\|_{H},
  \quad z\in D(\mathcal A).
  \]
\end{itemize}
With these definitions, one obtains the continuous embeddings
\[
H_1\;\subset\;H\;\subset\;H_{-1}.
\]
In addition, throughout the paper, we refer to the control that drives the state to zero as the “null-control.”


Now, we provide some basic results on the population semigroup for the linear age structured model with diffusion. We give the existence of the semigroup in the Hilbert space $H.$ For it, we define the operator 
\begin{equation*}
\mathcal{A}_m: \mathcal{D}(\mathcal{A}_m)\subset H \longrightarrow H, \qquad \mathcal{A}_m \varphi = \mathcal{A}_1 \varphi + \mathcal{A}_2 \varphi
\end{equation*}
with
\begin{align*}
\mathcal{A}_1 \varphi = \Delta \varphi \quad &\text{(the diffusion operator)} \\
\mathcal{A}_2 \varphi = -\dfrac{\partial \varphi}{\partial a} - \mu(a) \varphi  \quad &\text{(the operator without diffusion)}
\end{align*}
where
\begin{align*}
\mathcal{D}(\mathcal{A}_2) = \Big\{ \varphi \in H \, \big| \, \varphi(x, \cdot) &\text{ is locally absolutely continuous on } [0, A), \\
&\varphi(x, 0) = \int_{0}^{A} \beta(x, a) \varphi(x, a) \, da \text{ a.e. on } \Omega, \\
&\dfrac{\partial \varphi}{\partial a} + \mu(a) \varphi \in H \Big\}
\end{align*}
and
\begin{align*}
\mathcal{D}(\mathcal{A}_1) = \Big\{ \varphi \in L^2(0, A; H^2(\Omega)) \, \big| \, \dfrac{\partial \varphi}{\partial \nu} = 0 \Big\}.
\end{align*}
Finally, \begin{align*}
    \mathcal{D}(\mathcal{A}_m)=\mathcal{D}(\mathcal{A}_1)\cap\mathcal{D}(\mathcal{A}_2).
\end{align*}

We recall some fundamental notions concerning the semigroups associated with its operator. 

\begin{lemma}\label{le1.1}
The operator \((\mathcal{A}_m, D(\mathcal{A}_m))\) is the infinitesimal generator of a strongly continuous semigroup \(\mathcal{T} = (\mathcal{T}_t)_{t\geq 0}\) on \(H\).
\end{lemma}
\begin{proof}
    See \cite{ref72}.
\end{proof}
Since \(\mathcal{A}_m\) generates a \(C_0\)-semigroup of linear operators on \(H\), we can define the mild solution of \eqref{eq1.1} accordingly, as will be detailed in subsequent sections.

Furthermore, we introduce the adjoint operator \(\mathcal{A}_m^*\). As specified in \cite{ref3}, the adjoint of \((\mathcal{A}_m, D(\mathcal{A}_m))\) in \(H\) is characterized by
\[
D(\mathcal{A}_m^*) = D(\mathcal{A}_0), \qquad \mathcal{A}_m^* \psi = \frac{\partial \psi}{\partial a} - \mu(a)\psi + \beta(x,a)\psi(x,0) + \triangle \psi,
\]
where the operator \(\mathcal{A}_0\) is defined as
\begin{align*}
 D(\mathcal{A}_0)=\{\varphi\in H|\varphi(x,.)\quad\hbox{is locally absolutely continuous on}\,[0,A),\\
 \lim_{a\rightarrow A}\varphi(x,a)=0 \quad\hbox{a.e. on}\, \Omega,\dfrac{\partial\varphi}{\partial\nu}=0 \,\text{on}\,\partial\Omega\,; \dfrac{\partial \varphi}{\partial a}-\mu(a)\varphi+\triangle\varphi\in H\}
\end{align*} 
with
\[
\mathcal{A}_0 \varphi = \frac{\partial \varphi}{\partial a} - \mu(a)\varphi + \triangle \varphi \quad (\varphi \in D(\mathcal{A}_0)).
\]

Since system \eqref{eq1.1} constitutes a boundary control problem,  we rewrite it in the following boundary input–output form:
\begin{equation}\label{e2.1}
\left\lbrace 
\begin{array}{ll}
\partial_{t}y(x,a,t)=\mathcal{A} y(x,a,t) & \text{ in }  Q,\\
\\y(x,a,0)=y_{0}(x,a)&\text{ in } Q_A,\\
\\(\delta_0-\Lambda) y(x,a,t)=\mathds{1}_{\omega}(x)v(x,t) &\text{ in }  Q_{T},\\
\\z(x,t)=\Lambda y(x,a,t)=\displaystyle\int_0^A\beta(x,a)y(x,a,t)da
\end{array}
\right.
\end{equation}
and  the operator 
\begin{align*}
    \mathcal{A} = \mathcal{A}_{m \mid H_1} \in \mathcal{L}(H_1, H)\qquad\text{generates a $C_0$-semigroup}\; \mathbb{T}=\left(\mathbb{T}(t)\right)_{t\geq 0}\; \text{on}\; H_1= \ker (\delta_0),
\end{align*}
(see \cite{ref78} for more details)  where
\begin{align*}
\mathcal{D}(\mathcal{A}) = \Big\{ \varphi \in D(\mathcal{A}_m) \, \mid \, \delta_0\varphi = 0 \Big\}.
\end{align*} 
The generator \(\mathcal{A}_m\) reduces to \(\mathcal{A}\) when \(\Lambda= 0\). More generally, \(\mathcal{A}_m\) is obtained by perturbing the domain \(D(\mathcal{A})\) via the operator \(\Lambda\). We therefore refer to \(\Lambda\) as the boundary perturbation of \(\mathcal{A}\).
  
Based on all the above, we introduce below the notion of an abstract non-homogeneous boundary problem.
\begin{definition}\cite{ref48}\label{d2.2}
 An abstract non-homogeneous boundary problem on \( L^2(\omega) \), \( D(\mathcal{A}_m) \), and \( H \) is defined as any pair \( (\mathcal{A}_m, \delta_0) \in \mathcal{L}(D(\mathcal{A}_m), H) \times \mathcal{L}(D(\mathcal{A}_m), L^2(\omega)) \) such that there exists \( \lambda \in \mathbb{C} \) for which the following conditions are satisfied : 
\begin{itemize} 
\item[1.] \( \delta_0 \) is surjective,  
\item[2.] \( \ker (\delta_0) \) is dense in \( H \),  
\item[3.] \( (\lambda I -\mathcal{A}_m)|_{\ker (\delta_0)} \) is bijective.
\end{itemize}
\end{definition} 
 It is evident that \(\delta_{0}\) is surjective. Indeed, fix any \(g \in L^{2}(\omega)\) and \(\lambda \in \mathbb{C}\). We must solve  
 \begin{equation}\label{e2.2}
	\left\lbrace\begin{array}{ll}
	\left(\lambda-\mathcal{A}_m\right)f=0\\
	\delta_0f=u\in L^2(\omega)
	\end{array}\right.
\end{equation}

 Set  
 \[
   f(x,a) \;=\; \exp\!\Bigl(-\!\int_{0}^{a}\mu(s)\,ds\Bigr)\,\hat f(x,a).
 \]
 A straightforward calculation shows that \(\hat f\) satisfies  
 \begin{equation}
	\left\lbrace\begin{array}{ll}
		\partial_a\hat{f}-\Delta\hat{f}+\lambda \hat{f}=0&\hbox{ in }\Omega\times(0,A)\\
		\dfrac{\partial \hat{f}}{\partial \nu}=0&\hbox{ in }\partial\Omega\times(0,A)\\
		\hat{f}(.,0)=\mathds{1}_{\omega}(x)u&\hbox{ in }\Omega
	\end{array}\right.
\end{equation}
 This is the classical linear heat  problem with homogeneous Neumann boundary conditions on \(\partial\Omega\) and initial data \(g\) at \(a=0\). Standard parabolic theory guarantees a unique solution  
 \[
   \hat f \;\in\; H^{1}\bigl(0,A;L^{2}(\Omega)\bigr)\,\cap\,L^{2}\bigl(0,A;H^{2}(\Omega)\bigr).
 \]
 Consequently, \(f\in H^1(0,A;L^2(\Omega)) \cap L^2(0,A;H^2(\Omega))\) and satisfies \eqref{e2.2}. Since for every \(u \in L^{2}(\omega)\) there is a unique such \(f\in D(\mathcal{A}_{m})\), the operator \(\delta_{0}\) is surjective.

 Because \(\mathcal{A}\) generates a \(C_0\)-semigroup on \(H_1\) and \(\delta_{0}\) is surjective, we obtain the following Lemma, which has already been proven in \cite{ref77} :
\begin{lemma}\label{l2.3}
The following assertions are true for each $\lambda\in \rho(\mathcal{A}) :$
\begin{itemize}
\item[(i)]$D(\mathcal{A}_m)=D(\mathcal{A})\oplus ker(\lambda I-\mathcal{A}_m),$
\item[(ii)] $\delta_{0_{\mid ker(\lambda I-\mathcal{A}_m)}}$ is inversible and the operator $\mathbb{D}_{\lambda}:=\left(\delta_{0_{|\ker(\lambda-\mathcal{A}_m)}}\right)^{-1}:L^2(\omega)\longrightarrow \ker (\lambda-\mathcal{A}_m)\subset D(\mathcal{A}_m)$ is bounded and $\delta_0\mathbb{D}_{\lambda}=I.$
\end{itemize}
\end{lemma}
Operator $\mathbb{D}_{\lambda}$ is crucial for reformulating boundary control problem \eqref{e2.1} as a distributed control problem.

We now consider the following Lemma, which allows us to rewrite problem \eqref{e2.1} as a problem with a source term. 
\begin{lemma}\label{l2.4}
    Let $\lambda\in\rho(\mathcal{A})$ 
\begin{itemize}
    \item[1.]  There exists a unique operator \(\mathcal{B} \in \mathcal{L}(L^2(\omega), H_{-1})\) such that
    \begin{align}\label{e2.4}
        \mathcal{A}_m = \mathcal{A}_{-1} + \mathcal{B}\, \delta_0.
    \end{align}
\item[2.]   Moreover, \((\lambda I -\mathcal{A}_{-1})^{-1}\mathcal{B} \in \mathcal{L}(L^2(\omega); D(\mathcal{A}_m))\) and
\begin{align}
    \delta_0\, (\lambda I - \mathcal{A}_{-1})^{-1}\mathcal{B} = I.
\end{align}
\end{itemize}
\end{lemma}
\begin{proof}[of Lemma \ref{l2.4}]
\begin{itemize}
\item[1.]  \(\mathcal{B} \in \mathcal{L}(L^2(\omega), H_{-1})\) satisfies \eqref{e2.4} if and only if  
\begin{align}
    \mathcal{B}\,\delta_0 = \mathcal{A}_m - \mathcal{A}_{-1}.
\end{align}
Since \(\delta_0 \in \mathcal{L}(D(\mathcal{A}_m), L^2(\omega))\) is surjective by condition (1) of Definition \ref{d2.2},  we introduce its right-inverse \(\mathbb{D}_{\lambda}\in \mathcal{L}(L^2(\omega), D(\mathcal{A}_m))\), as constructed in Lemma \ref{l2.3}. 

Suppose that the operator \(\mathcal{B}\) exists. Because \(\delta_0\,\mathbb{D}_{\lambda} = I \in \mathcal{L}(L^2(\omega))\), it follows from \eqref{e2.4} that for every \(u \in L^2(\omega)\)
\[
\mathcal{B}u = \mathcal{B}\,\delta_0(\mathbb{D}_{\lambda}u) = (\mathcal{A}_m - \mathcal{A}_{-1})\,\mathbb{D}_{\lambda}u,
\]
so that
\[
\mathcal{B} = (\mathcal{A}_m - \mathcal{A}_{-1})\,\mathbb{D}_{\lambda}.
\]
This demonstrates the uniqueness of \(\mathcal{B}\).

Conversely, define \(\mathcal{B} = (\mathcal{A}_m - \mathcal{A}_{-1})\,\mathbb{D}_{\lambda}\) and show that \(\mathcal{B} \in \mathcal{L}(L^2(\omega), H_{-1})\). Since \(\mathbb{D}_{\lambda} \in \mathcal{L}(L^2(\omega), D(\mathcal{A}_m))\) and \(\mathcal{A}_m \in \mathcal{L}(D(\mathcal{A}_m), H)\), and because the embedding
$$H_1\subset H \subset H_{-1}$$ is continuous, we have \(\mathcal{A}_m\,\mathbb{D}_{\lambda} \in \mathcal{L}(L^2(\omega), H_{-1})\). Moreover, the continuity of the embedding $D(\mathcal{A}_m) \subset H$ ensures that \(\mathbb{D}_{\lambda} \in \mathcal{L}(L^2(\omega), H)\), so that \(\mathcal{A}_{-1}\,\mathbb{D}_{\lambda} \in \mathcal{L}(L^2(\omega), H_{-1})\). Hence, \(\mathcal{B} \in \mathcal{L}(L^2(\omega), H_{-1})\).

\item[2.] Let \(\lambda \in \rho(\mathcal{A})\). By definition, \((\lambda I - \mathcal{A}_{-1})\) is invertible with inverse \((\lambda I - \mathcal{A}_{-1})^{-1} \in \mathcal{L}(H_{-1}, H)\), and therefore \((\lambda I - \mathcal{A}_{-1})^{-1}\mathcal{B} \in \mathcal{L}(L^2(\omega), H)\). We now show that, in fact, \((\lambda I - \mathcal{A}_{-1})^{-1}\mathcal{B} \in \mathcal{L}(L^2(\omega), D(\mathcal{A}_m))\). Since $\mathcal{B} = (\mathcal{A}_m - \mathcal{A}_{-1})\,\mathbb{D}_{\lambda},
$
we have
\[
(\mathcal{A}_m - \mathcal{A}_{-1})\,\mathbb{D}_{\lambda} = (\lambda I - \mathcal{A}_{-1})\,\mathbb{D}_{\lambda} - (\lambda I - \mathcal{A}_m)\,\mathbb{D}_{\lambda},
\]
with \((\lambda I - \mathcal{A}_m)\,\mathbb{D}_{\lambda} \in \mathcal{L}(L^2(\omega), H)\). Therefore,
\begin{align}\label{e2.7}
(\lambda I - \mathcal{A}_{-1})^{-1}\mathcal{B} = \mathbb{D}_{\lambda} - (\lambda I - \mathcal{A}_{-1})^{-1}(\lambda I - \mathcal{A}_m)\,\mathbb{D}_{\lambda}    
\end{align}
\begin{align*}
    = \mathbb{D}_{\lambda} - (\lambda I - \mathcal{A})^{-1}(\lambda I - \mathcal{A}_m)\,\mathbb{D}_{\lambda}.
\end{align*}
Since \((\lambda I - \mathcal{A})^{-1} \in \mathcal{L}(H, H_1)\), it follows that
\[
(\lambda I - \mathcal{A})^{-1}(\lambda I - \mathcal{A}_m)\,\mathbb{D}_{\lambda} \in \mathcal{L}(L^2(\omega), H_1),
\]
and hence,
\[
(\lambda I - \mathcal{A})^{-1}(\lambda I - \mathcal{A}_m)\,\mathbb{D}_{\lambda} \in \mathcal{L}(L^2(\omega), D(\mathcal{A}_m))\quad\text{since}\; \|\cdot\|_{H_1} = \|\cdot\|_{D(\mathcal{A}_m)}.
\]
This implies, by \eqref{e2.7},
\[
(\lambda I - \mathcal{A}_{-1})^{-1}\mathcal{B} \in \mathcal{L}(L^2(\omega), D(\mathcal{A}_m)).
\]
Moreover, 
\[
(\lambda I - \mathcal{A}_{-1})^{-1}\mathcal{B} - \mathbb{D}_{\lambda} \in \mathcal{L}(L^2(\omega), H_1),
\]
from which it follows that
\[
\delta_0(\lambda I - \mathcal{A}_{-1})^{-1}\mathcal{B} = I.
\]
\end{itemize}
\end{proof} 
Building on Lemma \ref{l2.4}, the transformation \eqref{e2.4} allows us to rewrite system \eqref{e2.1} as follows:

\begin{equation}\label{e2.8}
\left\lbrace 
\begin{array}{ll}
\partial_{t}y(x,a,t)=\mathcal{A}_{-1} y(x,a,t)+\mathcal{B}z(x,t)+\mathcal{B}\mathds{1}_{\omega}(x)v & \text{ in }  Q,\\
\\y(x,a,0)=y_{0}(x,a)&\text{ in } Q_A,\\
\\z(x,t)=C y(x,a,t)
\end{array}
\right.
\end{equation}
where $C=\Lambda_{\mid D(\mathcal{A})}.$ Furthermore, by \cite[Theorem 4.1]{ref76}, system \eqref{e2.8} is equivalent to
\begin{equation}\label{e2.9}
\left\lbrace 
\begin{array}{ll}
\partial_{t}y(x,a,t)=(\mathcal{A}_{-1} +\mathcal{B}C_{\Lambda})y(x,a,t)+\mathcal{B}\mathds{1}_{\omega}(x)v & \text{ in }  Q,\\
\\y(x,a,0)=y_{0}(x,a)&\text{ in } Q_A,
\end{array}
\right.
\end{equation}
where $C_{\Lambda}$ is the Yosida extension of $C$ for $\mathcal{A}_{-1}.$  In this framework, with \(\Lambda\) restricted to \(D(\mathcal{A})\), the mild solution of \eqref{e2.9} can be written as
 \begin{align}
     y(t)
 = \mathcal{T}(t)\,y_0
 + \Phi_t v,
 \quad t\in(0,T),
 \end{align}
 where $\mathcal{T}$ is the semigroup generated by \((\mathcal{A}_{-1} + \mathcal{B}\Lambda)\) and  
 \begin{align}
    \Phi_t v
 = \int_{0}^{t}
   \mathcal{T}_{-1}(t - s)\,\mathcal{B}\,\mathds{1}_{\omega}(x)v(s)
 \,\mathrm{d}s,
 \quad t\ge0. 
 \end{align}
It is clear that \((\mathcal{A}_{-1} + \mathcal{B}\Lambda)\) coincides with \(\mathcal{A}_m\) (see \cite{ref76} for more details). Indeed, define
\[
\mathcal{A}_{\Lambda} x := \bigl(\mathcal{A}_{-1} + \mathcal{B}\Lambda\bigr)x,
\quad
x \in D(\mathcal{A}_{\Lambda})
:= \bigl\{x \in D(\Lambda) \mid (\mathcal{A}_{-1} + \mathcal{B}\Lambda)x \in H\bigr\},\; D(\Lambda)= \bigl\{x \in H \mid \Lambda x \in H\bigr\}.
\]

\begin{itemize}
  \item \(\mathcal{A}_m \subset \mathcal{A}_{\Lambda}\).  
    Let \(x \in D(\mathcal{A}_m)\). Since \(D(\mathcal{A}_m)\subset D(\Lambda)\) and \(\delta_0 x = \Lambda x\),  
    \[
      (\mathcal{A}_{-1} + \mathcal{B}\Lambda)x
      = \mathcal{A}_{-1}x + \mathcal{B}\,\delta_0 x
      = \mathcal{A}_m x \in H.
    \]
    Hence \(x \in D(\mathcal{A}_{\Lambda})\) and \(\mathcal{A}_{\Lambda}x = \mathcal{A}_m x\).

  \item \(\mathcal{A}_{\Lambda} \subset \mathcal{A}_m\).  
    Let \(x \in D(\mathcal{A}_{\Lambda})\). Then
    \[
      (\mathcal{A}_{-1} + \mathcal{B}\Lambda)x \in H
      \;\Longrightarrow\;
      \mathcal{A}_{-1}(x - \mathbb{D}_{\lambda}\Lambda x)
      - \lambda\,\mathbb{D}_{\lambda}\Lambda x \in H,
    \]
    since \(\mathcal{B} = (\lambda - \mathcal{A}_{-1})\,\mathbb{D}_{\lambda}\) for any \(\lambda\in\rho(\mathcal{A})\).  
    In particular, \(\mathcal{A}_{-1}(x - \mathbb{D}_{\lambda}\Lambda x)\in H\), so \(x - \mathbb{D}_{\lambda}\Lambda x\in D(\mathcal{A})\).  
    Since \(D(\mathcal{A})\subset D(\mathcal{A}_m)\) and
    \[
      \delta_0(x - \mathbb{D}_{\lambda}\Lambda x)
      = \delta_0 x - \delta_0(\mathbb{D}_{\lambda}\Lambda x)
      = \Lambda x - \Lambda x = 0,
    \]
    it follows that \(x\in D(\mathcal{A}_m)\) and
    \[
      \mathcal{A}_m x
      = \mathcal{A}_{-1}x + \mathcal{B}\,\delta_0 x
      = (\mathcal{A}_{-1} + \mathcal{B}\Lambda)x
      = \mathcal{A}_{\Lambda}x\in H.
    \]
\end{itemize}

Since both inclusions hold, we conclude \(D(\mathcal{A}_{\Lambda}) = D(\mathcal{A}_m)\) and \(\mathcal{A}_{\Lambda} = \mathcal{A}_m\), as claimed.

We are concerned with the following abstract dynamics on $H$ : 

\begin{equation*}
\left\lbrace 
\begin{array}{ll}
&\quad \text{(a)}\quad y(t) \;=\; \mathcal{T}(t)\,y_0 \;+\; \Phi_tv, \quad y_0 \in H, \\
&\quad \text{(b)}\quad \Phi_tv \;=\; (\lambda I-\mathcal{A})\,\displaystyle\int_0^t \mathcal{T}_{-1}(t-s)\,(\lambda I-\mathcal{A})^{-1} \mathcal{B}\,\mathds{1}_{\omega}(x)v(s)\,\mathrm{d}s, \\
&\quad \text{(c)}\quad \mathcal{B} \;\in\; \mathcal{L}\bigl(L^2(\omega);\,H_{-1}\bigr)
\quad \text{so that} \quad (\lambda I-\mathcal{A})^{-1} \mathcal{B} \;\in\; \mathcal{L}\bigl(L^2(\omega);\,H_1\bigr),\; (\lambda I-\mathcal{A}) \in \mathcal{L}(H_1, H).
\end{array}
\right.
\end{equation*}
For practical purposes, one typically seeks continuous \(H\)-valued functions. It is therefore necessary to characterize a class of control operators \(\mathcal{B}\) that ensure the state of system  \eqref{e2.9} remains in \(H\). The following result provides a characterization of the well-posedness of the boundary input-output system \eqref{eq1.1}.
\begin{lemma}\label{l2.5}
Consider the unique control operator $\mathcal{B}\in \mathcal{L}(L^2(\omega), H_{-1})$ and the generator \(\mathcal{A}_{-1}\) associated with \((D(\mathcal{A}_m), \delta_0)\) generator of a \(C_0\) semigroup \(H\longrightarrow H_{-1}\). For some fixed \(T>0\) and any \(v \in L^2((0,+\infty); L^2(\omega))\), we have
\begin{align}\label{e2.12}
    \Phi_T(v) \in H.
\end{align} 
\end{lemma}
\begin{proof}[of Lemma \ref{l2.5}]
The system \eqref{e2.9} admits a mild solution in \(H_{-1}\) of the form
\begin{align}\label{ee2.13}
    y(\cdot,\cdot,t)=\mathcal{T}(t)y_0(\cdot,\cdot)+\int_0^t \mathcal{T}_{-1}(t-s)\,\mathcal{B}\,\mathds{1}_{\omega}(x)v(s)\,\mathrm{d}s,\quad t\in (0,\infty).
\end{align}
We rewrite the integral form on the right-hand side in the form of equality (b). Thus, for some \(T>0\), we have
\[
\Phi_Tv= (\lambda I-\mathcal{A})\,\int_0^T \mathcal{T}_{-1}(T-s)\,(\lambda I-\mathcal{A})^{-1}\mathcal{B}\,\mathds{1}_{\omega}(x)v(s)\,\mathrm{d}s.
\]
It follows that \(\Phi_T\) is closed. Therefore, by the closed graph theorem, it is bounded in \(H\).

Let's now show that it is bounded independently of time.  Let $t\in (0,T),$
\begin{align}\label{e2.13}
    \Phi_{T}(u\underset{T-t}{\diamondsuit} v)=\mathcal{T}_{t}\Phi_{T-t}u+\Phi_{t}v
\end{align}
 where $u\underset{T}{\diamondsuit}v$\footnote{the function in $L^2((0,+\infty); L^2(\omega))$ defined by \begin{equation*}
(u\underset{T}{\diamondsuit}v)(t)=\left\lbrace 
\begin{array}{ll}
u(t),\quad t\in (0,T),\\
v(t-T),\quad t\geq T.
\end{array}
\right.
\end{equation*}} is the $T-$concatenation of $u$ and $v$ in $L^2((0,\infty); L^2(\omega)).$ By taking \( u=0 \) in the composition property \eqref{e2.13} and choosing \(\|v\| = 1\), we obtain
\[
\Phi_{T}\Bigl(0\underset{T-t}{\diamondsuit} v\Bigr) = \Phi_{t}(v),
\]
which implies
\[
\|\Phi_{t}\| = \Bigl\|\Phi_{T}\bigl(0\underset{T-t}{\diamondsuit} v\bigr)\Bigr\| \leq \|\Phi_{T}\|.
\]
Thus, we conclude that
\[
\|\Phi_{t}\| \leq \|\Phi_{T}\|, \quad \text{for } t \leq T.
\]
In conclusion, \eqref{e2.12} is satisfied and system \eqref{eq1.1} is well-posed.
\end{proof}
Moreover, the solution depends continuously on the data in the following sense :
\begin{definition}\label{d2.6}
  Let \(T>0\), \(y_0\in H\), and \(v\in L^2(\omega\times(0,T))\). Then the system \eqref{eq1.1} admits a unique solution \(y\), and there exists a constant \(C(T)>0\) such that
  \[
    \|y\|_{L^2(0,T;H)}
    \;\le\;
    C(T)\,\bigl(\|y_0\|_H + \|v\|_{L^2(\omega\times(0,T))}\bigr).
  \]
\end{definition}
\begin{remark}
System $\eqref{eq1.1}$ is said to be  positively well-posed if it is well-posed and, moreover, every nonnegative initial state together with any nonnegative control yields a nonnegative state for all $t\in(0,T)$.

Formally, define the Hilbert space

$$
H_+ \;=\;\bigl\{\,y\in H\;\big|\;y\ge0\;\text{a.e. }(x,a)\in\Omega\times(0,A)\bigr\}.
$$

Then, for any control $v\in L^2(\omega\times(0,T))$ with $v\ge0$ a.e., and any initial datum $y_0\in H_+$, the solution $y$ of the form \eqref{ee2.13} remains in $H_+$ for all $t\in (0,T)$.
\end{remark}
\section{Main Results}
 In this section, we present the main results of this paper, namely the null controllability of model \eqref{eq1.1} and the turnpike property. Prior to that, we examine the null-controllability and convergence of a control system that is localized on \(\omega \times [0,a_0]\) in the framework of a population dynamics model. Following the approach of \cite{ref18}, we investigate the null-controllability of the distributed problem \eqref{eq1.5} below while explicitly characterizing the control. Consider the following model :
\begin{equation}\label{eq1.5}
\left\lbrace 
\begin{array}{ll}
\partial_{t}y(x,a,t)+\partial_{a}y(x,a,t)-\triangle y(x,a,t)+\mu(a) y(x,a,t)=\mathds{1}_{\lbrace\omega\times[0,a_0]\rbrace}(x,a)v(x,a,t) &\text{in }\quad Q,\\
 \\\partial_{\nu}y(x,a,t)=0 &\text{on }\quad \Sigma,\\
 \\y(x,0,t)=\displaystyle\int_{0}^{A}\beta(x,a)y(x,a,t)da &\text{in }\quad Q_{T},\\
 \\y(x,a,0)=y_{0} &\text{in }\quad Q_{A},
\end{array}
\right.
\end{equation} 
with \(v\) as the control function. A result of null-controllability is already established in \cite{ref3}  by combining final-state observability estimates with the use of characteristics and the associated semigroup.
\begin{proposition}\label{th1.3}
Consider the assumptions (H1-H2). Furthermore, suppose that the fertility rate is such that
\begin{equation*}
\beta(x,a)=0\qquad\qquad \forall\quad a\in [0, a_b]
\end{equation*}
for $a_b \in [0, A]$ and $a_b>0$. Thus, for every $T>A-a_0$ and  $y_0\in H,$ there exists a control $v\in L^2(\omega\times(0,a_0)\times(0,T) )$ such that $y$ solution to \eqref{eq1.5} satisfies
\begin{equation}\label{e3.2}
y(x,a,T)=0\qquad x\in\Omega,\,a\in [0,A],\quad\text{a.e.}.
\end{equation}
\end{proposition}
The null controllability result established in Proposition \ref{th1.3} is based on the following definition. \begin{definition}\cite[Definition 11.1.1, p. 364]{ref48}\label{d3.2}
    Let $T>0.$ The system \eqref{eq1.5} is null-controllable in time $T$ if
    \begin{align}\label{e3.3}
        \operatorname{Ran} (\mathcal{T}_T) \subset \operatorname{Ran}(\Phi_T).
    \end{align}
\end{definition}
Here, the notation \(\text{Ran}(\mathbb{T}_T)\) and \(\text{Ran}(\Phi_T)\) respectively denote the range of \(\mathbb{T}_T\) (the output map) and \(\Phi_T\) (the input map) where \(\mathbb{T}_T\) and \(\Phi_T\) are defined below.  Proposition \ref{th1.3} is proved by showing that condition \eqref{e3.3} holds for every $T>A-a_0$. Hence we obtain the following explicit control:

\begin{equation}\label{e3.4}
v(x,a,t)=\begin{cases}
0, & \text{if } t-a<0,\quad \forall\, x\in\Omega,\\[1mm]
-\dfrac{\displaystyle\int_{t-a}^{A}\beta(x,z)\pi(z)\,\dfrac{e^{(t-a)\Delta}y_0(x,z-(t-a))}{\pi(z-(t-a))}\,dz}{\displaystyle\int_{0}^{A-(t-a)}\frac{e^{-z\Delta}\,\mathds{1}_{\{\omega\times[0,a_0]\}}(x,z)}{\pi(z)}\,dz}, & \text{if } t-a\geq 0,\quad \forall\, x\in\Omega,
\end{cases}
\end{equation} 

and the corresponding controlled state

\begin{equation}\label{e3.5}
y(x,a,t)=\begin{cases}
\dfrac{\pi(a)}{\pi(a-t)}e^{t\Delta}y_0(x,a-t), & \text{if } t-a<0,\quad \forall\, x\in\Omega,\\[1mm]
\pi(a)e^{a\Delta}\eta(x,t-a)\left[1-\dfrac{\displaystyle\int_{0}^{a}\frac{e^{-s\Delta}\,\mathds{1}_{\{\omega\times[0,a_0]\}}(x,s)}{\pi(s)}\,ds}{\displaystyle\int_{0}^{A-(t-a)}\frac{e^{-z\Delta}\,\mathds{1}_{\{\omega\times[0,a_0]\}}(x,z)}{\pi(z)}\,dz}\right], & \text{if } t-a>0,\quad \forall\, x\in\Omega,
\end{cases}
\end{equation} 

which indeed satisfy condition \eqref{e3.2}.  The proof of Proposition \ref{th1.3} is provided in Appendix~\ref{annexe:A}. 

This strategy allows us to analyze the behavior of control \eqref{e3.4} and state \eqref{e3.5} as $a_0\to 0.$

Before establishing our first main result, that the system \eqref{eq1.1} is null-controllable, we introduce the following auxiliary model:
\begin{equation}\label{eq1.6} 
\left\lbrace 
\begin{array}{ll}
\partial_{t}y(x,a,t)+\partial_{a}y(x,a,t)-\triangle y(x,a,t)+\mu(a)y(x,a,t)=\mathds{1}_{\omega\times\lbrace 0\rbrace}(x,a)v_1(x,a,t) & \text{in}\quad Q\\
 \\\partial_{\nu}y(x,a,t)=0 & \text{on}\quad\Sigma\\
\\y(x,0,t)=\displaystyle\int_{0}^{A}\beta(x,a)y(x,a,t)da  & \text{in}\quad Q_{T}\\
 \\y(x,a,0)=y_{0}  & \text{in}\quad Q_{A}
\end{array}
\right.
\end{equation}
with $v_1$ the control.  It is evident, according to Lemma \ref{le1.1}, that the systems \eqref{eq1.5} and \eqref{eq1.6} are well-posed. 

 It is well known (see, e.g., \cite{ref24}) that the semigroup \(\mathcal{T}\) generated by \(\mathcal{A}_m\) is exponentially stable if its dominant eigenvalue \(\lambda_{1}^{0}\) (i.e., the principal eigenvalue of \(\mathcal{A}_m\)) satisfies \(\lambda_{1}^{0}<0\), where \(\lambda_{1}^{0}\) denotes the unique real solution to the characteristic equation
\[
\tilde{\beta}(\lambda)=\int_0^A \beta(x,a)e^{-\lambda a}\pi(a)\,da = 1,
\]
is strictly negative. This equation, known as Lotka's characteristic equation, defines the set of characteristic roots
\[
\{\lambda \in \mathbb{C} : \tilde{\beta}(\lambda)=1\}.
\]
\begin{remark}\label{re1.2}
Let \(R=\displaystyle\int_0^A \beta(x,a)\pi(a)\,da\) denote the reproduction number. We say that \(\bar{y}\) is a nonnegative steady state of \eqref{eq1.6} if 
 \[\bar{v}\ge0 \quad\text{a.e.}\; X\in\Omega,\; a\in(0,A),
 \]
\begin{equation}\label{s3.3}
\left\lbrace 
\begin{array}{ll}
\partial_{a}\bar{y}(x,a)-\triangle \bar{y}(x,a)+\mu(a) \bar{y}(x,a)=\mathds{1}_{\omega\times \lbrace0\rbrace}(x,a)\bar{v}(x,a) & \text{ in }  Q_1,\\
 \\\partial_{\nu}\bar{y}(x,a)=0 &\text{ on }  \Sigma_1,\\
 \\\bar{y}(x,0)=\displaystyle\int\limits_{0}^{A}\beta(x,a)\bar{y}(x,a)da &\text{ in }  \Omega
\end{array}
\right.
\end{equation}
It is well known that (see, for instance \cite[Theorem 3.1]{ref23}) :
\begin{itemize}
    \item[\(\bullet\)] If \(R<1\), then there exists a unique positive solution of system \eqref{s3.3}.
    \item[\(\bullet\)] If \(R=1\) and \(\bar{v}\equiv0\), then there exists an infinite family of solutions of \eqref{s3.3} given by \(\bar{y}(x,a)=\alpha\pi(a)\), with \(\alpha\in [0,\infty)\).
    \item[\(\bullet\)] If \(R>1\), then no positive solution exists for \eqref{s3.3}.
\end{itemize}
\end{remark}
For that purpose, we may assume without loss of generality that the so-called reproductive number satisfies \( R < 1 \) (see \cite[Remark 2.6 and Theorem 2.2]{ref24}),  which ensures that the semigroup \(\mathcal{T}\) is exponentially stable and hence bounded.

The work in \cite{ref18} highlighted the dependence of the control on the age parameter $a_0$. It paves the way to extend, to the diffusive case, the singular-perturbation result characterizing the behavior of the control \eqref{e3.4} and of the state \eqref{e3.5} as $a_0 \to 0$ (direct birth control). We are currently pursuing this analysis in order to track the control dynamics as $a_0$ tends to zero.  We set $a_0 = \varepsilon$ and, for system \eqref{eq1.5}, define the family of null controls $\bigl(v^{\varepsilon}\bigr)_{\varepsilon\in[0,A]}$ given in \eqref{e3.4}, along with their corresponding controlled states $\bigl(y^{\varepsilon}\bigr)_{\varepsilon\in[0,A]}$ defined in \eqref{e3.5}.
\begin{theorem}\label{th1.4}
Consider the assumptions of Proposition \ref{th1.3}, the distributed control family  $(v^{\epsilon})_{\epsilon\in [0,A]}\subset L^2 (\omega\times[0,A]\times[0,T])$ of \eqref{eq1.5}   where $y^{\epsilon}$ is its controlled state. Then,  
\begin{itemize}
\item[1-] $\lim_{\epsilon\longrightarrow 0}\mathds{1}_{\lbrace\omega\times[0,\epsilon]\rbrace} v^{\epsilon}=\mathds{1}_{\omega\times\lbrace 0\rbrace}v_1$ weakly   in $L^2 (\omega\times(0,T))$ with $v_1$ the null-control of \eqref{eq1.6}.
\item[2-] $\lim_{\epsilon\longrightarrow 0} y^{\epsilon}=y$  strongly in $H,\quad \forall \;t\;\in (0,T)$  where $y$ is the  controlled state  of \eqref{eq1.6}.
\item[3-] For every $y_0\in H$ and every time horizon $T> A$, there exists $v_1 \in L^2(\omega\times[0,T])$ such that the solution $y$ to \eqref{eq1.6} and \eqref{eq1.1} satisfy 
\begin{align}\label{e3.8}
     y(x,a,T)=0\text{ a.e. }x\in\Omega\quad a\in(0,A).
\end{align} 
Moreover, there exists a constants $K(A)>0$ such that the control function  verify the following inequality :
\begin{align}
    \|v_1\|_{L^2(\omega\times(0,T))} \le K(A)\,\|y_{0}\|.
\end{align}
\end{itemize}
\end{theorem}
In fact, by rewriting systems \eqref{eq1.1} and \eqref{eq1.6} respectively from left to right in the form 
\noindent
\begin{minipage}{0.45\linewidth}
\begin{equation*}
\left\lbrace 
\begin{array}{ll}
\partial_{t}y(x,a,t)=\mathcal{A}_my(x,a,t) & \text{ in }  Q,\\
 \\ 
\delta_0(a)y(x,a,t)=\mathds{1}_{\omega}(x)v(x,t) &\text{ in }  Q_T,\\
\\y(x,a,0)=y_0 &\text{ in }  Q_A.
\end{array}
\right.
\end{equation*}
\end{minipage}
\hfill
\begin{minipage}{0.45\linewidth}
\begin{equation*}
\left\lbrace 
\begin{array}{ll}
\partial_{t}y(x,a,t)=\mathcal{A}_{m}y(x,a,t)+Bv(x,a,t) & \text{ in }  Q,\\
 \\ 
y(x,a,0)=y_0 &\text{ in }  Q_A.
\end{array}
\right.
\end{equation*}
\end{minipage}\\
\\with $B=\mathds{1}_{\lbrace\omega\times0\rbrace}(x,a),$ the null-controllability for system in the left can be readily derived from that for system in the right, and vice versa.  Thus, to introduce optimal control, we rewrite system \eqref{eq1.1} in the form of the system on the right, following the formulation in Section \ref{s2} such that 
\begin{align}
    \mathcal{B}\mathds{1}_{\omega}(x)v=Bv,\quad (v\in L^2((0,T); H)).
\end{align} 
Since the control operator $\mathcal{B}$ is admissible for $\mathcal{T}$, $B$ is admissible as well (refer to the example presented in \cite[Lemma 2.7]{ref79}).
 The following concerns the  optimal control problem, in order to reach an ideal population at a lower cost quantifying the proximity  of  optimal dynamic solutions in the steady state :  Turnpike Property.  
 We consider control systems in which the initial state resides in \( H \) and the control input \( v \) belongs to \( L^2(\omega \times [0, T]) \). Given \( T >A> 0 \), our focus is on system \(\eqref{eq1.1}\) as well as on its steady-state formulation, in which the mapping \( v \mapsto y \)  is linear (and therefore affine).  For any \( y_0 \in H \), we examine the behavior of global minimizers \(v_T \in L^2(\omega \times [0, T])\) for non-negative functionals of the form
\begin{align}\label{eq1.2}
   \min\limits_{v\in L^2(\omega\times[0,T])} J_1(v)=\dfrac{N}{2}\int\limits_{0}^{T}\Vert y(t)-y_d\Vert^2 dt+\dfrac{1}{2}\int\limits_{0}^{T}\Vert v(t)\Vert^2_{L^2(\omega)}dt+\dfrac{1}{2}\langle y(T),p(T)\rangle
\end{align}
 $N>0$ is given, considering suitably chosen time horizons \(T\) and the corresponding solutions $y(t)=y(x,a,t)$ to \eqref{eq1.6}. The last term represents an additional final pay-off (representing a final cost added to minimize the population density at the final time) associated with the adjoint state \(p \), which we will define later.  \( p(T) \) appears in the transversality condition of the Pontryagin maximum principle and reflects the sensitivity of the terminal cost with respect to \( y(T) \). This terminal term acts as a penalty, ensuring that, at the end of the interval, the trajectory remains close to the optimal solution (turnpike).  $y_d\in H$ is the prescribed running target, chosen as a steady-state solution of the following system in the absence of control (\( v \equiv 0 \))
\begin{equation}\label{eq1.3}	 
-\mathcal{A}_my(x,a)=Bv(x,a)  \qquad\;\text{in}\;Q_1.
\end{equation}
With the target incorporated into the definition of \(J_1\), which regulates the state over the entire interval \([0,T]\), the optimal control-state pair we designate, \((y_T, v_T)\), should remain close to the optimal steady control-state pair \((\bar{y}, \bar{v})\) the solution of the problem
\begin{equation}\label{eq1.4}
 \min\limits_{v\in L^2(\omega)}J_2(v)=\dfrac{N}{2}\Vert y-y_d\Vert^2+\dfrac{1}{2}\Vert v\Vert^2_{L^2(\omega)}.
 \end{equation}
Over long time horizons, this phenomenon is known as the turnpike property.
\begin{theorem}\label{th1.5}[Exponential Turnpike Property]
Suppose $y_0\in H,\quad y_d\in H$ are fixed and  let us the system \eqref{eq1.1}   null-controllable in some time $T> A$ in the sense of Theorem \ref{th1.4}. There exists a couple of positive constants  $(C,\nu),$  independent of $y_0$ and $y_d$, such that for any positive $T$ large enough, the unique solution  $(y_T,v_T)$  to \eqref{eq1.2} subject to \eqref{eq1.1}   satisfies 
\begin{align}\label{ee3.13}
\Vert y_T(t)-\bar{y}\Vert+\Vert v_T(t)-\bar{v}\Vert_{L^2(\omega)}\leq C\left(\Vert y_0 -\bar{y}\Vert e^{-\nu t}+\Vert p_T(T)-\bar{p}\Vert e^{-\nu (T-t)}\right)
\end{align}
for almost every  $t\in [0,T],$ where $(\Bar{y},\Bar{v})$ denotes the unique solution to  \eqref{eq1.4} subject to \eqref{eq1.3} and $(p_T,\Bar{p})$ are optimal dynamic and steady adjoint state.
\end{theorem} 
\begin{remark}
 Inequality \eqref{ee3.13} means that, for sufficiently large $T$, the trajectories $y_T(t)$ and $v_T(t)$ remain exponentially close to $\bar y$ and $\bar v$ over most of the interval $[0,T]$, except in thin boundary‐layer regions near $t=0$ and $t=T$.
\end{remark}

Now, let us mention some related works from the literature.

On the one hand, previous work  \cite{ref3} has demonstrated that system \eqref{eq1.5} can be controlled to any quasi-steady state via distributed control applied over a small age interval near zero, while ensuring the positivity of solutions.  The null-controllability of system \eqref{eq1.5} has been investigated in various contexts : in \cite{ref19} for the spatially independent case, subsequently improved in \cite{ref18} by constructing a control as a function of the initial condition with the use of characteristics, in \cite{ref20} by combining final-state observability estimates with the use of characteristics and the associated semigroup, in \cite{ref25} where the birth and mortality rates are modeled as nonlinear functions of the population size under the assumption that young individuals cannot reproduce before a certain age \(a_0>0\). Moreover, the exact controllability of system \eqref{eq1.5} under a positivity constraint was established in \cite{ref3} when the fertility function is independent of the spatial variable, adhering to the same reproductive conditions ($\mathcal{N}ull_{\beta}$) as in \cite{ref18}. A similar model, in which the control acts on a subspace, is considered in \cite{ref14, ref21, ref23, ref26}, while the null-controllability of the model with control acting over the entire age range was proven in \cite{ref24} by employing the Lebeau-Robbiano strategy originally developed for the heat equation. \cite{ref71} have analysed an approximate controllability result for a nonlinear
population dynamics model from the birth control, \cite{ref27} improve the result obtained in \cite{ref71}, by analyzing the birth lower frequency null controllability and the birth long time null controllability with a control domain  only measurable.

On the other hand, the optimal control of age-dependent population dynamics with diffusion is a problem of considerable importance that has been addressed by numerous researchers. Optimal control in this context has been examined in, for example, \cite{ref42, ref49, ref57}. The turnpike property, a powerful analytical tool for dynamic models emerged in economics through the seminal contributions of Paul Samuelson, Robert Solow, and Robert Dorfman in the \textsc{\romannumeral 20}\textsuperscript{th} century \cite{ref11}, which also gave rise to the term “turnpike.” Its application to discrete optimal control problems was first demonstrated in $1963$ by McKenzie \cite{ref39}. Several variants of the turnpike property have been proposed, with some being more stringent than others (see, for instance, \cite{ref54, ref56}). Exponential turnpike properties have been established in \cite{ref53, ref54, ref55} for the optimal triple derived via Pontryagin’s maximum principle, ensuring that the extremal solution (comprising the state, adjoint, and control) remains exponentially close to an optimal solution of the corresponding static control problem except at the beginning and end of the time interval when \(T\) is sufficiently large. Furthermore, researchers such as \cite{ref5, ref32, ref33, ref36, ref38, ref44, ref52} have emphasized the energy requirements necessary to remain on the turnpike as a function of the initial data.

In this paper, we improve the result obtained in \cite{ref27} by considering an initial condition in $H$ and by applying the control on a non-empty open subset \(\omega \subset \Omega\). Building upon the approach developed in \cite{ref18}, we demonstrate that there exists a control that is sufficient to drive the entire population to zero without imposing any additional restrictions on the initial condition, in contrast to the strategy employed in \cite{ref27}. Indeed, since model \eqref{eq1.1} is a boundary control system, we proposed to study the null controllability of model \eqref{eq1.5} (a result already established in \cite{ref3}) with the sole purpose of analyzing the convergence, or more generally the behavior, of the controlled state as well as the null control (i.e., the control that drives the state to zero) when the upper bound \(a_0\) tends to zero. Consequently, the resulting control brings both system \eqref{eq1.1} and system \eqref{eq1.6} to rest for all times \(T > A\).

The main contributions of our paper are as follows:

\begin{enumerate}
\item[•]  We have shown that the boundary control model \(\eqref{eq1.1}\) can be reformulated as a control model with a source term, and that the control operator in question is admissible.

\item[•]  We adopt an alternative analytical approach to examine model \(\eqref{eq1.5}\) from \cite{ref3} by incorporating spatial heterogeneity in the fertility rate. In doing so, we have elucidated the dependence of both the control and the state on \(a_0\) for every initial condition. While alternative methodologies have been explored previously, our approach synthesizes techniques from \cite{ref18} (pertaining to the null-controllability result) with the foundational model \eqref{eq1.5}.
 
\item[•] We analyzed the behavior of the controlled state and the null control of system \eqref{eq1.5} as \(a_0\) tends to zero. A similar case was previously studied in \cite{ref18} for a homogeneous population. We extend this analysis by incorporating diffusion and show, using Lebesgue’s dominated convergence theorem, that the null-control converges to a limiting control that drives systems \eqref{eq1.6} and \eqref{eq1.1} to zero. In other words, the null-control of the  system \eqref{eq1.6} is equivalent to that of the boundary system \eqref{eq1.1}.

\item[•] For the first time in the context of population dynamics, we highlight the turnpike property within birth control models. Although the analytical tools employed are well-established, this represents the inaugural application of the turnpike property to population dynamics, specifically in the domain of birth control for a sufficiently long time \(T\), particularly satisfying \(T > A\). In addressing the issue of stabilizability, we employ Phillips' theorem along with our null-controllability results to establish exponential decay towards the equilibrium point. In \cite{ref37}, stability and detectability were the main tools used to study the turnpike property. In our case, in addition to stability, we have system observability, a notion stronger than detectability.
\item[•]  We have established a weak version of the exponential turnpike property, namely the integral turnpike property for \(T \leq A\), by leveraging the system's stability and observability.
\end{enumerate}

The remainder of the paper is organized as follows:
\begin{itemize}
    \item[•] In \textbf{Section \ref{s4}}, we state our null‐controllability results (Theorems \ref{th1.4}). In this section, we show that, using the same control $v$ defined in \eqref{eq2.81}, the solutions of systems \eqref{eq1.1} and \eqref{eq1.6} are driven to rest.
    \item[•] In {\bf Section \ref{s5}}, we undertake a qualitative analysis of the dynamic systems. We establish the turnpike property (Theorem \ref{th1.5}) by comparing the behavior of the optimal solution to the dynamic control problem with that of the equilibrium solution to the stationary control problem.
     \item[•] In {\bf Section \ref{s6}}, we examine a weakened notion of the exponential turnpike property. Since the system is not controllable for \(T \le A\) (see, Proposition \ref{pr2.2}), we establish the integral turnpike property (which directly implies the measure turnpike property) by relying on the system’s stability and observability. We also highlight the case of an exponential turnpike property when the reproduction number \(R < 1\).
    \item[•] In {\bf Section \ref{s7}}, we perform numerical simulations of model \eqref{eq1.1} by projecting the state onto a basis generated by the eigenfunctions of \(L^2(\Omega)\), in order to characterize the turnpike phenomenon. 
    \item[•] \textbf{Section \ref{s8}} is devoted to a discussion and to exploring new perspectives. In light of the theoretical and numerical results obtained, we propose directions for future work.
\end{itemize}
\section{Proof of Null-controllability by birth control}\label{s4}
In order to establish Proposition \ref{th1.3}, we use the characteristic lines to demonstrate that \(\mathrm{Ran}(\mathbb{T}_T) \subset \mathrm{Ran}(\Phi_T)\) at the appropriate time \(T\), where \(\mathbb{T}_T\) and \(\Phi_T\) are defined below. Theorem \ref{th1.4} is based on a similar strategy and relies on the analysis of the behavior of the controlled state of system \eqref{eq1.5}, as well as its control, as \(a_0\) approaches zero. 

We investigate the asymptotic behavior of both the control and the state as \(a_0 = \epsilon \rightarrow 0\). According to Theorem \ref{th1.3}, for every \(y_0 \in H\) there exists a control \(v^{\epsilon} \in L^2(0,T;H)\) such that the solution of \eqref{eq1.5} satisfies \eqref{e3.2}. Choosing \(T = A\) (so that \(T > A-a_0\)), the corresponding null-control \(v^{\epsilon}\) is given by \eqref{e3.4}, and its associated controlled state \(y^{\epsilon}\) is expressed by \eqref{e3.5}.
Before proceeding, we introduce the following definition.
\begin{definition}\cite{ref3}
Let $w\in L^2(Q_1)$ be a distributed control such that  $\bar{y}\in  L^2(Q_1)$ is a positive function solution of the system
\begin{equation}\label{eq2.1} 
\left\lbrace 
\begin{array}{ll}
 \partial_{a}\bar{y}(x,a)-\triangle \bar{y}(x,a)+\mu(a) \bar{y}(x,a)=\mathds{1}_{\lbrace\omega\times[0,a_0]\rbrace}(x,a) w(x,a) &\text{in }\quad Q_1,\\
 \\\partial_{\nu}\bar{y}(x,a)=0 &\text{in }\quad\Sigma_1,\\
 \\\bar{y}(x,0)=\displaystyle\int_{0}^{A}\beta(x,a)y(x,a)da &\text{in }\quad\Omega,\\
\end{array}
\right.
\end{equation}
 is said to be a positive steady state of \eqref{eq1.5} when $w\geq 0$ a.e. on $x\in \Omega,\; a\in (0,A).$
\end{definition}
The so-called mild solution of \eqref{eq1.5}, derived via Duhamel's formula, is given by
\begin{equation}\label{eq2.2}
 y(.,.,t)=\mathcal{T}_t y_0 +\Phi_t v\qquad (x\in \Omega,\, t\geq 0,\quad v\in L^2([0,\infty);H)),
\end{equation}
where
\begin{equation}\label{eq2.3}
\Phi_tv =\int_0^t \mathcal{T}_{t-\sigma}\mathds{1}_{\lbrace\omega\times[0,a_0]\rbrace}(x,a)v(\sigma)d\sigma\quad v\in L^2([0,\infty);H),
\end{equation} 
 and 
\begin{equation}\label{eq2.4}
(\mathcal{T}_t y_0)(a)=\left\lbrace
\begin{array}{ll}
\frac{\pi(a)}{\pi(a-t)}e^{t\triangle}y_0(x,a-t)\quad t\leq a, \quad \forall x\in \Omega,\\
\\\pi(a)e^{a\triangle}b(x,t-a)\quad t> a, \quad \forall x\in \Omega.
\end{array}
\right.
\end{equation}   
Moreover,
$$b(x,t)=(\mathcal{T}_t y_0)(0)=\int_0^A\beta (x,a) (\mathcal{T}_t y_0)(a)da \qquad \forall y_0\in H.$$
By employing the method of characteristics, we obtain
\begin{equation}\label{eq2.5}
y(x,a,t)=\left\lbrace
\begin{array}{ll}
\dfrac{\pi(a)}{\pi(a-t)}e^{t\triangle}y_{0}(x,a-t)+\displaystyle\int_{a-t}^{a}\frac{\pi(a)}{\pi(s)}e^{(a-s)\triangle}\mathds{1}_{\lbrace\omega\times[0,a_0]\rbrace}(x,s)v(x,s,s-a+t)ds\quad t\leq a,\\
\\\pi(a)e^{a\triangle}b(x,t-a)+\displaystyle\int_{0}^{a}\frac{\pi(a)}{\pi(s)}e^{(a-s)\triangle}\mathds{1}_{\lbrace\omega\times[0,a_0]\rbrace}(x,s) v(x,s,s-a+t)ds\quad t> a,
\end{array}
\right.
\end{equation}
where
\begin{align*}
    b(x,t)=\displaystyle\int_{0}^{A}\beta(x,a)  y(x,a,t)da.
\end{align*}
To reformulate the control strategy, we introduce a new control function of the form $v(x,a,t)=u(x,t-a)$ defined over $\Omega\times [-A,T].$ Consequently, the system takes the form
\begin{equation}\label{eq2.6}
y(x,a,t)=\left\lbrace
\begin{array}{ll}
\frac{\pi(a)}{\pi(a-t)}e^{t\triangle}y_{0}(x,a-t)+u(x,t-a)\displaystyle\int_{a-t}^{a}\frac{\pi(a)}{\pi(s)}e^{(a-s)\triangle}\mathds{1}_{\lbrace\omega\times[0,a_0]\rbrace}(x,s)ds\quad t\leq a,\quad \forall x\in \Omega\\
\\\pi(a)e^{a\triangle}b(x,t-a)+u(x,t-a)\displaystyle\int_{0}^{a}\frac{\pi(a)}{\pi(s)}e^{(a-s)\triangle}\mathds{1}_{\lbrace\omega\times[0,a_0]\rbrace}(x,s)ds\quad t> a, \quad \forall x\in \Omega.
\end{array}
\right.
\end{equation}
Using characteristic methods, the state of system \eqref{eq1.6} is then given by
\begin{equation}\label{ee4.7}
y(x,a,t)=\left\lbrace
\begin{array}{ll}
\dfrac{\pi(a)}{\pi(a-t)}e^{t\triangle}y_{0}(x,a-t)+\displaystyle\int_{a-t}^{a}\frac{\pi(a)}{\pi(s)}e^{(a-s)\triangle}\mathds{1}_{\omega\times \lbrace 0\rbrace}(x,s)v_1(x,s,s-a+t)ds\quad t\leq a,\\
\\\pi(a)e^{a\triangle}b(x,t-a)+\displaystyle\int_{0}^{a}\frac{\pi(a)}{\pi(s)}e^{(a-s)\triangle}\mathds{1}_{\omega\times \lbrace 0\rbrace}(x,s) v_1(x,s,s-a+t)ds\quad t> a.
\end{array}
\right.
\end{equation}
This solution, which naturally leaves a trace at $t-a$, motivates us to investigate a shaping control strategy $u(x,t-a)$  defined over 
 $\Omega\times [-A,T].$ 
\begin{proof}[of Theorem \ref{th1.4}] 
The proof is carried out in three steps.

{\bf Step 1}: Define the functions \( f, k^{\epsilon} \) from \eqref{e3.4} as follows:
\begin{align}\label{eq2.57}
f(x,a,t)=-\int_{t-a}^A \beta(x,z)\pi(z) \dfrac{e^{(t-a)\triangle}y_{0}(x,z-(t-a))}{\pi(z-(t-a))}dz, \quad k^{\epsilon}(x,a,t)=\int_{0}^{A-(t-a)}\frac{e^{-z\triangle}\mathds{1}_{\lbrace\omega\times[0,\epsilon]\rbrace}(x,z)}{\pi(z)}dz.
\end{align} 
Thus, we obtain:
\begin{align*}
v^{\epsilon}(x,a,t)=\dfrac{f(x,a,t)}{k^{\epsilon}(x,a,t)} \qquad \text{for} \quad t-a<0, \quad \forall x\in \Omega.
\end{align*}
Now, let us choose \( \varphi \in L^2 (0,A; D(\mathcal{A}^{\ast}_m)) \) and define
\begin{align}
g^{\epsilon}(x,t) = \int_0^A \mathds{1}_{[0,\epsilon]}(a) v^{\epsilon}(x,a,t)\varphi(x,a,t)da.
\end{align}
As \( \epsilon \) approaches zero, we obtain:
\begin{align}
    \lim_{\epsilon\to 0} g^{\epsilon}(x,t)=\lim_{\epsilon\to 0}\int_0^\epsilon \dfrac{f(x,a,t)}{k^{\epsilon}(x,a,t)}\varphi(x,a,t)da.
\end{align}
For sufficiently small \( \epsilon \) such that \( \epsilon<A-(t-a) \) for all \( a\in [0,\epsilon] \), we have:
\begin{align}\label{eq2.60}
k^{\epsilon}(x,a,t)=\int_{0}^{A-(t-a)}\frac{e^{-z\triangle}\mathds{1}_{\omega\times[0,\epsilon]}(x,z)}{\pi(z)}dz =\int_{0}^{\epsilon}\frac{e^{-z\triangle}\mathds{1}_{\omega}(x)}{\pi(z)}dz.
\end{align}
Consequently,
\begin{align*}
g^{\epsilon}(x,t)= \frac{1}{\displaystyle\int_{0}^{\epsilon}\frac{e^{-z\triangle}\mathds{1}_{\omega}(x)}{\pi(z)}dz} \int_0^\epsilon f(x,a,t)\varphi(x,a,t)da.
\end{align*}
Taking the limit as \( \epsilon \to 0 \), we get:
\begin{align*}
    \lim_{\epsilon\to 0} g^{\epsilon}(x,t)=\dfrac{f(x,0,t)\varphi(x,0,t)}{\mathds{1}_{\omega}(x)}\quad\text{for all}\quad t\in [0,A],\;  x\in \omega.
\end{align*}
Define \( v_1 (x,t)=\dfrac{f(x,0,t)}{\mathds{1}_{\omega}(x)} \) and evaluate at \( t=A \):
\begin{align*}
k^{\epsilon}(x,a,A) =\int_{0}^{a}\frac{e^{-z\triangle}\mathds{1}_{\omega}(x)}{\pi(z)}dz\quad \text{for } a\in [0,\epsilon].
\end{align*}
Thus,
\begin{align*}
g^{\epsilon}(x,A)=0,
\end{align*} 
as \( \epsilon \) tends to zero.
Hence, we conclude:
\begin{align}
g^{\epsilon}(x,t) \xrightarrow[\epsilon\to 0]{} v_1(x,t)\varphi(x,0,t) \quad\text{a.e.} \quad t\in [0,A], \quad x\in \omega.
\end{align}
Next, we show that \( g^{\epsilon}(x,t) \) is dominated by an integrable function in \( L^2(\omega\times(0,A)) \). For sufficiently small \( \epsilon \) such that \( \epsilon < A - t \) for all \( a \in [0, \epsilon] \) and \( \epsilon < A - (t-a) \), we have:
\begin{align*}
k^{\epsilon}(x,a,t)=\int_{0}^{\epsilon}\frac{e^{-z\triangle}\mathds{1}_{\omega}(x)}{\pi(z)}dz.
\end{align*}
By equation \eqref{eq2.60}, for \( t \in [0,A-\epsilon] \). Otherwise, if \( t \in [A-\epsilon,A] \), we also have \( A-(t-a) < \epsilon \), and we construct \( k^{\epsilon} \) over \( 0 < a < \epsilon+t-A \) and then over \( \epsilon+t-A < a < \epsilon \).
For \( \epsilon+t-A < a < \epsilon \), we always obtain:
\begin{align*}
k^{\epsilon}(x,a,t)=\int_{0}^{\epsilon}\frac{e^{-z\triangle}\mathds{1}_{\omega}(x)}{\pi(z)}dz.
\end{align*}
For \( 0<a<\epsilon+t-A \), since \( A-t+a<\epsilon \), we get:
\begin{align*}
k^{\epsilon}(x,a,t)=\int_{0}^{A-(t-a)}\frac{e^{-z\triangle}\mathds{1}_{\omega}(x)}{\pi(z)}dz.
\end{align*}
In summary, for \( t\in [A-\epsilon,A] \) we have
\begin{equation}
k^{\epsilon}(x,a,t)=\left\lbrace
\begin{array}{ll}
\displaystyle\int_{0}^{A-(t-a)}\frac{e^{-z\triangle}\mathds{1}_{\omega}(x)}{\pi(z)}dz\quad a\in [0,\epsilon+t-A],\\
\\\displaystyle\int_{0}^{\epsilon}\frac{e^{-z\triangle}\mathds{1}_{\omega}(x)}{\pi(z)}dz \quad a\in [\epsilon+t-A,\epsilon].
\end{array}
\right.
\end{equation}
Thus, for \( t\in [0,A] \) the function \( g^{\epsilon}(x,t) \) can be written as
\begin{equation}\label{eq2.63}
g^{\epsilon}(x,t)=
\begin{cases}
\dfrac{1}{\displaystyle\int_{0}^{\epsilon}\frac{e^{-z\triangle}\mathds{1}_{\omega}(x)}{\pi(z)}dz}\,\displaystyle\int_0^\epsilon f(x,a,t)\varphi(x,a,t)da, & \text{if } t\in [0,A-\epsilon],\quad x\in \omega, \\[2ex]
\displaystyle\int_0^{\epsilon+t-A} \frac{f(x,a,t)}{\displaystyle\int_{0}^{A-(t-a)}\frac{e^{-z\triangle}\mathds{1}_{\omega}(x)}{\pi(z)}dz}\varphi(x,a,t)da \\[1.5ex]
\quad\quad+\; \dfrac{1}{\displaystyle\int_{0}^{\epsilon}\frac{e^{-z\triangle}\mathds{1}_{\omega}(x)}{\pi(z)}dz}\,\displaystyle\int_{\epsilon+t-A}^\epsilon f(x,a,t)\varphi(x,a,t)da, & \text{if } t\in [A-\epsilon,A],\quad x\in \omega.
\end{cases}
\end{equation}

On the one hand, when \( t\in [0,A-\epsilon] \) (assuming \(\epsilon\in [0,A/2]\)), we have
\[
\left|g^{\epsilon}(x,t)\right| \leq \Vert f(x)\Vert_{L^{\infty}([0,A]^2)}\,\Vert\varphi(x,\cdot,t)\Vert_{L^{\infty}([0,A/2])},\quad x\in \omega,\quad \text{for a.e. } t\in (0,A).
\]
On the other hand, for \( t\in [A-\epsilon,A] \) one obtains
\[
\left|g^{\epsilon}(x,t)\right| \leq \int_0^{\epsilon} \frac{\bigl| f(x,a,t)\varphi(x,a,t)\bigr|}{A-(t-a)}da + \int_{0}^\epsilon \frac{\bigl| f(x,a,t)\varphi(x,a,t)\bigr|}{\epsilon}da,\quad x\in \omega,\quad \text{for a.e. } t\in (0,A).
\]

Recalling equation \eqref{eq2.57} and performing a change of variables (by setting \( u=z-(t-a) \)), we write
\[
f(x,a,t)=-\int_{0}^{A-(t-a)} \beta\bigl(x,z+(t-a)\bigr)\frac{\pi(z+(t-a))}{\pi(z)}e^{(t-a)\triangle}y_{0}(x,z)\,dz.
\]
Since 
\[
\left\Vert \frac{\pi(z+(t-a))}{\pi(z)} \right\Vert \leq 1,
\]
\(\beta\in L^{\infty}(\Omega \times [0,A])\) and, by \cite[Lemma 4.1]{ref3}, we deduce that
\[
\left| f(x,a,t)\right| \leq \Vert \beta\Vert_{L^{\infty}(\Omega\times [0,A])}\,\int_{0}^{A-(t-a)} \Vert y_{0}(\cdot,z)\Vert_{L^{\infty}(\Omega)}dz.
\]
Applying the Cauchy–Schwarz inequality yields
\[
\left| f(x,a,t)\right| \leq \Vert \beta\Vert_{L^{\infty}(\Omega\times [0,A])}\,\sup_{x\in\Omega}\left(\int_{0}^{A-(t-a)}y_{0}^2(x,z)\,dz\right)^{1/2}\,\left(\int_{0}^{A-(t-a)}dz\right)^{1/2},
\]
so that
\[
\left| f(x,a,t)\right| \leq \Vert \beta\Vert_{L^{\infty}(\Omega\times [0,A])}\,\Vert y_{0}\Vert_{L^2(\Omega\times[0,A])}\,\sqrt{A-(t-a)}.
\]

Returning to equation \eqref{eq2.63}, we obtain
\[
\left| g^{\epsilon}(x,t)\right| \leq \Vert \beta\Vert_{L^{\infty}(\Omega\times [0,A])}\,\Vert y_{0}\Vert\,\Vert\varphi(x,\cdot,t)\Vert_{L^{\infty}([0,A/2])}\,\int_0^{A/2}\frac{\sqrt{A-(t-a)}}{A-(t-a)}da
\]
\[
\quad + \; \Vert f(x)\Vert_{L^{\infty}([0,A]^2)}\,\Vert\varphi(x,\cdot,t)\Vert_{L^{\infty}([0,A/2])}\,\int_{0}^{A/2}\frac{2}{A}da.
\]
we deduce that
\[
\left| g^{\epsilon}(x,t)\right| \leq \Vert \beta\Vert_{L^{\infty}(\Omega\times [0,A])}\,\Vert y_{0}\Vert\,\Vert\varphi(x,\cdot,t)\Vert_{L^{\infty}([0,A/2])}\,\int_0^{A/2}\frac{1}{\sqrt{A-(t-a)}}da
\]
\[
\quad + \; 2\,\Vert f(x)\Vert_{L^{\infty}([0,A]^2)}\,\Vert\varphi(x,\cdot,t)\Vert_{L^{\infty}([0,A/2])}.
\]
For \( t\in [A-\epsilon,A] \) we have
\[
\int_0^{A/2}\frac{1}{\sqrt{A-(t-a)}}da=2\Bigl(\sqrt{\frac{3A}{2}-t}-\sqrt{A-t}\Bigr)\leq 2\sqrt{\frac{3A}{2}-t}\leq 2\sqrt{\frac{3A}{2}}.
\]
Hence, we obtain
\[
\left| g^{\epsilon}(x,t)\right| \leq 2\sqrt{\frac{3A}{2}}\,\Vert \beta\Vert_{L^{\infty}(\Omega\times [0,A])}\,\Vert y_{0}\Vert\,\Vert\varphi(x,\cdot,t)\Vert_{L^{\infty}([0,A/2])}
\]
\[
\quad + \; 2\,\Vert f(x)\Vert_{L^{\infty}([0,A]^2)}\,\Vert\varphi(x,\cdot,t)\Vert_{L^{\infty}([0,A/2])},\quad x\in \omega,\quad \text{for a.e. } t\in (0,A).
\]
Define the function \( h \) by
\[
h(x,t)=2\sqrt{\frac{3A}{2}}\,\Vert \beta\Vert_{L^{\infty}(\Omega\times [0,A])}\,\Vert y_{0}\Vert\,\Vert\varphi(x,\cdot,t)\Vert_{L^{\infty}([0,A/2])}
+2\,\Vert f(x)\Vert_{L^{\infty}([0,A]^2)}\,\Vert\varphi(x,\cdot,t)\Vert_{L^{\infty}([0,A/2])}.
\]
Since \( h(x,t)\in L^2(\omega\times[0,A]) \), it follows that
\[
\left| g^{\epsilon}(x,t)\right| \leq h(x,t),\quad x\in \omega,\; t\in [0,A].
\]
Thus, by Lebesgue’s dominated convergence theorem, we conclude that
\[
g^{\epsilon}(\cdot,\cdot) \longrightarrow v_1(\cdot,\cdot)\varphi(\cdot,0,\cdot) \quad \text{in } L^2(\omega\times[0,A]),
\]
which, together with Proposition \ref{th1.3}, implies that
\[
\lim_{\epsilon\to 0}\mathds{1}_{\lbrace\omega\times(0,\epsilon)\rbrace}\,v^{\epsilon} \rightharpoonup \mathds{1}_{\omega\times\lbrace 0\rbrace}\, v_1 \quad \text{weakly in } L^2(\omega\times[0,A]),
\]
with
\begin{align}
    v_1(x,t)= -\int_{t}^A \beta(x,a)\pi(a)\, \frac{e^{t\triangle}y_{0}(x,a-t)}{\pi(a-t)\mathds{1}_{\omega}(x)}\,da \quad \text{for all } t<a.
\end{align}
{\bf Step 2} : For \( t-a>0 \), we have
\begin{align}
y^{\epsilon}(x,a,t)=\pi(a)e^{a\triangle}\eta(x,t-a)\left[1-\dfrac{\displaystyle\int_{0}^{a}\frac{e^{-s\triangle}\mathds{1}_{\omega}(x)\mathds{1}_{[0,\epsilon]}(z)}{\pi(z)}dz}{\displaystyle\int_{0}^{A-(t-a)}\frac{e^{-s\triangle}\mathds{1}_{\omega}(x)\mathds{1}_{[0,\epsilon]}(z)}{\pi(z)}dz}\right].
\end{align}
Assuming that \(\epsilon\) is sufficiently small so that \(\epsilon < A-(t-a)\) and restricting to \(a\in [0,\epsilon]\), the expression simplifies to
\begin{align}
y^{\epsilon}(x,a,t)=\pi(a)e^{a\triangle}\eta(x,t-a)\left[1-\dfrac{\displaystyle\int_{0}^{a}\frac{e^{-s\triangle}\mathds{1}_{\omega}(x)}{\pi(s)}ds}{\displaystyle\int_{0}^{\epsilon}\frac{e^{-z\triangle}\mathds{1}_{\omega}(x)}{\pi(z)}dz}\right],
\end{align}
Hence, we obtain
\begin{align}
\lim_{\epsilon\longrightarrow 0} y^{\epsilon}(x,a,t)=0 \qquad\text{for every}\quad t-a>0,\quad x\in \Omega\;\text{a. e.}.
\end{align}

Next, we demonstrate that for each fixed \(t\in[0,A]\), the function \(y^{\epsilon}(\cdot,\cdot,t)\) is dominated by an element of \(L^{2}(0,A; L^2(\Omega))\). According to \eqref{e3.5}, we have
\begin{align}
y^{\epsilon}(x,a,t)=\pi(a)e^{a\triangle}\eta(x,t-a)\left[1-\dfrac{\displaystyle\int_{0}^{a}\frac{e^{-s\triangle}\mathds{1}_{\lbrace\omega\times[0,\epsilon]\rbrace}(x,z)}{\pi(z)}dz}{\displaystyle\int_{0}^{A-(t-a)}\frac{e^{-z\triangle}\mathds{1}_{\lbrace\omega\times[0,\epsilon]\rbrace}(x,z)}{\pi(z)}dz}\right] \qquad t-a>0, \quad \forall \;x\in \Omega.
\end{align}
Thus, by taking the \(L^2\)-norm and applying the uniform bound \(\Vert \eta\Vert_{L^{\infty}(\Omega\times[0,A])}\), we obtain
\begin{align}
\int\limits_0^A \Vert y^{\epsilon}(.,a,t)\Vert_{L^2(\Omega)}^2 da\leq\Vert\eta\Vert^2_{L^{\infty}(\Omega\times[0,A])}\int\limits_0^A\int\limits_\Omega\left[1-\dfrac{\displaystyle\int_{0}^{a}\frac{e^{-s\triangle}\mathds{1}_{\lbrace\omega\times[0,\epsilon]\rbrace}(x,z)}{\pi(z)}dz}{\displaystyle\int_{0}^{A-(t-a)}\frac{e^{-z\triangle}\mathds{1}_{\lbrace\omega\times[0,\epsilon]\rbrace}(x,z)}{\pi(z)}dz}\right]^2dxda,
\end{align}
This yields the uniform estimate
\begin{align}
\Vert y^{\epsilon}(.,.,t)\Vert_{L^2(\Omega\times[0,A])}^2\leq\Vert\eta\Vert^2_{L^{\infty}(\Omega\times[0,A])}A \vert \Omega\vert,\quad t\in (0,A).
\end{align}

In summary, \(y^{\epsilon}(\cdot,\cdot,t)\) converges in \(L^{2}(0,A; L^2(\Omega))\) for every \(t\in (0,A)\) to the controlled state \(y(x,a,t)\) defined by
\begin{equation}
y(x,a,t)=\left\lbrace
\begin{array}{ll}
\dfrac{\pi(a)}{\pi(a-t)}e^{t\triangle}y_{0}(x,a-t)&\qquad  t-a<0,\\
\\0 &\qquad t-a>0.
\end{array}
\right.
\end{equation}

{\bf Step 3 } : The system state in \eqref{ee4.7} can be written as follows:
\begin{equation}
y(x,a,t)=\left\lbrace
\begin{array}{ll}
\frac{\pi(a)}{\pi(a-t)}e^{t\triangle}y_{0}(x,a-t)+\pi(a)e^{a\triangle}\mathds{1}_{\omega}(x)v(x,t-a)\quad t\leq a,\quad (x\in \Omega),\\
\\\pi(a)e^{a\triangle} b(x,t-a)+\pi(a)e^{a\triangle}\mathds{1}_{\omega}(x)v(x,t-a)\quad t> a, \quad (x\in \Omega).
\end{array}
\right.
\end{equation}
By applying the same approach as in Propositions \ref{pr2.3} and \ref{pr2.5} from the Appendix~\ref{annexe:A} to system \eqref{eq1.1}, we obtain the null-controllability result.  In particular, the null-control is given by
\begin{equation}\label{eq2.81}
v(x,t)=
\begin{cases}
0, & \quad t-a<0,\\[1ex]
-\displaystyle\int_{t}^A \beta(x,a)\pi(a)\,\frac{e^{t\triangle}y_{0}(x,a-t)}{\pi(a-t)\mathds{1}_{\omega}(x)}\,da, & \quad t-a>0,
\end{cases}
\end{equation}
where
\[
b(x,t)=\int_{t}^A \beta(x,a)\pi(a)\,\frac{e^{t\triangle}y_{0}(x,a-t)}{\pi(a-t)}da.
\]
Thus, the controlled state becomes
\begin{equation}\label{e4.25}
y(x,a,t)=\left\lbrace
\begin{array}{ll}
\dfrac{\pi(a)}{\pi(a-t)}e^{t\triangle}y_{0}(x,a-t)&\qquad  t-a<0,\\
\\0 &\qquad t-a>0.
\end{array}
\right.
\end{equation}
 Accordingly, using the same strategy as in Proposition \ref{th1.3} of Appendix~\ref{annexe:A}, we establish that the control $v \equiv v_1$, applied to systems \eqref{eq1.1}–\eqref{eq1.6} with the controlled state defined by \eqref{e4.25}, satisfies condition \eqref{e3.8} at time $T > A$.

 Moreover, for every $(x,t)\in\omega\times(0,T)$, define

\begin{align}
  I(x,t)=\int_{t}^{A}\beta(x,a)\,\pi(a)\,\frac{(e^{t\Delta}y_{0})(x,a-t)}{\pi(a-t)}\,da.  
\end{align}
Then,

\begin{align}
    \|v\|^2_{L^2(\omega\times(0,T))} = \int_{0}^{T}\!\int_{\omega}|I(x,t)|^2\,dx\,dt.
\end{align}

Since $\pi$ is non-increasing on $[0,A]$, we have $\left\Vert \frac{\pi(a)}{\pi(a-t)} \right\Vert \leq 1,$ for all $a \ge t$. Therefore,

$$
\left|\beta(x,a)\,\pi(a)\,\frac{(e^{t\Delta}y_{0})(x,a-t)}{\pi(a-t)}\right|
\le \|\beta\|_{L^\infty}\,\left|e^{t\Delta}y_{0}(x,a-t)\right|.
$$

By the Cauchy–Schwarz inequality, it follows that

$$
|I(x,t)|^2 \le \|\beta\|_{L^\infty}^2\,(A-t)\int_{a=t}^{A}\left|e^{t\Delta}y_{0}(x,a-t)\right|^2\,da.
$$

Setting $s = a - t$, and using \cite[Lemma 4.1]{ref3}, we obtain

$$
\|v\|^2_{L^2(\omega\times(0,T))} \le \|\beta\|_{L^\infty}^2\int_{0}^{T}(A-t)\int_{0}^{A-t}\!\int_{\Omega}|y_{0}(\xi,s)|^2\,d\xi\,ds\,dt.
$$

Applying Fubini’s theorem to the domain $\{0 \le t \le A,\; 0 \le s \le A-t\}$, and observing that

$$
\int_{0}^{A-s}(A-t)\,dt = \tfrac12(A-s)^2 \le \tfrac12 A^2,
$$

we finally get

$$
\|v\|^2_{L^2(\omega\times(0,T))} \le  \frac{\|\beta\|_{L^\infty}^2 A^2}{2} \,\|y_{0}\|^2.
$$

Consequently, by setting $K(A) =\frac{ \|\beta\|_{L^\infty} A}{\sqrt{2}}$, we deduce the estimate

\begin{align}
    \|v\|_{L^2(\omega\times(0,T))} \le K(A)\,\|y_{0}\|.
\end{align}

\end{proof}
\begin{remark}
In accordance with Theorem \ref{th1.4}, the boundary control problem \eqref{eq1.1} is reformulated as the control problem \eqref{eq1.6} by introducing an age-dependent Dirac delta, which naturally leads to the definition of the operator $B$.
\end{remark}
\section{Proof of Exponential Turnpike Property}\label{s5}
\subsection{Linear-Quadratic optimal control problems}

In this section, we introduce an optimal control system governed by a quadratic cost functional to be minimized. Such problems are commonly referred to as linear–quadratic optimal control problems. By Theorem \ref{th1.4}, for every $T>A$ there exists a control $v\in L^2(\omega\times[0,T])$ such that the solution $y$ of \eqref{eq1.1} satisfies

$$
y(x,a,T)=0 \quad \text{a.e. }(x,a)\in \Omega\times[0,A],
$$

for all initial data $y_0\in H$; this property is equivalent to trajectory controllability. Our goal is to steer the state toward a prescribed trajectory $\bar y$ so as to track a continuously evolving target $y_d\in H$.

We now formulate the corresponding linear–quadratic optimal control problem. Originating from a population dynamics model with age structure and spatial diffusion, this dynamic optimization problem consists in finding $v^*\;\in\;L^2(\omega\times[0,T])
$
that minimizes the cost functional
 \begin{align}\label{eq3.1}
     J_1(v^*)=\min_{v\in L^2(\omega\times[0,T])}J_1(v)
 \end{align}
	where
 \begin{align}\label{eq3.2}
     J_1(v)=\dfrac{N}{2}\int\limits_{0}^{T}\Vert y(t)-y_d\Vert^2 dt+\dfrac{1}{2}\int\limits_{0}^{T}\Vert v(t)\Vert^2_{L^2(\omega)}dt+\dfrac{1}{2}\langle y_T(T),p(T)\rangle
 \end{align}
   and $N>0,$  
    its solution $y(t)=y(x,a,t)$ subject to 
\begin{equation}\label{eq3.3}	
\left\lbrace\begin{array}{ll}
\dfrac{\partial y}{\partial t}(x,a,t)=\mathcal{A}_m y(x,a,t)+ Bv(x,a,t)&\hbox{ in }Q,\\  
\\y\left(x,a,0\right)=y_{0}(x,a)&\hbox{ in }  Q_A.
\end{array}\right.
\end{equation}
 The system \eqref{eq1.1} is reformulated in the form \eqref{eq3.3} thanks to the construction from {\bf Section \ref{s2} (and Theorem \ref{th1.4})}, with the control operator denoted by $B$. Minimizing the functional \eqref{eq3.1} subject to constraint \eqref{eq3.3} corresponds to the “linear regulator problem,” or equivalently, the linear-quadratic problem.

To study the turnpike property, we introduce the steady‑state versions of \eqref{eq3.1} and \eqref{eq3.3}. We thus consider the associated stationary system :
\begin{equation}\label{eq3.4}	 
-\mathcal{A}_m y(x,a)=Bv(x,a) \qquad\text{in}\; Q_1,
\end{equation}
and the following minimization problem 
\begin{equation}\label{eq3.5}
J_2(\bar{v})=\min_{v\in L^2(\omega)}J_2(v),
\end{equation}
where 
\begin{equation}\label{eq3.6}
J_2(v)=\frac{N}{2}\Vert y-y_d\Vert^2+\frac{1}{2}\Vert v\Vert^2_{L^2(\omega)},
\end{equation}
and $y$ denotes the solution to \eqref{eq3.4}.  We now state the following proposition.
\begin{proposition}\label{pr3.1}
There exists an optimal control minimizing the cost functional \(J_1\).
\end{proposition}
To prove that $J_1$ attains its minimum over $L^2(\omega \times [0,T])$, we will establish the following properties:
\begin{enumerate}
\item $J_1$ is convex,
\item \(J_1\)  is weakly lower semicontinuous,
\item \(J_1\) is coercive, i.e., $$\lim_{\Vert v\Vert_{L^2(\omega\times(0,T))} \longrightarrow\infty}J_1(v)=\infty.$$
\end{enumerate}
\begin{proof}[of Proposition \ref{pr3.1}]
{\bf Coercivity of the cost functional $J_1$
}:

Let $y = e^{\alpha t}\tilde{y}$ be a solution of
\begin{equation}\label{eq3.7}	 
\left\lbrace
\begin{array}{ll}
\partial_{t}\tilde{y}(x,a,t) + \partial_{a}\tilde{y}(x,a,t) - \Delta \tilde{y}(x,a,t) + (\mu(a)+\alpha)\tilde{y}(x,a,t) = e^{-\alpha t}Bv(x,a,t) & \text{ in } Q,\\[1mm]
\partial_{\nu}\tilde{y}(x,a,t)=0 & \text{ on } \Sigma,\\[1mm]
\tilde{y}(x,0,t)=\displaystyle\int_{0}^{A}\beta(x,a)\tilde{y}(x,a,t)da & \text{ in } Q_{T},\\[1mm]
\tilde{y}(x,a,0)= y_{0}(x,a) & \text{ in } Q_A,
\end{array}
\right.
\end{equation}
where the parameter \(\alpha\) will be determined later.

Multiplying the first equation by \(\tilde{y}\) and integrating the resulting expression over \(Q\), we obtain
\[
\int_{0}^{A}\int_{\Omega}\tilde{y}^2(x,a,T)\,dx\,da + \int_{0}^{T}\int_{0}^{A}\int_{\Omega}2(\mu(a)+\alpha)\tilde{y}^2(x,a,t)\,dx\,da\,dt
\]
\[
\leq 2\Vert v\Vert^2_{L^2(\omega\times[0,T])} + 2A\Vert \beta\Vert^2_{L^\infty}\int_{0}^{T}\int_{0}^{A}\int_{\Omega}\tilde{y}^2(x,a,t)\,dx\,da\,dt + \Vert y_0\Vert^2.
\]
Choosing \(\alpha = A\Vert \beta\Vert^2 + 1\), it follows that
\begin{align}\label{eq3.8}
    \Vert y(T)\Vert \leq C'\Bigl(\Vert v\Vert_{L^2(\omega\times[0,T])} + \Vert y_0\Vert\Bigr).
\end{align}
Hence, we deduce that
\[
\bigl|\langle y_T(T),p(T)\rangle\bigr| \leq C'\Bigl(\Vert v\Vert_{L^2(\omega\times[0,T])} + \Vert y_0\Vert\Bigr)\Vert p(T)\Vert.
\]
Moreover, since
\[
\frac{N}{2}\int_{0}^{T}\Vert y(t)-y_d\Vert^2\,dt > 0,
\]
we obtain
\[
J_1(v) \geq \Vert v\Vert_{L^2(\omega\times[0,T])}\left(\frac{1}{2}\Vert v\Vert_{L^2(\omega\times[0,T])} - C'\Vert p(T)\Vert\right) - C'\Vert y^0\Vert\Vert p(T)\Vert.
\]
Since the right-hand side tends to infinity as \(\Vert v\Vert \to \infty\), we conclude that
\[
\lim_{\Vert v\Vert \to \infty}J_1(v)=\infty,
\]
thus establishing the coercivity of the cost in \(L^2(\omega\times[0,T])\).

{\bf The Convexity of the Cost} :

The functional \(J_1\) is quadratic, and the dependence of \(y\) on \(v\) is affine. Consequently, \(J_1\) is strictly convex.

{\bf Continuity}
Let $y = y(v,y_0)$ be a solution of \eqref{eq1.1}. The mapping

$$
(v, y_0) \longmapsto y(v, y_0)
$$

defines a continuous function from $L^2(\omega \times (0,T)) \times H$ into $L^2(0,T; H)$. To establish this continuity property, consider an arbitrary pair $(v, y_0) \in L^2(\omega \times (0,T)) \times H$ and set $\bar{y} = y(v, y_0)$. Then $\bar{y}$ satisfies the following problem:
\[
\begin{cases}
\partial_{t}\bar{y} + \partial_{a}\bar{y} - \Delta \bar{y} + \mu(a) \bar{y} = Bv  & \text{in } Q, \\[1ex]
\partial_{\nu}\bar{y} = 0 & \text{on } \Sigma, \\[1ex]
\bar{y}(x,0,t) = \displaystyle\int_{0}^{A}\beta(x,a)\bar{y}\,da & \text{in } Q_{T}, \\[1ex]
\bar{y}(x,a,0) = y_0  & \text{in } Q_A.
\end{cases}
\]

Next, define $z = e^{-rt}\bar{y}$  for some  $r > 0.$ Then, \(z\) is a solution to:
\[
\begin{cases}
\partial_{t}z + \partial_{a}z - \Delta z + (r+\mu(a))z = e^{-rt}Bv +  & \text{in } Q, \\[1ex]
\partial_{\nu}z = 0 & \text{on } \Sigma, \\[1ex]
z(x,0,t) = \displaystyle\int_{0}^{A}\beta(x,a)z\,da & \text{in } Q_{T}, \\[1ex]
z(x,a,0) = y_0 & \text{in } Q_A.
\end{cases}
\]

Multiplying the first equation of this system by \(z\) and integrating by parts over \(Q\), we obtain
\[
\frac{1}{2}\Vert z(\cdot,\cdot,T)\Vert^2 - \frac{1}{2}\Vert y_0\Vert^2 - \frac{1}{2}\Vert z(\cdot,0,t)\Vert_{L^2(\Omega\times(0,T))}^2 + \Vert\nabla z\Vert_{L^2((0,T);H)}^2 + \int_0^T\int_0^A\int_{\Omega}(r+\mu(a))z^2\,dx\,da\,dt = \int_0^T\int_{\omega}e^{-rt}vz\,dx\,dt.
\]

From this, we deduce that
\[
\Vert\nabla z\Vert^2_{L^2((0,T);H)} + r\Vert z\Vert^2 \le \frac{1}{2}\Vert y_0\Vert^2 + \frac{1}{2}\Vert z(\cdot,0,t)\Vert^2_{L^2(\Omega\times(0,T))} + \frac{1}{2}\Vert v\Vert^2 + \frac{1}{2}\Vert z\Vert^2.
\]
Since
\[
\frac{1}{2}\Vert z(\cdot,0,t)\Vert^2_{L^2(\Omega\times(0,T))} \le \frac{1}{2}A\Vert\beta\Vert^2_{L^{\infty}}\Vert z\Vert^2,
\]
we obtain
\[
\Vert\nabla z\Vert^2_{L^2((0,A)\times(0,T))} + \left(r - \frac{A\Vert\beta\Vert^2_{L^{\infty}}+1}{2}\right)\Vert z\Vert^2 \le \frac{1}{2}\Vert y_0\Vert^2 + \frac{1}{2}\Vert v\Vert^2.
\]
Choosing \(r\) such that
\[
r>\frac{A\Vert\beta\Vert^2_{L^{\infty}}+1}{2},
\]
yields
\[
\Vert z\Vert^2_{L^2((0,A)\times(0,T); H^1_0(\Omega))} \le \frac{1}{2}\Vert y_0\Vert^2 + \frac{1}{2}\Vert v\Vert^2.
\]
This implies that
\[
\Vert \bar{y}\Vert_{L^2((0,A)\times(0,T); H^1_0(\Omega))} \le e^{rT}\frac{\sqrt{2}}{2}\left(\Vert y_0\Vert^2 + \Vert v\Vert^2\right).
\]
 
 Since \(H^1_0(\Omega)\) is compactly embedded in \(L^2(\Omega)\), it follows that \(\bar y\) is strongly continuous in $L^2\bigl((0,A)\times(0,T);\,L^2(\Omega)\bigr).$

 \(J_1\) is continuous because the state \(y\) depends continuously on the initial data \((v,y_0)\). Therefore, by virtue of its continuity and coercivity, \(J_1\) is weakly lower semicontinuous. It follows that $v^* \in L^2(\omega \times [0,T])$ satisfies \eqref{eq3.1} and therefore $v^*$ is an optimal control. We denote it by $v_T$, corresponding to the optimal state $y_T$.
\end{proof}
 Considering the stationary version of the state equation, the associated steady‑state cost $J_2$ admits an optimal control for problem \eqref{eq3.4} by the same standard methods used in Proposition \ref{pr3.1}. We denote this optimal pair by $(\bar y,\bar v)$.
\subsection{Preliminary}
In this subsection, we present, on the one hand, the optimality systems associated with the dynamic and steady control problems based on Pontryagin’s maximum principle, and on the other hand, some tools (or definitions) useful for the study of the stabilization of a dynamic system.

{\bf Necessary Optimality Conditions for the Dynamic Optimal Control Problem} :  
By applying Pontryagin’s Maximum Principle to the optimal pair \((y_T,v_T)\) associated with \eqref{eq3.1}, we deduce the existence of an adjoint variable \(p_T \in L^2([0,T]; H)\) such that
 \begin{equation}\label{eq3.9}	
\left\lbrace\begin{array}{ll}
\dfrac{\partial y_T}{\partial t}(x,a,t)=\mathcal{A}_m y_T(x,a,t)+Bv_T(x,a,t) &\hbox{ in }Q,\\  
\\y_T\left(x,a,0\right)=y_{0}(x,a)&\hbox{ in }  Q_A,\\
\\-\dfrac{\partial p_T}{\partial t}(x,a,t)-\mathcal{A}^*_m p_T(x,a,t)=N(y_T-y_d) &\hbox{ in }Q,\\  
\\p_T\left(x,a,T\right)=p_T(T) &\hbox{ in }Q_A,\\
\\ v_T(x,a,t)=-B^*p_T(x,a,t) &\hbox{ in }Q.
	\end{array}\right.
\end{equation}
{\bf Necessary Optimality Conditions for the Static Optimal Control Problem} : Similarly, by applying the Pontryagin Maximum Principle to any optimal pair $(\bar{y}, \bar{v})$ of \eqref{eq3.5}, we deduce the existence of an adjoint state $\bar{p} \in H$ such that:
 \begin{equation}\label{eq3.10}	 
	\left\lbrace\begin{array}{ll}
	-\mathcal{A}_m\bar{y}(x,a)=B\bar{v} (x,a) &\hbox{ in }Q_1,\\  
\\-\mathcal{A}^*_m \bar{p}(x,a)=N(\bar{y}-y_d) &\hbox{ in }Q_1,\\
\\\bar{v}(x,a)=-B^*\Bar{p}(x,a) & \hbox{ in }Q_1.
	\end{array}\right.
\end{equation}
The notion of stability plays a crucial role in the realization of the turnpike phenomenon. Its characterization in infinite-dimensional settings is significantly more intricate than in the finite-dimensional case. In \cite{ref31}, it is already established that if a linear control system is both controllable and observable in finite dimensions, then there exists a linear feedback law  \(y\mapsto Ky\) such that the zero equilibrium is globally asymptotically stable for the equation  
\[
\partial_t y = \mathcal{A}_my + BK y.
\]
Essentially, there are three main types of conditions that ensure the stabilizability of a system state, such as that in (\ref{eq3.3}). The first type requires specific structural properties of the generator $K$ ; the second relies on the concept of controllability; and the third is expressed in terms of a corresponding Riccati operator equation.

System (\ref{eq3.3}) is said to be exponentially stabilizable if it is possible to design a state feedback law that stabilizes it, as specified in the following definition.
\begin{definition}\label{de3.3}
The state equation in (\ref{eq3.3}) is said to be stabilizable if there exists a feedback matrix \(K\) such that
\[
\Vert \exp(t(\mathcal{A}_m - BK)) y_0 \Vert \leq \Vert y_0 \Vert \exp(-t\nu)
\]
for some constants \(C,\nu > 0\).
\end{definition}

In practical applications, the choice of $K$ is crucial for achieving optimal performance of system (\ref{eq3.3}). In general, the objective is to design robust feedback laws suitable for uncertain linear control systems. Various tools are available to tackle this problem. In our approach, we derive a simplified form of optimal control in which the control is expressed as a linear function of the corresponding state, a formulation that is particularly advantageous in engineering applications. Specifically, this feedback is constructed using the Riccati operator associated with the finite-horizon problem. This operator is known to generate a feedback control law that ensures exponential stability over an infinite time horizon, as shown in \cite{ref32}.
 \subsection{Proof of  Theorem \ref{th1.5}}

Before proceeding further, we define the following cost functional for the case where the target satisfies \(y_d\equiv 0\) :
 \begin{equation}\label{eq3.11}
J_1^0(v_T)=\dfrac{N}{2}\int\limits_{0}^{T}\Vert y_T(t)\Vert^2dt+\dfrac{1}{2}\int\limits_{0}^{T}\Vert v_T(t)\Vert_{L^2(\omega)}^2dt+\dfrac{1}{2}\Vert y_T(T)\Vert^2,
\end{equation} 
with \(v_T\) being the optimal solution to
\begin{equation}\label{eq3.12}
    \min\limits_{v\in L^2(\omega\times[0,T])}J_1^0(v).
\end{equation}
The cost functional is well-defined for every $v\in L^2(\omega\times[0,T])$. Indeed, since the $C_0$-semigroup $e^{t\mathcal{A}_m}$ is exponentially stable, and therefore $L^2$-stable, it follows that $e^{t\mathcal{A}_m}y_0$ belongs to $L^2(0,T;H)$ (see, e.g., \cite{ref59} for details).

Since equation \((\ref{eq3.3})\) has been shown to be null-controllable under the condition $T > A$, we now state the following Lemma.

\begin{lemma}\label{le3.4}
Then there exists a unique, symmetric, positive definite matrix \(\mathcal{E}\), belonging to \(C^1\left([0,+\infty);H\right)\), which solves the Riccati differential equation associated with the linear–quadratic optimal control problem (\ref{eq3.12}) as follows:
\begin{equation}\label{eq3.13}
\left\lbrace\begin{array}{ll}
\mathcal{E}_t(T-t)=I.N+\mathcal{E}(T-t)\mathcal{A}_m+\mathcal{A}_m^{*}\mathcal{E}(T-t)-\mathcal{E}(T-t)BB^*\mathcal{E}(T-t)&\hbox{ in } (0,\infty)\\
\mathcal{E}(0)=I
\end{array}\right.
\end{equation}
with the terminal condition \(p(x,a,T)=y(x,a,T)\), and the corresponding feedback control law
\[
v_T=-B^*\mathcal{E}(T-t)y_T.
\]

Furthermore, the following properties hold:
\begin{enumerate}
    \item[(i)] There exists \(M>0\) such that \(0<\mathcal{E}(t)\leq M\) for every \(t>0\).
    \item[(ii)]  For all \(t_1 \leq t_2\), we have \(\mathcal{E}(t_1)\leq \mathcal{E}(t_2)\).
    \item[(iii)]  As \(t \to +\infty\), \(\mathcal{E}(t)\) converges to \(\hat{E}\), where \(\hat{E}\) is the unique symmetric, positive definite matrix satisfying the algebraic Riccati equation
\begin{equation}\label{eq3.14}	\hat{E}\mathcal{A}_m+\mathcal{A}_{m}^{*}\hat{E}-\hat{E}BB^*\hat{E}+I.N=0,
\end{equation}
\item[(iv)]  There exists a Lyapunov function of the form

 \begin{align}
      V(y_T) = \langle \hat{E}\,y_T,\;y_T\rangle
 \end{align}
 such that the closed-loop system

 \begin{equation}\label{eq3.15}
\left\lbrace\begin{array}{ll}
\partial_t y_T=(\mathcal{A}_m-BB^*\hat{E})y_T&\hbox{ in }(0,\infty)\\
y_T(0)=y_0
\end{array}\right.
\end{equation}
 is globally asymptotically stable.
\end{enumerate}
\end{lemma}
 \begin{proof}[of Lemma \ref{le3.4}]
In the reference case $y_d \equiv 0$, we denote by $J_1^0$ the cost functional defined in \eqref{eq3.11}. First, we introduce the matrix term $\mathcal{E}(T)$. The solution of the optimal control problem \eqref{eq3.12} yields an optimal cost given by the nonnegative quadratic form
\begin{align}
    \dfrac{1}{2}\langle\mathcal{E}(T)y_0 ,y_0\rangle.
\end{align}
For \(T>0\), we define the operator 
\[
\mathcal{E}(T): H \to H,\quad \text{by} \quad \mathcal{E}(T)y_0(x,a) := p_T(x,a,0) \quad \text{for } y_0 \in H.
\]
It is evident that 
\[
p_T(x,a,t) = \mathcal{E}(T-t)y_T(x,a,t),
\]
which implies  
\begin{equation}\label{eq3.17}
v_T(x,t) = -B^*\mathcal{E}(T-t)y_T(x,a,t) \quad \forall\, t \in [0,T].
\end{equation}
The expression \eqref{eq3.17} gives the feedback optimal control for the cost functional \eqref{eq3.11}.

The operator $\mathcal{E}(T)$ is obtained by multiplying the first equation of \eqref{eq3.9} by $p_T$ and integrating over $Q$. This yields the identities
 \begin{align}
\langle\mathcal{E}(T)y_0 ,y_0 \rangle = N\int\limits_{0}^{T}\Vert y_T(t)\Vert^2dt+\int\limits_{0}^{T}\Vert p_T(0,t)\Vert_{L^2(\omega)}^2dt+\Vert y_T(T)\Vert^2,
\end{align}
and 
\begin{equation}\label{eq3.19}
\dfrac{1}{2}\langle\mathcal{E}(T)y_0 ,y_0 \rangle =\inf_{v\in L^2 (\omega\times[0,T])} J_1 ^0 (v) .
\end{equation} 
 Thus, the minimal value of the cost functional \eqref{eq3.11} is $
 \frac{1}{2}\bigl\langle \mathcal{E}(T)y_0,\,y_0 \bigr\rangle,
 $ and the corresponding optimal control is given by \eqref{eq3.17}.

{\bf \underline{Case (i)}} : We now prove that there exists a constant \(M>0\) such that \(0 < \mathcal{E}(T) \leq M\) for all \(T > 0\).

Starting from the estimate
\begin{align}
\dfrac{1}{2}\langle\mathcal{E}(T)y_0 ,y_0 \rangle \leq \dfrac{N}{2}\int\limits_{0}^{T}\Vert y(t)\Vert^2dt+\dfrac{1}{2}\int\limits_{0}^{T}\Vert v(t)\Vert_{L^2(\omega)}^2dt
\end{align}
where \(v\) denotes the null-control for \((\ref{eq3.3})\) at time \(T > A\). Consequently, there exists a constant \(M>0\) such that
\begin{align}
\langle\mathcal{E}(T)y_0 ,y_0 \rangle\leq M \Vert y_0\Vert^2.
\end{align} 
This shows that \(\mathcal{E}(T)\) is uniformly bounded from above in \(H\) as \(T\to\infty\), so that \(\mathcal{E}(T) \leq M\).

Since the cost functional (\ref{eq3.11}) is nonnegative by definition, we also have \(\langle\mathcal{E}(T)y_0, y_0\rangle \geq 0\). Moreover, if \(\langle\mathcal{E}(T)y_0, y_0\rangle = 0\), then it must be that \(y_T \equiv 0\) and \(v\equiv 0\) for all \(t\in[0,T]\) and almost every \(x \in \Omega\). By uniqueness,  \(y_0 \equiv 0\). Thus, \(\langle\mathcal{E}(T)y_0, y_0\rangle = 0\) if and only if \(y_0 \equiv 0\), showing that  \(\mathcal{E}(T)\)   is positive definite. 

Therefore, for all $T>A,$
\[
0 < \mathcal{E}(T) \leq M,
\]
as claimed.

{\bf \underline{Case (ii)}} :We now show that the operator \(\mathcal{E}\) is nondecreasing in the time horizon \(T\).  Intuitively, extending the integration interval \([0,T]\) in the cost functional can only increase the minimal cost. Let $0<T\leq T_1,$ and denote by  \(v_{T}\) and \(v_{T_1}\) the respective minimizers of  \(J_{T}^0\) and \(J_{T_1}^0,\) with corresponding states \(y_{T}\) and \(y_{T_1}.\) By definition (cf.\ref{eq3.19})
\[
\frac{1}{2}\langle \mathcal{E}(T)y_0, y_0 \rangle 
= \frac{N}{2}\int_{0}^{T}\|y_T(t)\|^2\,dt + \frac{1}{2}\int_{0}^{T}\|v_T(t)\|_{L^2(\omega)}^2\,dt + \frac{1}{2}\|y_T(T)\|^2.
\]
Since \(v_{T_1}\)  is admissible for the shorter horizon \([0,T]\), we have
\[
\frac{1}{2}\langle \mathcal{E}(T)y_0, y_0 \rangle 
\leq \frac{N}{2}\int_{0}^{T}\|y_{T_1}(t)\|^2\,dt + \frac{1}{2}\int_{0}^{T}\|v_{T_1}(t)\|_{L^2(\omega)}^2\,dt + \frac{1}{2}\|y_{T_1}(T)\|^2.
\]
Moreover, since \(T \leq T_1\) each integral and terminal term over  \([0,T]\) is bounded by its counterpart over \([0,T_1].\) Hence,
\[
\frac{N}{2}\int_{0}^{T}\|y_{T_1}(t)\|^2\,dt + \frac{1}{2}\int_{0}^{T}\|v_{T_1}(t)\|_{L^2(\omega)}^2\,dt + \frac{1}{2}\|y_{T_1}(T)\|^2
\leq \frac{N}{2}\int_{0}^{T_1}\|y_{T_1}(t)\|^2\,dt + \frac{1}{2}\int_{0}^{T_1}\|v_{T_1}(t)\|_{L^2(\omega)}^2\,dt + \frac{1}{2}\|y_{T_1}(T_1)\|^2.
\]
Combining these, we obtain
\[
\frac{1}{2}\langle \mathcal{E}(T)y_0, y_0 \rangle \leq \frac{1}{2}\langle \mathcal{E}(T_1)y_0, y_0 \rangle,
\]
for all $y_0\in H.$ Thus \(\mathcal{E}(T)\leq \mathcal{E}(T_1),\) i.e. $\mathcal{E}(T)$  is monotonically increasing in \(T\).

{\bf \underline{Case (iii)}} : Thanks to the monotonicity and uniform boundedness of \(\mathcal{E}(T),\) there exists a unique symmetric positive-definite operator \(\hat{E}\), such that
\[
\lim_{T\to\infty}\mathcal{E}(T)=\hat{E}.
\]
This limit \(\hat{E}\) satisfies the algebraic Riccati equation (\ref{eq3.14}). Moreover, if \((y_\infty, p_\infty)\) denotes the solution pair of (\ref{eq3.9}) taken over the infinite horizon \(T=\infty\), then the minimal cost is \(\langle\hat{E}y_0,y_0\rangle\).
 
 {\bf \underline{Case (iv)}} : We now demonstrate that the closed‑loop system  (\ref{eq3.15}) is globally asymptotically stable. Recall that the limiting operator $\hat{E}=\lim_{T\to\infty}\mathcal{E}(T)$  is symmetric and positive‑definite, and satisfies the algebraic Riccati equation \eqref{eq3.14}. The corresponding infinite‑horizon optimal pair \((y_\infty, p_\infty)\) yields the feedback control \[
v_{\infty}(x,t) = -B^*\hat{E}\,y_{\infty}(x,a,t).
\]
 
Define  the Lyapunov function
\[
V(y_{\infty}) = \langle\hat{E}y_{\infty}, y_{\infty}\rangle.
\]
Then,
\begin{align}
    \frac{d}{dt}\langle\hat{E}y_{\infty},y_{\infty}\rangle = 2\langle \hat{E}(y_{\infty})_t, y_{\infty}\rangle 
= 2\langle \hat{E}\mathcal{A}_m y_{\infty}, y_{\infty}\rangle - 2\langle \hat{E}BB^*\hat{E}y_{\infty}, y_{\infty}\rangle.
\end{align}
Using symmetry and the Riccati equation  (\ref{eq3.14}), one obtains
\begin{align}
    \frac{d}{dt}\langle\hat{E}y_{\infty},y_{\infty}\rangle = -\Bigl( N\Vert y_{\infty}(t)\Vert^2 + \Vert B^*\hat{E}y_{\infty}(t)\Vert^2_{L^2(\omega)} \Bigr)\leq 0.
\end{align}
with strict negativity unless $y_0=0.$ Hence $V$ is a proper Lyapunov function :
\begin{itemize}
    \item Positive definiteness : $V(y_{\infty})>0$ for $y_{\infty}\neq 0,$
    \item  Radial unboundedness :  $\Vert y_{\infty}\Vert\to\infty \quad\Longrightarrow\quad V(y_{\infty})\to\infty.$
\end{itemize}
By standard results (see, e.g., \cite[Theorem 12.2]{ref21}, \cite[Theorem 1.2.3]{ref61}),  these properties guarantee global asymptotic stability of the closed‑loop system.

Moreover, integrating $\partial_t V$ over $(0,+\infty)$  and using $\lim_{t\to\infty}y_{\infty}(t)=0$ shows

\begin{equation}\label{eq3.27}
\dfrac{1}{2}\langle\hat{E}y_0,y_0\rangle =\inf_{v\in L^2(\omega\times[0,T])} J^0 _{\infty}(v),
\end{equation}
where \(J^0_{\infty}\) 
  is the infinite‑horizon version of (\ref{eq3.11}). This confirms  \(\hat{E}\) is the unique positive‑definite solution to the algebraic Riccati equation.

\end{proof}
\begin{remark}
Controllability of the state equation is essential to ensure that the optimal pair $(y_T, v_T)$ remains uniformly bounded in $T$.  Moreover, the gain operator $B^*\mathcal{E}(T)$ converges to $B^*\hat{E}$ as $T\to\infty$. Hence, the infinite‑horizon controller $ v_\infty = -\,B^*\,\hat{E}\,y_\infty$
 is asymptotically stabilizing: it confines the state $y_\infty$ to a neighborhood of the equilibrium and maintains effective control over the entire infinite time horizon.

\end{remark}
We now build on the global asymptotic stability of \eqref{eq3.15} to show that the closed‑loop system is in fact exponentially stable. In finite dimensions, global asymptotic stability (see, e.g., \cite{ref61, ref60}) is closely tied to exponential stability: the eigenvalues of the stability matrix

$$
M = \mathcal{A}_m - B B^* \hat{E}
$$

all have strictly negative real parts; equivalently, $\Re(\sigma(M))\subset(-\infty,0)$. Moreover, by \cite[Theorem 13.1.1]{ref58}, the equilibrium $0$ is (Lyapunov) stable if and only if every eigenvalue of $M$ is either strictly negative or, if it has zero real part, is simple. This result follows from the existence of an appropriate Lyapunov function. Consequently, the closed‑loop evolution

\begin{align}
    \partial_t y_\infty(t) \;=\; M\,y_\infty(t), 
\quad y_\infty(0)=y_0,
\end{align}

is exponentially stable.

In infinite-dimensional settings, however, the operator $M$ may have more intricate spectral properties—both lower and upper stability indices can arise \cite{ref63}. In such contexts, the interplay between controllability and stabilizability is substantially more involved. The following results address conditions ensuring exponential decay in these infinite-dimensional systems.
\begin{lemma}\label{le3.8}
The null-controllable system (\ref{eq3.3}) is exponentially stabilizable.
\end{lemma}

\begin{proof}[of Lemma \ref{le3.8}]
Since $\mathcal{A}_m$ generates an exponentially stable $C_0$-semigroup on $H$ and $\hat{E}$ is bounded on $H$, Phillips’ theorem (see \cite[Theorem 1.5, p. 188]{ref37}) ensures that

$$
M \;=\;\mathcal{A}_m \;-\; B\,B^*\,\hat{E}
$$

also generates a $C_0$-semigroup with domain $D(M)=D(\mathcal{A}_m)$. Combined with the null‑controllability of system \eqref{eq3.3}, this yields

\begin{align}
    \int_{0}^{\infty}\|y_\infty(t)\|\,dt < \infty 
\quad\text{for all }y_0\in H.
\end{align}

Moreover, by \cite[Theorem 3.3(i), p.222]{ref37} (or \cite[Corollary 6.1.14]{ref48}), the semigroup $(S(t))_{t\ge0}$ generated by $M$ is exponentially stable; that is, there exist constants $C>0$ and $\nu>0$ (depending on $\mathcal{A}_m$, $\omega$, and $\hat{E}$) such that

\begin{align}
    \|S(t)\|\le C\,e^{-\nu t},\quad t\ge0.
\end{align}

Hence, the closed‑loop system \eqref{eq3.3} is exponentially stabilizable.

\end{proof}

From \cite[Theorem 4.4, p. 241]{ref37}, Lemma \ref{le3.8} implies that equation (\ref{eq3.14}) has at least one non-negative solution \(\hat{E}\in H\).

\begin{remark}
 The exponential stability established in Lemma \ref{le3.8} is, in fact, a strictly stronger result than the condition given in \cite[Theorem 3.1(ii), p. 222]{ref37}.
\end{remark}

In summary, we have shown that the spectral bound of $M$ satisfies

$$
\sup\{\Re(\nu) : \nu\in\sigma(M)\} < 0,
$$

so that $M$ is boundedly invertible (see, e.g., \cite[Proposition 3.1]{ref63}, \cite[Theorem 3.1, p. 222]{ref37}). Under the resulting feedback control, the zero state becomes the equilibrium of the dynamic system \eqref{eq3.3}.

The turnpike property describes how, over long time horizons, the optimal state, control, and adjoint state stay close to an optimal steady state for the majority of the interval. Our analysis of this phenomenon hinges on diagonalizing the Hamiltonian operator associated with the system, with stability playing a central role. In particular, the extremal triple $(\bar{y}, \bar{v}, \bar{p})$ solving the steady optimal control problem $J_2$ under constraint \eqref{eq3.4} is precisely the equilibrium of the Hamiltonian system \eqref{eq3.9}.

To analyze exponential decay toward the steady state, define the deviation variables
\begin{align}\label{ee5.29}
    \tilde{y} = y_T - \bar{y},\quad \tilde{p} = p_T - \bar{p},\quad \text{for all } t\in [0,T].
\end{align}
By combining systems \eqref{eq3.9} and \eqref{eq3.10}, these deviations satisfy the Hamiltonian system 
\begin{equation}\label{eq3.32}
\frac{\partial}{\partial t}\begin{bmatrix}\tilde{y} \\ \tilde{p}\end{bmatrix} = Ham\, \begin{bmatrix}\tilde{y} \\ \tilde{p}\end{bmatrix},
\end{equation}
where the Hamiltonian operator \(Ham\colon D(\mathcal{A}_m)\times D(\mathcal{A}_m^*)\to H\times H\) is defined by
\[
Ham = \begin{bmatrix} \mathcal{A}_m & -BB^* \\ -I.N & -\mathcal{A}_m^* \end{bmatrix}.
\]
In finite dimensions, absence of imaginary‑axis eigenvalues for $Ham$ is equivalent to the plant being both stabilizable and detectable (see, e.g., \cite{ref46}). Here, stabilizability has already been established for system \eqref{eq3.3}. Detectability (the dual concept of stabilizability) means there exists an operator $\mathcal{L}$ such that $\mathcal{A}_m + \mathcal{L}I$ generates an exponentially stable semigroup (e.g., \cite{ref64,ref52,ref32}), with observation operator $C=I$.  Exponential stabilizability of system \eqref{eq3.3} implies its exponential detectability  \cite{ref38}. Moreover, observability of the state system entails null‑controllability of its adjoint, which guarantees the existence of an operator $\mathcal{L}$ ensuring this exponential detectability. It should be emphasized that observability (and thus detectability) together with stabilizability are two key elements in the proof of Theorem \ref{th1.4}.

We next state a Lemma that is pivotal for establishing exponential decay.

  \begin{lemma}\label{le3.10}
\begin{enumerate}
  \item[(1)] There exists a unique symmetric positive definite solution \(\hat{E}\) to the algebraic Riccati equation
  \begin{equation}\label{eq3.33}
  \hat{E}\mathcal{A}_m+\mathcal{A}_m^*\hat{E}-\hat{E}BB^*\hat{E}+I.N=0,
  \end{equation}
  such that the operator \(\mathcal{A}_m-BB^*\hat{E}\) is stable ; that is, it generates a semigroup that is exponentially stable.
  
  \item[(2)] Define 
  \begin{align}
      \Lambda = \begin{bmatrix}
  I & S\\[1mm]
  \hat{E} & \hat{E}S+I
  \end{bmatrix},
  \end{align}
  where \(S\) is the solution to the Lyapunov equation
  \[
  S(\mathcal{A}_m-BB^*\hat{E}) + (\mathcal{A}_m-BB^*\hat{E})S = BB^*.
  \]
  Then, \(\Lambda\) is invertible and
  \begin{align}
    \Lambda\, Ham\, \Lambda^{-1} = \begin{bmatrix}
  \mathcal{A}_m-BB^*\hat{E} & 0\\[1mm]
  0 & -(\mathcal{A}_m^*-BB^*\hat{E})
  \end{bmatrix}.  
  \end{align}
  Consequently, the operator \(Ham\) is boundedly invertible and its diagonal blocks are generators of semigroups that are exponentially stable.
\end{enumerate}
\end{lemma}

\begin{remark}
The Hamiltonian operator $Ham$ has no purely imaginary eigenvalues; it is thus hyperbolic, meaning all its spectrum lies off the imaginary axis. Consequently, $Ham$ admits a block‑diagonal decomposition (or spectral splitting) into a “stable” part, whose spectrum has strictly negative real parts, and an “unstable” part, whose spectrum has strictly positive real parts. In infinite dimensions, this corresponds to the generator of an exponentially dichotomous semigroup, a stronger notion than mere hyperbolicity in finite dimensions.
\end{remark}

\begin{proof}[of Lemma \ref{le3.10}]
By Lemmas \ref{le3.4} and \ref{le3.8}, the limit operator $\hat{E} = \lim_{t\to\infty}\mathcal{E}(t)$ is positive definite and defines a Lyapunov function. The semigroup $(S(t))_{t\ge0}$ generated by $M := \mathcal{A}_m - B B^* \hat{E}$ is exponentially stable, and so is the semigroup generated by its adjoint $M^*$.

Consider now the operator Riccati equation for the steady gain $S$:

$$
S\,M + M^*\,S = B B^*.
$$

One shows that the block operator

$$
      \Lambda = \begin{bmatrix}
  I & S\\[1mm]
  \hat{E} & \hat{E}S+I
  \end{bmatrix},
$$

is invertible, with inverse

\[
\Lambda^{-1} = \begin{bmatrix}
I+SE & -S\\[1mm]
-E & I
\end{bmatrix}.
\]

Using the algebraic Riccati equation \eqref{eq3.33}, one checks directly \eqref{eq3.33}, which implies that the Hamiltonian operator
$Ham$ is boundedly invertible on $D(\mathcal{A}_m)\times D(\mathcal{A}_m^*)$.

\end{proof}

\begin{remark}
 The semigroups
 $\{e^{t(\mathcal{A}_m - B B^* \hat{E})}\}_{t\ge0}$ and
 $\{e^{-t(\mathcal{A}_m - B B^* \hat{E})^*}\}_{t\ge0}$ on $H$ are exponentially stable, with generators $\mathcal{A}_m - B B^* \hat{E}$ and $(\mathcal{A}_m - B B^* \hat{E})^*$, respectively. Equivalently, the system \eqref{eq3.3} is exponentially stabilizable and exponentially detectable \cite{ref32}.
\end{remark}
 
Next, we state the following Lemma.

\begin{lemma}\label{le3.12}
There exists a constant $C>0$, independent of the time horizon $T$, such that for the optimal pair $(y_T,p_T)$ one has
\[
\Vert y_T(T)\Vert \leq C,\qquad \Vert p_T(0)\Vert \leq C.
\]
\end{lemma}
\begin{proof}[of Lemma \ref{le3.12}]
For the first estimate, \(\|y_T(T)\|\le C\), we directly invoke the a priori bound given by \eqref{eq3.8}. 

For the second estimate, \(\|p_T(0)\|\le C\), we apply the observability inequality for the adjoint system (see \cite[Proposition 3.1, p. 39]{ref27}), which follows from null‑controllability and ensures the right‑hand bound.
\end{proof}


Using Lemmas \ref{le3.10} and \ref{le3.12}, we are now in a position to prove Theorem \ref{th1.5}.
\begin{proof}[of Theorem \ref{th1.5}]
Consider equation \eqref{eq3.32} and employ the same dichotomy transformation as in the proof of Lemma \ref{le3.10}:
\begin{equation}\label{eq3.37}
\begin{pmatrix}
z(t) \\[1mm]
q(t)
\end{pmatrix}
=\Lambda \begin{pmatrix}
\tilde{y}(t) \\[1mm]
\tilde{p}(t)
\end{pmatrix},\qquad t\in (0,T)\quad \text{a.e.},
\end{equation}
where, as established earlier, \(\Lambda\) is an invertible matrix. Through this transformation, one obtains a decoupled evolution system. After some calculations, this yields the decoupled system
\begin{equation}\label{eq3.38}
\frac{d}{dt}\begin{pmatrix}
z(t) \\[1mm]
q(t)
\end{pmatrix}
=\begin{pmatrix}
\mathcal{A}_m-BB^*\hat{E} & 0\\[1mm]
0 & -\bigl(\mathcal{A}_m-BB^*\hat{E}\bigr)^*
\end{pmatrix}\begin{pmatrix}
z(t) \\[1mm]
q(t)
\end{pmatrix},\qquad t\in (0,T)\quad \text{a.e.}
\end{equation}
Thus, it follows that
\begin{equation}\label{eq3.39}
z(t)=S(t)z(0),\qquad q(t)=S(T-t)^*q(T),\qquad t\in (0,T)\quad \text{a.e.},
\end{equation}
where \(\{S(t)\}_{t\ge0}\) and \(\{S(t)^*\}_{t\ge0}\) are the \(C_0\) semigroups, exponentially stable, generated respectively by \(\mathcal{A}_m-BB^*\hat{E}\) and by \(-\bigl(\mathcal{A}_m-BB^*\hat{E}\bigr)^*\); note that \(\mathcal{A}_m-BB^*\hat{E}\) and its adjoint possess the same analytic and topological properties in the Hilbert space \cite{ref58}. By inverting transformation \eqref{eq3.37}, one obtains
\[
z(0)=\tilde{y}_T(0)+S\tilde{p}_T(0),\qquad q(T)=E\tilde{y}_T(T)+(ES+I)\tilde{p}_T(T).
\]
To complete the proof of Theorem \ref{th1.5}, it suffices to show that the norms \(\Vert\tilde{p}_T(0)\Vert\) and \(\Vert\tilde{y}_T(T)\Vert\) are uniformly bounded, i.e.,
\begin{align}
    \Vert \tilde{p}_T(0)\Vert\leq C \quad \text{and} \quad \Vert\tilde{y}_T(T)\Vert\leq C,
\end{align}
with \(C\) independent of \(T\); this follows from Lemma \ref{le3.12}. Consequently, we deduce that there exist positive constants \(C_1\) and \(C_2\) such that
\begin{align}
    \Vert z(0)\Vert\leq C_1\Vert \tilde{y}_T(0)\Vert,\quad  \Vert q(T)\Vert\leq C_2\Vert \tilde{p}_T(T)\Vert.
\end{align}
Therefore, there exist constants \(C, C'>0\) and \(\lambda>0\) such that
\begin{equation}\label{eq3.43}
\Vert z(t)\Vert\leq C\, \Vert \tilde{y}_T(0)\Vert e^{-\nu t},\quad  \Vert q(t)\Vert\leq C'\, \Vert \tilde{p}_T(T)\Vert e^{-\nu(T-t)},\qquad t\in (0,T)\quad \text{a.e.}
\end{equation}
Using \eqref{eq3.37} in conjunction with \eqref{eq3.43}, we deduce that
\begin{equation}\label{eq3.44}
\Vert y_T(t)-\bar{y}\Vert \leq K_1\, \Vert y_0-\bar{y}\Vert e^{-\nu t},\quad\text{and}\quad \Vert p_T(t)-\bar{p}\Vert \leq K_2\, \Vert p_T(T)-\bar{p}\Vert e^{-\nu (T-t)},\qquad t\in (0,T)\quad \text{a.e.},
\end{equation}
which, in turn, implies that there exists a constant \(K>0\) such that
\begin{equation}\label{eq3.45}
\Vert y_T(t)-\bar{y}\Vert + \Vert v_T(t)-\bar{v}\Vert_{L^2(\omega)} \leq K\left( \Vert y_0-\bar{y}\Vert\, e^{-\nu t} + \Vert p_T(T)-\bar{p}\Vert\, e^{-\nu (T-t)} \right),\qquad t\in (0,T)\quad \text{a.e.}
\end{equation}
\end{proof}

\begin{remark}
 The final estimate \eqref{eq3.45} quantifies the energy the dynamic system must expend to stay near the turnpike $\bar{y}$ while tracking the running target $y_d$. In summary, the turnpike property is a powerful tool: qualitatively, it describes the tendency of optimal trajectories to spend most of their time near the steady state ; quantitatively, it provides explicit bounds on the “energy” (or cost) required to maintain this proximity over long horizons.
\end{remark}
\section{Integral Turnpike Property}\label{s6}
 In this section, we study the behavior of the optimal trajectories $(y_T,v_T)$ for large $T$, in the regime $T\le A$.  Given the equilibrium $(\bar y,\bar v)$, we ask whether $(y_T,v_T)$ remains in a neighborhood of $(\bar y,\bar v)$ over the entire time interval $[0,T]$ when $T\le A$. 
\subsection{Dissipativity}
  We establish a turnpike property by employing the concept of dissipativity introduced by Willems in \cite{ref75}. The following definitions summarize the key notions required for this analysis. Recall that
$$
 \mathcal{K} = \bigl\{\alpha:[0,\infty)\to[0,\infty)\,\bigm|\,\alpha\text{ is continuous, strictly increasing, and }\alpha(0)=0\bigr\}.
$$
\begin{definition}[Strict dissipation inequality]\label{d6.1}
Let \((\bar y,\bar v)\) be an optimal stationary point of \eqref{eq3.5}, and define the supply rate
\[
w(y_s,v)
=\bigl\|y_s - y_d\bigr\|^2
\;+\;\bigl\Vert y_s(T)\bigr\Vert^2
\;-\;\bigl\|\bar y - y_d\bigr\|^2.
\]
We say that \eqref{eq3.1} is \emph{strictly dissipative} at \((\bar y,\bar v)\) with respect to \(w\) if there exist
\begin{itemize}
  \item a storage function \(S\colon H\to\mathbb{R}\) that is locally bounded and bounded below,
  \item and a function \(\alpha\in\mathcal{K}\),
\end{itemize}
such that for every \(T\in(0,A)\) and every \(\tau\in(0,T)\), any solution \((y_s,v_s)\) of \eqref{eq3.3} with \(v_s\equiv\bar v\) satisfies
\begin{align}
    S\bigl(y_s(\tau)\bigr)\;-\;S\bigl(y_s(0)\bigr)
\;\le\;
\int_{0}^{\tau} w\bigl(y_s(t),\,v_s(t)\bigr)\,dt
\;-\;
\int_{0}^{\tau} \alpha\bigl(\|y_s(t)-\bar y\|\bigr)\,dt.
\end{align}
\end{definition}
\begin{proposition}\label{p6.2}
System \eqref{eq3.3} is strictly dissipative in the sense of Definition \ref{d6.1}.
\end{proposition}

\begin{proof}[of Proposition \ref{p6.2}]
Multiply the first equation in \eqref{eq3.9} by \(\bar p\) and integrate by parts over \([0,\tau]\).  This yields
\begin{align}
\langle y_0,\bar p\rangle \;-\;\langle y(\tau),\bar p\rangle
\;-\;\int_0^{\tau}\langle y(t),\,\bar y - y_d\rangle\,dt
&=\int_0^{\tau}\langle v(t),\,\bar p\rangle\,dt
=\int_0^{\tau}\langle \bar v,\,\bar p\rangle\,dt,\\
\langle y_0,\bar p\rangle \;-\;\langle y(\tau),\bar p\rangle
&=\int_0^{\tau}\langle y(t)-\bar y,\,\bar y - y_d\rangle\,dt.
\end{align}
Since
\[
\langle y-\bar y,\,\bar y - y_d\rangle
=\tfrac12\bigl(\|y-y_d\|^2 - \|\bar y - y_d\|^2 - \|y-\bar y\|^2\bigr),
\]
we obtain
\begin{align}
\langle y_0,\bar p\rangle - \langle y(\tau),\bar p\rangle
+ \tfrac12 \int_0^{\tau}\|y(t)-\bar y\|^2\,dt
&= \int_0^{\tau}\tfrac12\bigl(\|y(t)-y_d\|^2 - \|\bar y - y_d\|^2\bigr)\,dt.
\end{align}
Noting that \(\Vert y(\tau)\Vert^2\ge0\), this is exactly the strict dissipation inequality (6.1) with storage function \(S(y)=\langle y,\bar p\rangle\) and
\(\alpha(s)=\tfrac12\,s^2\).  Hence the system is strictly dissipative.
\end{proof}

\begin{remark}
Under the assumption $R < 1$, the operator $\mathcal{A}_m$ generates an exponentially stable semigroup (i.e., when $v \equiv 0$). As a consequence, the system \eqref{eq3.3} is stabilizable. Indeed, from equation \eqref{ee5.29}, we derive the following coupled evolution system:

\begin{align}
    \frac{d}{dt}
\begin{pmatrix}
\tilde{y}(t) \\
\tilde{p}(t)
\end{pmatrix}
=
\begin{pmatrix}
\mathcal{A}_m & 0 \\
0 & -\mathcal{A}_m^*
\end{pmatrix}
\begin{pmatrix}
\tilde{y}(t) \\
\tilde{p}(t)
\end{pmatrix}
+
\begin{pmatrix}
BB^*\bar{p} \\
-N\tilde{y}(t)
\end{pmatrix},
\quad \text{for a.e. } t\in(0,T).
\end{align}

The corresponding solutions can be expressed as:

\begin{align}
    \tilde{y}(t) = e^{t\mathcal{A}_m} \tilde{y}(0) + \int_0^t e^{(t-s)\mathcal{A}_m} BB^*\bar{p} \, ds,
\end{align}

\begin{align}
    \tilde{p}(t) = e^{(T-t)\mathcal{A}_m^*} \tilde{p}(T) + N \int_t^T e^{(s-t)\mathcal{A}_m^*} \tilde{y}(s)\,ds.
\end{align}

We estimate $\tilde{y}(t)$ as:

\begin{align}
    \|\tilde{y}(t)\|
\le M e^{-\nu t} \|\tilde{y}(0)\|
+ \frac{M \|BB^*\bar{p}\|}{\nu} \left(1 - e^{-\nu t} \right).
\end{align}

Hence, there exists a constant $C_1 > 0$ such that:

\begin{align}\label{ee6.9}
    \|\tilde{y}(t)\| \le C_1 \|\tilde{y}(0)\| e^{-\nu t}.
\end{align}

By injecting the estimate \eqref{ee6.9} into the adjoint convolution, we deduce:

\begin{align}
    \|\tilde{p}(t)\|
\le M e^{-\nu(T - t)} \|\tilde{p}(T)\|
+ C_2 \|\tilde{y}(0)\| e^{-\nu t},
\end{align}
for some constant $C_2 > 0.$ We thus conclude that there exists $K > 0$ such that,

\begin{align}
    \|y_T(t) - \bar{y}\| + \|p_T(t) - \bar{p}\|
\le K\left(
\|y_0 - \bar{y}\| e^{-\nu t}
+ \|p_T - \bar{p}\| e^{-\nu(T - t)}
\right),\quad \text{a.e.}\; t \in (0,T).
\end{align}
\end{remark}

Since the system is strictly dissipative, it follows that it is stabilizable, which in turn allows us to establish the integral turnpike property.
  \subsection{Integral turnpike} 

This subsection addresses the integral turnpike property. We establish a weak version of the turnpike and show that the cost‑minimizing control remains close to the static optimal control.

\begin{theorem}\label{t6.4}
Let \(T \le A\). There exists a constant \(M>0\), independent of \(T\), such that the optimal solution \((y_T,v_T)\) of the dynamic control problem and the solution \((\bar y,\bar v)\) of the static control problem satisfy
\begin{align}
          \displaystyle\int_0^T\left(\Vert y_T-\bar{y}\Vert+\Vert v_T-\bar{v}\Vert\right)dt\leq  M.
      \end{align}
\end{theorem}

\begin{proof}[of Theorem \ref{t6.4}]
Define the error variables
\[
\tilde y(t) = y_T(t) - \bar y,
\quad
\tilde p(t) = p_T(t) - \bar p,
\]
with boundary conditions
\[
\tilde y(0)=y_0-\bar y,\quad
\tilde y(T)=y_T(T)-\bar y,
\quad
\tilde p(0)=p_T(0)-\bar p,\quad
\tilde p(T)=p_T(T)-\bar p.
\]
Then the state–adjoint system becomes
\begin{equation}\label{err-sys}
\begin{cases}
\partial_t\tilde y = \mathcal A_m\,\tilde y + B\,\tilde v,\\[4pt]
-\partial_t\tilde p = \mathcal A_m^*\,\tilde p + N\,\tilde y.
\end{cases}
\end{equation}
Multiply the first equation by \(\tilde p\), the second by \(\tilde y\), and subtract to obtain
\[
N\int_0^T \|\tilde y\|^2\,dt
\;+\;
\int_0^T \|\tilde v\|^2\,dt
=
\langle \tilde y(0),\,\tilde p(0)\rangle
\;-\;
\langle \tilde y(T),\,\tilde p(T)\rangle.
\]
By Cauchy–Schwarz and the assumptions
\(\|\tilde y(T)\|^2 \le \epsilon\) and \(\|\tilde p(0)\|^2 \le \epsilon\), it follows that
\begin{align}
    N\int_0^T\Vert y_T-\bar{y}\Vert^2\,dt+\displaystyle\int_0^T\Vert v_T-\bar{v}\Vert^2_{L^2(\omega)}\,dt  \leq M.
\end{align}
for some constant \(M>0\). 
\end{proof}

\begin{remark}
By Theorem \ref{t6.4} and the arguments in \cite{ref54}, one also obtains the standard estimate for the time‑average deviation (turnpike measure) of \((y_T,v_T)\) from \((\bar y,\bar v)\). The realization of the turnpike property in the case $T \leq A$ confirms that the population can be stabilized before the next generation (Theorem \ref{t6.4}).
\end{remark}

\section{Numerical Simulations}\label{s7}
 Let $\{\varphi_k\}_{k\ge0}$ be an orthonormal basis of $L^2(0,L)$ formed by the eigenfunctions of the Neumann Laplacian, and let $\{\lambda_k\}_{k\ge0}$ be the corresponding nondecreasing eigenvalues.  Then for each $k\ge0$, the pair $(\varphi_k,\lambda_k)$ solves

\begin{equation}\label{e7.1}
\left\lbrace\begin{array}{ll}
-\partial_{xx} \varphi_k =\lambda_k\varphi_k & \text{in}\quad (0,L),\\
\\\dfrac{\partial\varphi_k}{\partial\nu} =0 & \text{on}\quad\lbrace 0,L\rbrace.
\end{array}
\right.
\end{equation} 
Under the Neumann boundary conditions, each eigenfunction $\varphi_k$ of \eqref{e7.1} can be written up to normalization as $
\varphi_k(x) = A_k \cos\bigl(\sqrt{\lambda_k}\,x\bigr),
$ with the nonzero eigenvalues given by $
\lambda_k = \frac{k^2\pi^2}{L^2}, 
\quad k=1,2,3,\dots.
$ Projecting the age–structured state $y(x,a,t)$ onto this basis, $
y(x,a,t) = \displaystyle\sum_{k=1}^\infty y_k(a,t)\,\varphi_k(x),
$ yields for each mode $y_k$ the decoupled system

\begin{equation}\label{e7.2}
\left\lbrace 
\begin{array}{ll}
\partial_t y_{k}(a,t)+\partial_a y_{k}(a,t)+(\mu(a) +\lambda_{k}) y_{k}(a,t)=\mathds{1}_{\lbrace 0\rbrace}(a)v_k(a,t), & \text{ in } [0,A]\times[0,T],\\  
\\y_k\left(0,t\right) =\displaystyle\int_0^A\beta_k(a)y_k(a,t)da, & \text{ in }  [0,T], \\  
\\y_{k}\left(a,0\right)=y_{0,k}(a), &  
\text{ in } [0,A].  
\end{array}
\right.
\end{equation}

In our numerical experiments, we consider the following data

$$
\beta_k(a) = 60^5\,a^2\,(A-a)^2\,\exp\left(-3\left(a - \frac{A}{2}\right)^2\right), 
\quad \mu(a) = \frac{1}{50(A-a)}.
$$


For the linear-quadratic optimal control formulation, we consider the dynamic problem:
\begin{align}
\min_{Y,V} \ \frac{1}{2}\Bigl((Y-Y_d)^T(Y-Y_d) + V^T V + Y^T Y\Bigr),
\end{align}
where $Y$ and $V$ represent the state and control vectors, respectively, subject to the dynamics governed by system \eqref{eq3.3}.

For more details on the discretization methods employed, the reader is referred to \cite{ref27}.

\vspace{0.5em}
\noindent
\textbf{Implementation details.} The optimal control problem is solved using Python, within the Anaconda distribution. We utilize the Spyder development environment and the CasADi toolbox, which enables efficient symbolic differentiation and numerical optimization.


\begin{figure}[H]
    \centering
  \includegraphics[width=0.7\textwidth]{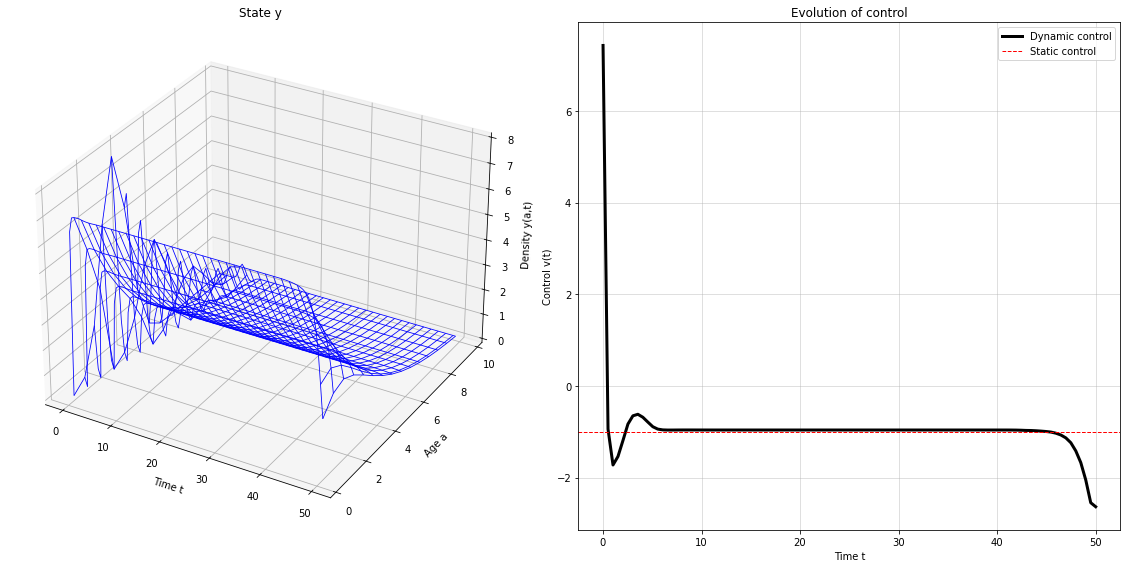}
    \caption{Optimal state and control observed}
    \label{fi1}
\end{figure}



\begin{figure}[H]
    \centering
  \includegraphics[width=0.7\textwidth]{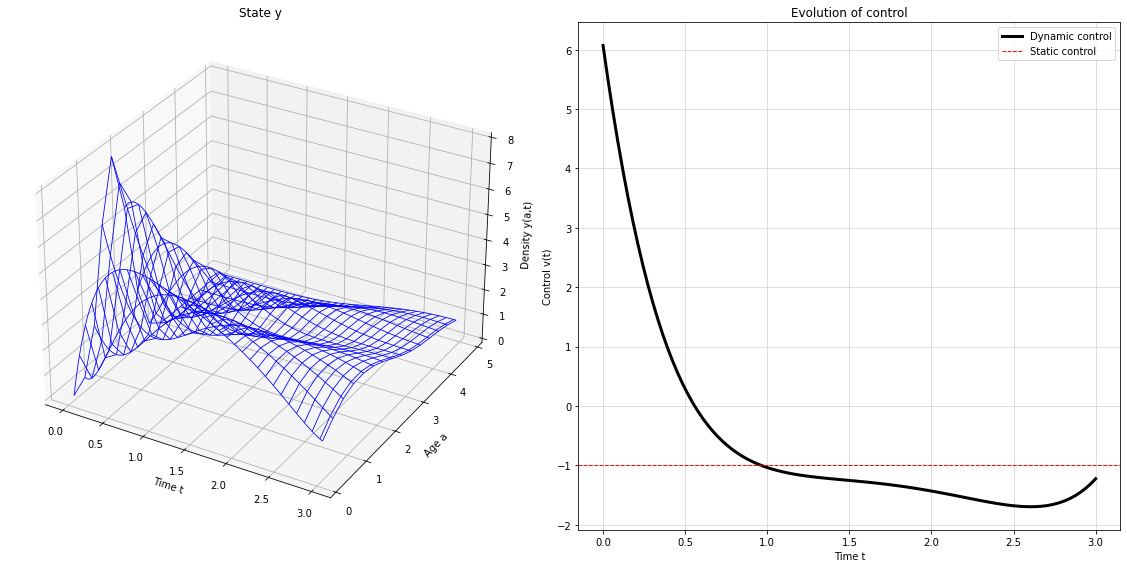}
    \caption{Optimal state and control observed over a short time horizon. The turnpike property fails to hold, confirming that a sufficiently long duration is both necessary and sufficient for its emergence.}
    \label{fi2}
\end{figure}

Despite these variations in data, the phenomenon endures, demonstrating that it is not governed by a single parameter. In fact, beyond the time factor, it is also shaped by the system’s dissipative structure and by the form of the cost function in the optimal control problem.




These numerical results corroborate the turnpike property asserted in Theorems \ref{th1.5}.  The figures above illustrate the temporal evolution of the population with respect to age, under the influence of birth control interventions. These visualizations aim to demonstrate how optimal control strategies can effectively guide the population dynamics toward the desired target state $y_d$, often referred to as the "ideal population." Remarkably, this objective is achieved by applying a control that remains relatively stable throughout the time horizon, with significant adjustments occurring primarily at the initial and final phases. This behavior exemplifies the turnpike phenomenon, where the optimal trajectory stays close to a steady-state regime for most of the control period.





\section{Open Problems and Comments}\label{s8}  
These population dynamics models serve a dual purpose : first, to explain the evolution of a population over time; and second, to provide insights into birth and death processes by addressing key questions such as: When is the extinction of a population almost certain? and How can we achieve an ideal population at minimal cost? To our knowledge, the turnpike property emerges as a powerful tool for addressing the latter.

For instance, in the context of malaria, various global strategies aim to reduce disease transmission by targeting the population of disease vectors. Mosquito control remains a central concern, particularly through genetic modification techniques intended to limit or suppress vector reproduction. One widely used method consists of irradiating insects to induce sterility prior to release, thereby decreasing the reproductive capacity of the target population. If near-certain extinction of such a population is achievable, then identifying an optimal control strategy that drives the population density toward a desired threshold (e.g., below a critical viability level) becomes both economically viable and ecologically sustainable. As a perspective, we propose the following research directions:
\begin{itemize}
\item In system \eqref{e7.2}, the operator $\lambda$ acts as a mortality term, influencing the system without affecting the turnpike property. It would be insightful to investigate the impact of time-dependent mortality on this behavior.
\item Establish a turnpike-type theorem for the birth control model in the context of shape optimization.
\item Analyze the effect of constraining the final state $y_T(T)$ in the optimal control formulation. In particular, study how minimizing the discrepancy between $y_T(T)$ and the target state $y_d$ influences system dynamics.
\item Extend the theoretical framework, especially null-controllability and the turnpike property, to nonlinear systems, an area that remains largely unexplored.
\item Go beyond simulation by utilizing numerical results for population forecasting. This would enable the design of control strategies with robust long-term predictive power.
\item Incorporate real-time data to refine control strategies in dynamic optimization settings. Machine learning techniques could be integrated to anticipate intervention needs and enhance cost-effectiveness.
\item Improve model adaptability to account for unexpected events (e.g., political instability, environmental shifts, or public health emergencies) that may affect population trajectories. Scenario-based simulations can guide real-time adjustments to control policies.
\item In the present study, the turnpike property was established using a feedback control whose positivity is not guaranteed over the interval $(0,T)$. Imposing a positivity constraint (interpreted as a reduction in mortality or increase in life expectancy) could yield a novel and significant extension of the turnpike result.
\item Explore the turnpike phenomenon in predator-prey systems, particularly within multi-species trophic networks as discussed in \cite{ref74}, to examine mechanisms of population persistence.
\end{itemize}

These research directions outline both theoretical challenges and practical opportunities, offering promising avenues for advancing population dynamics modeling and optimal control theory.
\appendix
\section{Proof of Proposition \ref{th1.3}}\label{annexe:A}
The Proof of Proposition \ref{th1.3} is based on Definition \ref{d3.2}. Thus, we examine this property in the cases where \(T < A - a_0\) and \(T > A - a_0\).
\begin{proposition}\label{pr2.2}
The system \eqref{eq1.5} is not null-controllable for any  \(T < A-a_0\).
\end{proposition}

\begin{proof}[of Proposition \ref{pr2.2}]
Let \(T \in (0, A-a_0)\) and consider \(a\) satisfying \(T+a_0 < a \leq A\). On one hand, for all \(v \in L^2(0,T;H)\), we have
\[
(\Phi_T v)(a)= \int_{a-T}^{a}\frac{\pi(a)}{\pi(s)}e^{(a-s)\triangle}\mathds{1}_{\{\omega\times[0,a_0]\}}(x,s)v(x,s,s-a+T)\,ds = 0,
\]
since \(\mathds{1}_{\{\omega\times[0,a_0]\}}(x,s)=0\) whenever \(s > a_0\), which holds because \(a_0 < a-T \leq s\).

On the other hand, the evolution of the initial state is given by
\[
(\mathcal{T}_T y_0)(a)= \frac{\pi(a)}{\pi(a-T)}e^{T\triangle}y_{0}(x,a-T), \quad \forall\, a \in (T+a_0, A].
\]
Hence, if \(y_0(x,a) \neq 0\) for all \(a \in (a_0, A-T]\), then it is impossible for \((\mathcal{T}_T y_0)(a)\) to vanish. Consequently, we conclude that
\[
\operatorname{Ran} (\mathcal{T}_T) \not\subset \operatorname{Ran}(\Phi_T),
\]
implying that system \eqref{eq1.5} is not null-controllable when \(T < A-a_0\).
\end{proof}
\begin{proposition}\label{pr2.3}  
Let us define the admissible set as follows:  
\begin{equation*}
\mathcal{R}=\left\lbrace \bar{y}\in H \mid \dfrac{ e^{-a\triangle}\bar{y}(x,a)}{\pi(a)}\in H ,\quad \dfrac{\dfrac{e^{-a\triangle}\bar{y}(x,a)}{\pi(a) }-\displaystyle\int_{0}^{A-a}\beta(x,z)\dfrac{\pi(z)}{\pi(z+a)}e^{-a\triangle}\bar{y}(x,z+a)dz}{\displaystyle\int_{0}^{a}\frac{e^{-z\triangle}\mathds{1}_{\lbrace\omega\times[0,a_0]\rbrace}(x,z)}{\pi(z) }dz}\in H \right\rbrace.
\end{equation*}  
Under the assumptions of Proposition \ref{th1.3}, we have that \(\mathcal{R}\subset \operatorname{Ran}(\Phi_T)\) for every \(T > A - a_0\).
\end{proposition}
\begin{proof}[of Proposition \ref{pr2.3}]
Before determining the admissible space, we first investigate the existence of a control. We assume that \( a_0 \leq a_b \) and consider \( T \in \;] A-a_0, A] \). Our objective is to find a control function \( u(x,t-a) \in L^2 (\Omega\times [-A,T]) \) satisfying  
\begin{equation}\label{eq2.8}
(\Phi_T v)(a) =\bar{y}(x,a) \qquad \forall \;a\in [0,A],\quad x\in \Omega.
\end{equation}  

{\bf \underline{Case 1 : \( -A < t-a < T-A \)}}
We first set \( u(x,t-a) = 0 \) and, using equation \eqref{eq2.6}, we obtain  
\begin{equation}\label{eq2.9}
(\Phi_T v)(a) = 0 \qquad \forall \;t-a<T-A.
\end{equation}  

{\bf \underline{Case 2: \( T-A \leq t-a \leq 0 \)}}  
By employing equation \eqref{eq2.6}, we derive  
\begin{equation}
(\Phi_T v)(a) = u(x,T-a) \int_{a-T}^{a} \frac{\pi(a)}{\pi(s)} e^{(a-s)\triangle} \mathds{1}_{\lbrace\omega\times[0,a_0]\rbrace}(x,s) ds, \qquad \forall a\in [T,A],\quad \forall x\in \Omega.
\end{equation}  

To ensure that equation \eqref{eq2.8} holds, we must solve  
\begin{align*}
u(x,T-a) \int_{a-T}^{a} \frac{\pi(a)}{\pi(s)} e^{(a-s)\triangle} \mathds{1}_{\lbrace\omega\times[0,a_0]\rbrace}(x,s) ds &= \bar{y}(x,a), \qquad \forall a\in [T,A], \\
u(x,T-a) &= \dfrac{e^{-a\triangle} \bar{y}(x,a) / \pi(a)}{\displaystyle\int_{a-T}^{a} \frac{e^{-s\triangle} \mathds{1}_{\lbrace\omega\times[0,a_0]\rbrace}(x,s)}{\pi(s)} ds}.
\end{align*}  

By setting \( T-a = s \), we rewrite the control as  
\begin{equation}\label{eq2.11}
u(x,s) = \dfrac{e^{(s-T)\triangle} \bar{y}(x,T-s) / \pi(T-s)}{\displaystyle\int_{-s}^{T-s} \frac{e^{-z\triangle} \mathds{1}_{\lbrace\omega\times[0,a_0]\rbrace}(x,z)}{\pi(z)} dz}, \qquad \forall s\in [T-A,0],\quad \forall x\in \Omega.
\end{equation}  

Since we have \( -a_0 < T-A \leq s \), it follows that \( -s < a_0 \), which ensures that  
\begin{equation*}
\int_{-s}^{T-s} \frac{e^{-z\triangle} \mathds{1}_{\lbrace\omega\times[0,a_0]\rbrace}(x,z)}{\pi(z)} dz > 0.
\end{equation*}  
 \underline{\bf{Case 3 : $0\leq t-a\leq  a_b$}} 
Using \eqref{eq2.6} and leveraging \eqref{eq2.11}, we obtain  
\begin{align}\label{eq2.12}
(\Phi_t v)(a) = \pi(a) \dfrac{\displaystyle\int_{a-t}^{a} \frac{e^{-z\triangle} \mathds{1}_{\lbrace\omega\times[0,a_0]\rbrace}(x,z)}{\pi(z)} dz}{\displaystyle\int_{a-t}^{T+(a-t)} \frac{e^{-z\triangle} \mathds{1}_{\lbrace\omega\times[0,a_0]\rbrace}(x,z)}{\pi(z)} dz} \dfrac{e^{(t-T)\triangle} \bar{y}(x,T+(a-t))}{\pi(T+(a-t))}.
\end{align}  

From equation \eqref{eq2.6}, we have  
\begin{equation}\label{eq2.13}
(\Phi_T v)(a) = \pi(a)e^{a\triangle}b(x,T-a) + u(x,T-a) \int_{0}^{a} \frac{\pi(a)}{\pi(s)} e^{(a-s)\triangle} \mathds{1}_{\lbrace\omega\times[0,a_0]\rbrace}(x,s) ds, \quad a\in [ T-a_b, T],
\end{equation}  
where  
\begin{equation*}
b(x,t) = \int_{a_b}^{A} \beta(x,a)(\Phi_t u)(a) da, \qquad t\in [0,a_b].
\end{equation*}  
Decomposing \( b(x,t) \), we write  
\begin{align*}
b(x,t) &= \int_{a_b}^{t+A-T} \beta(x,a)(\Phi_t u)(a) da + \int_{t+A-T}^{A} \beta(x,a)(\Phi_t u)(a) da, \\
b(x,t) &= \int_{a_b}^{t+A-T} \beta(x,a)(\Phi_t u)(a) da.
\end{align*}  
For \( a \leq t+A-T \), from \eqref{eq2.12}, we obtain  
\begin{equation}
b(x,t) = \int_{a_b}^{t+A-T} \beta(x,a)  \pi(a) \dfrac{\displaystyle\int_{a-t}^{a} \frac{e^{-z\triangle} \mathds{1}_{\lbrace\omega\times[0,a_0]\rbrace}(x,z)}{\pi(z)} dz}{\displaystyle\int_{a-t}^{T+(a-t)} \frac{e^{-z\triangle} \mathds{1}_{\lbrace\omega\times[0,a_0]\rbrace}(x,z)}{\pi(z)} dz} \dfrac{e^{(t-T)\triangle} \bar{y}(x,T+(a-t))}{\pi(T+(a-t))} da.
\end{equation}  
Moreover,  
\begin{equation*}
\int_{a-t}^{T+(a-t)} \frac{e^{-z\triangle} \mathds{1}_{\lbrace\omega\times[0,a_0]\rbrace}(x,z)}{\pi(z)} dz = \int_{a-t}^{a} \frac{e^{-z\triangle} \mathds{1}_{\lbrace\omega\times[0,a_0]\rbrace}(x,z)}{\pi(z)} dz.
\end{equation*}  
Thus,  
\begin{equation*}
b(x,t) = \int_{a_b}^{t+A-T} \beta(x,a) \pi(a) \dfrac{e^{(t-T)\triangle} \bar{y}(x,T+(a-t))}{\pi(T+(a-t))} da, \qquad t\in [0,a_b], \quad \forall x\in \Omega,
\end{equation*}  
which implies  
\begin{equation*}
b(x,T-a) = \int_{a_b}^{A-a} \beta(x,z)  \pi(z) \dfrac{e^{-a\triangle} \bar{y}(x,z+a)}{\pi(z+a)} dz, \qquad a\in[T-a_b, T].
\end{equation*}  

Substituting this result into \eqref{eq2.13}, we obtain  
\begin{equation}
(\Phi_T v)(a) = \pi(a) \int_{a_b}^{A-a} \beta(x,z)  \pi(z) \dfrac{\bar{y}(x,z+a)}{\pi(z+a)} dz + u(x,T-a) \pi(a) \int_{0}^{a} \frac{e^{(a-s)\triangle} \mathds{1}_{\lbrace\omega\times[0,a_0]\rbrace}(x,s)}{\pi(s)} ds, \quad a\in [T-a_b, T].
\end{equation}  

To satisfy \eqref{eq2.8}, we impose  
\begin{align*}
 \pi(a) \int_{a_b}^{A-a} \beta(x,z)  \pi(z) \dfrac{\bar{y}(x,z+a)}{\pi(z+a)} dz + u(x,T-a) \pi(a) \int_{0}^{a} \frac{e^{(a-s)\triangle} \mathds{1}_{\lbrace\omega\times[0,a_0]\rbrace}(x,s)}{\pi(s)} ds = \bar{y}(x,a).
\end{align*}  
This leads to the control function  
\begin{align}
    u(x,T-a) = \dfrac{e^{-a\triangle} \dfrac{\bar{y}(x,a)}{\pi(a)} - \displaystyle\int_{a_b}^{A-a} \beta(x,z)  \pi(z) \dfrac{e^{-a\triangle} \bar{y}(x,z+a)}{\pi(z+a)} dz}{\displaystyle\int_{0}^{a} \frac{e^{-s\triangle} \mathds{1}_{\lbrace\omega\times[0,a_0]\rbrace}(x,s)}{\pi(s)} ds}, \quad a\in [T-a_b,T],
\end{align}  
where  
\begin{equation*}
\int_{0}^{a} \frac{e^{-s\triangle} \mathds{1}_{\lbrace\omega\times[0,a_0]\rbrace}(x,s)}{\pi(s)} ds > 0.
\end{equation*}  

By setting \( s = T-a \), we further refine  
\begin{align*}
u(x,s) = \dfrac{\dfrac{e^{(s-T)\triangle} \bar{y}(x,T-s)}{\pi(T-s)} - \displaystyle\int_{a_b}^{A-T+s} \beta(x,z)  \pi(z) \dfrac{e^{(s-T)\triangle} \bar{y}(x,T-(s-z))}{\pi(T-(s-z))} dz}{\displaystyle\int_{0}^{T-s} \frac{e^{-z\triangle} \mathds{1}_{\lbrace\omega\times[0,a_0]\rbrace}(x,z)}{\pi(z)} dz}, \quad s\in[0,a_b].
\end{align*}  

Defining  
\begin{equation}
\chi(x,s) = \int_{a_b}^{A-T+s} \beta(x,z)  \pi(z) \dfrac{e^{(s-T)\triangle} \bar{y}(x,T-(s-z))}{\pi(T-(s-z))} dz, \quad s\in[0,a_b],
\end{equation}  
we obtain the final expression  
\begin{align}\label{eq2.17}
    u(x,s) = \dfrac{e^{(s-T)\triangle} \bar{y}(x,T-s) - \pi(T-s) \chi(x,s)}{\pi(T-s) \displaystyle\int_{0}^{T-s} \frac{e^{-z\triangle} \mathds{1}_{\lbrace\omega\times[0,a_0]\rbrace}(x,z)}{\pi(z)} dz}, \quad s\in[0,a_b].
\end{align}  
\underline{\bf{Case 4 : $a_b\leq t-a\leq T$}}  
In this section, we will prove that  
\begin{align}
b(x,t) =\chi (x,t), \qquad t\in [ja_b, (j+1)a_b],
\end{align}  
and that \eqref{eq2.8} holds for all \( a \in [T-(j+1)a_b, T-ja_b] \) with \( j \geq 0 \).  

Indeed, from \eqref{eq2.17}, we have  
\begin{align*}
(\Phi_t v)(a) &= \pi(a)e^{a\triangle}b(x,t-a) + u(x,t-a) \int_{0}^{a} \frac{\pi(a)}{\pi(s)} e^{(a-s)\triangle} \mathds{1}_{\lbrace\omega\times[0,a_0]\rbrace}(x,s) ds, \\
(\Phi_t v)(a) &= \pi(a)e^{a\triangle}b(x,t-a) + \frac{e^{((t-a)-T)\triangle} \bar{y}(x,T-(t-a)) - \pi(T-(t-a)) \chi(x,t-a)}{\pi(T-(t-a)) \int_{0}^{T-(t-a)} \frac{e^{-z\triangle} \mathds{1}_{\lbrace\omega\times[0,a_0]\rbrace}(x,z)}{\pi(z)} dz} \\
&\quad \times \int_{0}^{a} \frac{\pi(a)}{\pi(z)} e^{(a-z)\triangle} \mathds{1}_{\lbrace\omega\times[0,a_0]\rbrace}(x,z) dz, \\
(\Phi_t v)(a) &= \pi(a)e^{a\triangle}b(x,t-a) \left[1 - \frac{\int_{0}^{a} \frac{e^{-z\triangle} m(x,z)}{\pi(z)} dz}{\int_{0}^{T-(t-a)} \frac{e^{-z\triangle} \mathds{1}_{\lbrace\omega\times[0,a_0]\rbrace}(x,z)}{\pi(z)} dz} \right] \\
&\quad + \pi(a) \frac{\int_{0}^{a} \frac{e^{-z\triangle} m(x,z)}{\pi(z)} dz}{\int_{0}^{T-(t-a)} \frac{e^{-z\triangle} \mathds{1}_{\lbrace\omega\times[0,a_0]\rbrace}(x,z)}{\pi(z)} dz} \times \frac{e^{(t-T)\triangle} \bar{y}(x,T-(t-a))}{\pi(T-(t-a))}.
\end{align*}  

To compute \( b(x,t) \) over the interval \([(j+1)a_b,(j+2)a_b]\), we use the fact that  
\begin{align*}
b(x,t) = \int_{a_b}^{t+A-T} \beta(x,a) (\Phi_t u)(a) da.
\end{align*}  
Thus, we decompose it as follows:  
\begin{align}
b(x,t) = \int_{a_b}^{t} \beta(x,a) (\Phi_t u)(a) da + \int_{t}^{t+A-T} \beta(x,a) (\Phi_t u)(a) da.
\end{align}  
After computation, we obtain  
\begin{align}
b(x,t) = \int_{a_b}^{A-T+t} \beta(x,z) \pi(z) \frac{e^{(t-T)\triangle} \bar{y}(x,T-(t-z))}{\pi(T-(t-z))} dz, \qquad t\in[(j+1)a_b,(j+2)a_b].
\end{align}  

In summary, the control function is given by  
\begin{equation}
u(x,s) =
\begin{cases}
0, & s \in [-A,T-A], \quad \forall x \in \Omega, \\
\frac{\frac{e^{(s-T)\triangle} \bar{y}(x,T-s)}{\pi(T-s)}}{\int_{-s}^{T-s} \frac{e^{-z\triangle} \mathds{1}_{\lbrace\omega\times[0,a_0]\rbrace}(x,z)}{\pi(z)} dz}, & s \in [T-A,0], \quad \forall x \in \Omega, \\
\frac{e^{(s-T)\triangle} \bar{y}(x,T-s) - \pi(T-s) \chi(x,s)}{\pi(T-s) \int_{0}^{T-s} \frac{e^{-z\triangle} \mathds{1}_{\lbrace\omega\times[0,a_0]\rbrace}(x,z)}{\pi(z)} dz}, & s \in [0,T], \quad \forall x \in \Omega,
\end{cases}
\end{equation}  
satisfying \eqref{eq2.8}, where  
\begin{equation}
\chi(x,t) = \int_{a_b}^{A-T+t} \beta(x,z) \pi(z) \frac{e^{(t-T)\triangle} \bar{y}(x,T-(t-z))}{\pi(T-(t-z))} dz, \qquad t \in [0,T], \quad \forall x \in \Omega.
\end{equation}  

Consequently, we have  
\begin{equation}
v(x,a,t) =
\begin{cases}
0, & t-a<T-A, \quad \forall x \in \Omega, \\
\frac{\frac{e^{((t-a)-T)\triangle} \bar{y}(x,T-(t-a))}{\pi(T-(t-a))}}{\int_{a-t}^{T-(t-a)} \frac{e^{-z\triangle} \mathds{1}_{\lbrace\omega\times[0,a_0]\rbrace}(x,z)}{\pi(z)} dz}, & T-A \leq t-a \leq 0, \quad \forall x \in \Omega, \\
\frac{e^{((t-a)-T)\triangle} \bar{y}(x,T-(t-a)) - \pi(T-(t-a)) \chi(x,t-a)}{\pi(T-(t-a)) \int_{0}^{T-(t-a)} \frac{e^{-z\triangle} \mathds{1}_{\lbrace\omega\times[0,a_0]\rbrace}(x,z)}{\pi(z)} dz}, & 0 < t-a \leq T, \quad \forall x \in \Omega.
\end{cases}
\end{equation}  

Finally, the operator \(\Phi_t v\) satisfies  
\begin{equation}
(\Phi_t v)(a) =
\begin{cases}
0, & t-a<T-A, \quad \forall x \in \Omega, \\
\pi(a) \frac{\int_{a-t}^{a} \frac{e^{-z\triangle} \mathds{1}_{\lbrace\omega\times[0,a_0]\rbrace}(x,z)}{\pi(z)} dz}{\int_{a-t}^{T+(a-t)} \frac{e^{-z\triangle} \mathds{1}_{\lbrace\omega\times[0,a_0]\rbrace}(x,z)}{\pi(z)} dz} \frac{e^{(t-T)\triangle} \bar{y}(x,T+(a-t))}{\pi(T+(a-t))}, & T-A \leq t-a \leq 0, \quad \forall x \in \Omega, \\
\pi(a)e^{a\triangle}b(x,t-a) \left[1 - \frac{\int_{0}^{a} \frac{e^{-z\triangle} \mathds{1}_{\lbrace\omega\times[0,a_0]\rbrace}(x,z)}{\pi(z)} dz}{\int_{0}^{T-(t-a)} \frac{e^{-z\triangle} \mathds{1}_{\lbrace\omega\times[0,a_0]\rbrace}(x,z)}{\pi(z)} dz} \right] \\
+ \pi(a) \frac{\int_{0}^{a} \frac{e^{-z\triangle} \mathds{1}_{\lbrace\omega\times[0,a_0]\rbrace}(x,z)}{\pi(z)} dz}{\int_{0}^{T-(t-a)} \frac{e^{-z\triangle} \mathds{1}_{\lbrace\omega\times[0,a_0]\rbrace}(x,z)}{\pi(z)} dz} \frac{e^{(t-T)\triangle} \bar{y}(x,T-(t-a))}{\pi(T-(t-a))}, & 0 < t-a \leq T, \quad \forall x \in \Omega.
\end{cases}
\end{equation}  
\end{proof}
\begin{remark}\label{re2.4}
Let \(\varphi \in D(\mathcal{A}_m)\) be such that \(\frac{\varphi}{\pi} \in H\). We aim to show that \(\varphi \in \mathcal{R}\). For every \(\varphi \in D(\mathcal{A}_m)\), there exists \(f \in H\) such that
\[
\varphi' + \mu \varphi - \Delta \varphi = f.
\]
Then, one can write
\[
\frac{e^{-a\Delta}\varphi(x,a)}{\pi(a)} = \varphi(x,0) + \int_0^a \left( \frac{e^{-s\Delta}\varphi(x,s)}{\pi(s)} \right)' ds,
\]
which, by applying the quotient rule, becomes
\begin{equation}\label{eq2.25}
\frac{e^{-a\Delta}\varphi(x,a)}{\pi(a)} = \varphi(x,0) + \int_0^a \frac{\left(e^{-s\Delta}\varphi(x,s)\right)' \pi(s) - \pi'(s) e^{-s\Delta}\varphi(x,s)}{\pi(s)^2} ds.
\end{equation}
Observing that
\[
\left(e^{-s\Delta}\varphi(x,s)\right)' \pi(s) = -\Delta \,e^{-s\Delta}\varphi(x,s) \,\pi(s) + e^{-s\Delta}\varphi'(x,s)\,\pi(s),
\]
and substituting into \eqref{eq2.25}, we obtain
\[
\frac{e^{-a\Delta}\varphi(x,a)}{\pi(a)} = \varphi(x,0) + \int_0^a \frac{\pi(s)e^{-s\Delta}\varphi'(x,s) + \mu(s)\pi(s)e^{-s\Delta}\varphi(x,s) - \pi(s)\Delta\,e^{-s\Delta}\varphi(x,s)}{\pi(s)^2} ds.
\]
Hence, one deduces that
\[
\frac{e^{-a\Delta}\varphi(x,a)}{\pi(a)} = \int_0^A \beta(x,a)\varphi(x,a)da + \int_0^a \frac{e^{-s\Delta}\left(\varphi'(x,s) + \mu(s)\varphi(x,s) - \Delta\varphi(x,s)\right)}{\pi(s)} ds,
\]
or equivalently,
\begin{equation}\label{eq2.26}
\frac{e^{-a\Delta}\varphi(x,a)}{\pi(a)} = \int_0^A \beta(x,a)\varphi(x,a)da + \int_0^a \frac{e^{-s\Delta}f(x,s)}{\pi(s)} ds.
\end{equation}

Define the functions
\[
I(x)=\int_0^A \beta(x,a)\varphi(x,a)da, \qquad i(x,a)=\int_0^{A-a} \beta(x,z)\frac{\pi(z)}{\pi(z+a)}e^{-a\Delta}\varphi(x,z+a)dz.
\]
Then, from \eqref{eq2.26}, we deduce that
\begin{align}\label{eq2.28}
    \frac{\frac{e^{-a\Delta}\varphi(x,a)}{\pi(a)} - \displaystyle\int_0^{A-a}\beta(x,z)\frac{\pi(z)}{\pi(z+a)}e^{-a\Delta}\varphi(x,z+a)dz}{\displaystyle\int_0^a \frac{e^{-z\Delta}\mathds{1}_{\{\omega\times[0,a_0]\}}(x,z)}{\pi(z)}dz} = \frac{\displaystyle\int_0^a \frac{e^{-s\Delta}f(x,s)}{\pi(s)}ds + I(x)- i(x,a)}{\displaystyle\int_0^a \frac{e^{-z\Delta}\mathds{1}_{\{\omega\times[0,a_0]\}}(x,z)}{\pi(z)}dz}.
\end{align}
Our objective is to show that the expression on the left-hand side of \eqref{eq2.28} belongs to \(H\); that is, \(\varphi \in \mathcal{R}\).

On the one hand, since \(\frac{e^{-s\Delta}f(x,s)}{\pi(s)} \in H\), it follows from the Hardy inequality \cite{ref43} that
\[
\frac{1}{a}\int_0^a \frac{e^{-s\Delta}f(x,s)}{\pi(s)}ds \in H.
\]
On the other hand, we need to establish that
\[
\frac{I(x)-i(x,a)}{a} \in H.
\]
To this end, we compute the derivative of \( I(x)-i(x,a) \) with respect to \(x\). Differentiating under the integral sign (assuming sufficient regularity of \(\beta(x,a)\) and \(\varphi(x,a)\)), we have
\[
\frac{\partial I(x)}{\partial x} = \int_0^A \frac{\partial}{\partial x}\bigl(\beta(x,a)\varphi(x,a)\bigr) da = \int_0^A \left(\frac{\partial \beta(x,a)}{\partial x}\varphi(x,a) + \beta(x,a)\frac{\partial \varphi(x,a)}{\partial x}\right) da.
\]
Similarly, for \( i(x,a) \) we obtain
\[
\frac{\partial i(x,a)}{\partial x} = \int_0^{A-a} \frac{\partial}{\partial x}\left(\beta(x,z)\frac{\pi(z)}{\pi(z+a)}e^{-a\Delta}\varphi(x,z+a)\right) dz.
\]
By applying Leibniz's rule to differentiate the product, we deduce that
\begin{align*}
    \frac{\partial}{\partial x} \left(I(x) - i(x, a)\right) = \int_{0}^{A} \left(\frac{\partial \beta(x, a)}{\partial x} \varphi(x, a) + \beta(x, a) \frac{\partial \varphi(x, a)}{\partial x}\right) \, da 
- \int_{0}^{A-a} \frac{\partial \beta(x, z)}{\partial x} \frac{\pi(z)}{\pi(z + a)} e^{-a\Delta} \varphi(x, z + a) dz\\
- \int_{0}^{A-a} \beta(x, z) \frac{\pi(z)}{\pi(z + a)} e^{-a\Delta} \frac{\partial \varphi(x, z + a)}{\partial x} \, dz.
\end{align*}
Given that \(\beta(x,a) \in C^1(\Omega \times (0,A))\), \(\varphi \in L^2(0,A; H^2(\Omega))\), and that the function \(\pi(\cdot)\) is bounded, it follows that \(\frac{\partial}{\partial x}\bigl(I(x)-i(x,a)\bigr) \in H\). Consequently, by the Hardy inequality, we have 
\[
\frac{I(x)-i(x,a)}{a} \in H.
\]
Thus, the left-hand side of \eqref{eq2.28} belongs to \(H\), which implies that \(\varphi \in \mathcal{R}\).
\end{remark}
The following proposition is a result of null-controllability.
 \begin{proposition}\label{pr2.5}
 Under the assumptions of Proposition \ref{th1.3}, for any $T>A-a_0$ we have 
 \begin{align}\label{eq2.31}
     \operatorname{Ran}( \mathcal{T}_T)\subset\mathcal{R}.
 \end{align} 
 \end{proposition}
 \begin{proof}[of Proposition \ref{pr2.5}]
We now demonstrate that \eqref{eq2.31} is satisfied, thereby allowing us to conclude null-controllability. Let \(\bar{y} \in \operatorname{Ran}(\mathcal{T}_T)\). Then, according to \eqref{eq2.4}, there exists \(y_0 \in H\) such that  
 \begin{equation}
\bar{y}(x,a)=\left\lbrace
\begin{array}{ll}
\frac{\pi(a)}{\pi(a-T)}e^{T\triangle}y_0(x,a-T)\quad & T\leq a, \quad\forall x\in\Omega,\\
\\\pi(a)e^{a\triangle}b(x,T-a)\quad & T>a\quad\forall x\in\Omega.
\end{array}
\right.
\end{equation}
where  
\[
b(x,t)=\int_0^A \beta(x,a)\, (\mathcal{T}_t y_0)(x,a)\, da.
\]
Thus, for every \(x\in \Omega\) we have  
\[
\frac{e^{-a\Delta}\bar{y}(x,a)}{\pi(a)} = \frac{e^{(T-a)\Delta}y_0(x,a-T)}{\pi(a-T)} \quad \text{for all } a \in [T,A],
\]
and  
\[
\frac{e^{-a\Delta}\bar{y}(x,a)}{\pi(a)} = b(x,T-a) \quad \text{for all } a \in [0,T].
\]
Hence, \(\frac{e^{-a\Delta}\bar{y}(x,a)}{\pi(a)}\in H\).

It remains to verify that  
\begin{equation}\label{eq2.33}
\frac{\frac{e^{-a\Delta}\bar{y}(x,a)}{\pi(a)} - \displaystyle\int_{0}^{A-a} \beta(x,z) \frac{\pi(z)}{\pi(z+a)}\, e^{-a\Delta}\bar{y}(x,z+a)dz}{\displaystyle\int_{0}^{a}\frac{e^{-s\Delta}\mathds{1}_{\{\omega\times[0,a_0]\}}(x,s)}{\pi(s)} ds} \in H.
\end{equation}
This is the desired condition that ensures \(\bar{y} \in \mathcal{R}\).

\underline{\bf Case 1} : For \(a \in [\epsilon, A]\) (with \(\epsilon > 0\)), consider the expression
\begin{align}
    \frac{e^{-a\Delta}\bar{y}(x,a)}{\pi(a)} - \int_{0}^{A-a}\beta(x,z)\frac{\pi(z)}{\pi(z+a)}\,e^{-a\Delta}\bar{y}(x,z+a)\,dz.
\end{align}
This expression can be equivalently rewritten as
\[
\frac{e^{-a\Delta}\bar{y}(x,a)}{\pi(a)} - \int_{0}^{A-a}\beta(x,z)\pi(z)\,\frac{e^{-a\Delta}\bar{y}(x,z+a)}{\pi(z+a)}\,dz.
\]
Noting that for \(a\) in the appropriate range we have
\[
\frac{e^{-a\Delta}\bar{y}(x,a)}{\pi(a)} = \frac{e^{(T-a)\Delta}y_0(x,a-T)}{\pi(a-T)},
\]
it follows that the expression becomes
\[
\frac{e^{(T-a)\Delta}y_0(x,a-T)}{\pi(a-T)} - \int_{0}^{A-a}\beta(x,z)\pi(z)\,\frac{e^{-a\Delta}\bar{y}(x,z+a)}{\pi(z+a)}\,dz.
\]
By decomposing the integral over the intervals \([0,\,T-a]\) and \([T-a,\,A-a]\), we obtain
\[
\frac{e^{(T-a)\Delta}y_0(x,a-T)}{\pi(a-T)} - \int_{0}^{T-a}\beta(x,z)\pi(z)\,b(x,T-(z+a))\,dz - \int_{T-a}^{A-a}\beta(x,z)\pi(z)\,\frac{e^{(T-a)\Delta}y_0(x,z+a-T)}{\pi(z+a-T)}\,dz.
\]
After further manipulations, this expression can be recast as
\[
\frac{e^{(T-a)\Delta}y_0(x,a-T)}{\pi(a-T)} + \int_{A-a}^{A}\beta(x,z)\pi(z)\,\frac{e^{(T-a)\Delta}y_0(x,z+a-T)}{\pi(z+a-T)}\,dz - b(x,T-a),
\]
which ultimately simplifies to
\[
\frac{e^{-a\Delta}\bar{y}(x,a)}{\pi(a)} - \int_{0}^{A-a}\beta(x,z)\frac{\pi(z)}{\pi(z+a)}\,e^{-a\Delta}\bar{y}(x,z+a)\,dz.=\int_{A-a}^{A}\beta(x,z)\pi(z)\,\frac{e^{(T-a)\Delta}y_0(x,z+a-T)}{\pi(z+a-T)}\,dz.
\]
Hence, the expression in \eqref{eq2.33} belongs to \(L^2(\epsilon, A; L^2(\Omega))\).

\underline{Case 2} : For \(a \in [0,\epsilon]\) (with \(\epsilon>0\)), we consider  
\begin{align}
    \frac{e^{-a\Delta}\bar{y}(x,a)}{\pi(a)} - \int_{0}^{A-a} \beta(x,z)\pi(z)\frac{e^{-a\Delta}\bar{y}(x,z+a)}{\pi(z+a)}dz.
\end{align}
This expression can be rewritten as  
\[
b(x,T-a) - \int_{0}^{T-a} \beta(x,z)\pi(z)\frac{e^{-a\Delta}\bar{y}(x,z+a)}{\pi(z+a)}dz - \int_{T-a}^{A-a} \beta(x,z)\pi(z)\frac{e^{-a\Delta}\bar{y}(x,z+a)}{\pi(z+a)}dz.
\]
Since  
\[
b(x,T-a) = \int_0^A \beta(x,a) y(x,a,T-a)\,da,
\]
we can express the above as  
\[
b(x,T-a) - \int_{0}^{T-a} \beta(x,z)\pi(z)\,b(x,T-(z+a))dz - \int_{T-a}^{A-a} \beta(x,z)\pi(z)\frac{e^{-a\Delta}\bar{y}(x,z+a)}{\pi(z+a)}dz.
\]
Thus,  
\begin{align}\label{eq2.36}
    \frac{e^{-a\Delta}\bar{y}(x,a)}{\pi(a)} - \int_{0}^{A-a}\beta(x,z)\pi(z)\frac{e^{-a\Delta}\bar{y}(x,z+a)}{\pi(z+a)}dz 
= b(x,T-a) - \int_{0}^{T-a} \beta(x,z)\pi(z)\,b(x,T-(z+a))dz 
\end{align}
\begin{align*}
- \int_{T-a}^{A-a} \beta(x,z)\pi(z)\frac{e^{(T-a)\Delta}y_0(x,z+a-T)}{\pi(z+a-T)}dz.
\end{align*}
By further manipulation, we obtain  
\begin{align}
    \frac{e^{-a\Delta}\bar{y}(x,a)}{\pi(a)} - \int_{0}^{A-a}\beta(x,z)\pi(z)\frac{e^{-a\Delta}\bar{y}(x,z+a)}{\pi(z+a)}dz 
= \int_{A-a}^{A} \beta(x,s)\pi(s)\frac{e^{(T-a)\Delta}y_0(x,s+a-T)}{\pi(s+a-T)}ds.
\end{align}
We now seek to establish an upper bound for this final integral. Specifically, we have
\[
\frac{\displaystyle\int_{A-a}^{A} \beta(x,s)\pi(s)\frac{e^{(T-a)\Delta}y_0(x,s+a-T)}{\pi(s+a-T)}ds}{\displaystyle\int_{0}^{a}\frac{e^{-s\Delta}\mathds{1}_{\{\omega\times[0,a_0]\}}(x,s)}{\pi(s)}ds} \leq \frac{\displaystyle\int_{A-a}^{A} \beta(x,s)\pi(s)\frac{e^{(T-a)\Delta}y_0(x,s+a-T)}{\pi(s+a-T)}ds}{a}.
\]
By a change of variable, setting \(z = s + a - T\), we rewrite the right-hand side as  
\[
\|\beta\|_{L^\infty(\Omega\times[0,A])} \frac{\displaystyle\int_{0}^{a} \left| e^{(T-a)\Delta}y_0(x,s+A-T)\right| ds}{a}.
\]
Then, by an application of the Hardy inequality, the function
\[
\frac{1}{a}\int_{0}^{a} \left| e^{(T-a)\Delta}y_0(x,s+A-T)\right| ds
\]
belongs to \(L^2(0,\epsilon;L^2(\Omega))\). Consequently, the entire expression in \eqref{eq2.36} belongs to \(L^2(0,\epsilon; L^2(\Omega))\), which implies that \(\bar{y}\in \mathcal{R}\).
\end{proof}
\begin{proof}[of Proposition \ref{th1.3}]
In accordance with Propositions \ref{pr2.3} and \ref{pr2.5}, we deduce that  
\begin{align}\label{eq2.42}
    \operatorname{Ran}(\mathcal{T}_T) \subset \operatorname{Ran}(\Phi_T),
\end{align}
which implies the null-controllability of system \eqref{eq1.5}. Consequently, since \eqref{eq2.42} holds, there exists a control \(v \in L^2([0,T];H)\) satisfying  
\begin{align}\label{eq2.43}
 \Phi_T v=-\mathcal{T}_{T} y_0.
\end{align}
By following an argument analogous to that in Proposition \ref{pr2.3}, we can construct a control fulfilling \eqref{eq2.43}. 

Specifically, for \(t-a < T-A\) we have \(\Phi_t v = 0\) and  
\begin{align}
    y(x,a,t) = \frac{\pi(a)}{\pi(a-t)}e^{t\Delta}y_0(x,a-t).
\end{align}

Next, for \(T-A \leq t-a \leq 0\), we obtain  
\begin{align*}
    y(x,T-(t-a),T) = \frac{\pi(T-(t-a))}{\pi(a-t)}e^{t\Delta}y_0(x,a-t) \\+ u(x,t-a)\int_{a-t}^{T-(t-a)} \frac{\pi(T-(t-a))e^{(T-(t-a)-s)\Delta}\mathds{1}_{\{\omega\times[0,a_0]\}}(x,s)}{\pi(s)}ds, \; \forall a\in [T,A],\, x\in \Omega.
\end{align*}
To ensure that \eqref{eq2.43} is satisfied, the control must be defined as  
\begin{align}
    v(x,a,t) = \frac{-\dfrac{e^{(t-a)\Delta}y_0(x,a-t)}{\pi(a-t)}}{\displaystyle\int_{a-t}^{T-(t-a)} \frac{e^{-z\Delta}\mathds{1}_{\{\omega\times[0,a_0]\}}(x,z)}{\pi(z)}dz}, \quad \text{for } T-A \leq t-a \leq 0.
\end{align}

For \(0 < t-a \leq T\), we have 
\begin{align}
    y(x,a,t) = \pi(a)e^{a\Delta}\eta(x,t-a) + u(x,t-a)\int_{0}^{a}\pi(a)\frac{e^{(a-s)\Delta}\mathds{1}_{\{\omega\times[0,a_0]\}}(x,s)}{\pi(s)}ds,
\end{align}
and equivalently,
\[
y(x,T-(t-a),T) = \pi(T-(t-a))e^{(T-(t-a))\Delta}\eta(x,t-a) + u(x,t-a)\int_{0}^{T-(t-a)}\frac{\pi(T-(t-a))e^{(T-(t-a)-s)\Delta}\mathds{1}_{\{\omega\times[0,a_0]\}}(x,s)}{\pi(s)}ds,
\]
with  
\[
\eta(x,t) = \int_{t+A-T}^{A}\beta(x,a)\pi(a)\frac{e^{t\Delta}y_0(x,a-t)}{\pi(a-t)}da.
\]
Thus, \eqref{eq2.43} is verified if we define
\[
v(x,a,t) = -\frac{\pi(T-(t-a))\eta(x,t-a)}{\pi(T-(t-a))\displaystyle\int_{0}^{T-(t-a)} \frac{e^{-z\Delta}\mathds{1}_{\{\omega\times[0,a_0]\}}(x,z)}{\pi(z)}dz}.
\]

In conclusion, we have derived a null control of the form
\begin{equation}\label{eq2.51}
v(x,a,t)=\begin{cases}
0, & \text{if } t-a<0,\quad \forall\, x\in\Omega,\\[1mm]
-\dfrac{\displaystyle\int_{t-a}^{A}\beta(x,z)\pi(z)\,\dfrac{e^{(t-a)\Delta}y_0(x,z-(t-a))}{\pi(z-(t-a))}\,dz}{\displaystyle\int_{0}^{A-(t-a)}\frac{e^{-z\Delta}\,\mathds{1}_{\{\omega\times[0,a_0]\}}(x,z)}{\pi(z)}\,dz}, & \text{if } t-a\geq 0,\quad \forall\, x\in\Omega,
\end{cases}
\end{equation}
with the corresponding controlled state given by
\begin{equation}\label{eq2.52}
y(x,a,t)=\begin{cases}
\dfrac{\pi(a)}{\pi(a-t)}e^{t\Delta}y_0(x,a-t), & \text{if } t-a<0,\quad \forall\, x\in\Omega,\\[1mm]
\pi(a)e^{a\Delta}\eta(x,t-a)\left[1-\dfrac{\displaystyle\int_{0}^{a}\frac{e^{-s\Delta}\,\mathds{1}_{\{\omega\times[0,a_0]\}}(x,s)}{\pi(s)}\,ds}{\displaystyle\int_{0}^{A-(t-a)}\frac{e^{-z\Delta}\,\mathds{1}_{\{\omega\times[0,a_0]\}}(x,z)}{\pi(z)}\,dz}\right], & \text{if } t-a>0,\quad \forall\, x\in\Omega,
\end{cases}
\end{equation}
which satisfies
\begin{equation}\label{eq2.53}
y(x,a,A)=0\qquad \text{a.e. } x\in\Omega,\quad \forall\, a\in [0,A].
\end{equation}

Taking \(v=0\) on the time interval \((A,T)\), equation \eqref{eq2.2} implies that
\[
y(\cdot,\cdot,T)=\mathcal{T}_T y_A(\cdot,\cdot),
\]
thus, we obtain the estimate
\begin{equation}\label{eq2.55}
\Vert y(\cdot,\cdot,T)\Vert_{H}\leq C \Vert y_A\Vert_{H}.
\end{equation}
Combining \eqref{eq2.55} with \eqref{eq2.53}, we deduce that for every \(T > A-a_0\),
\begin{equation}\label{eq2.56}
y(x,a,T)=0\qquad \text{a.e. } x\in\Omega,\quad  a\in [0,A].
\end{equation}

This establishes the null-controllability of the system \eqref{eq1.5}.
\end{proof}
\paragraph{\bf Acknowledgement}
The authors wish to thank Prof. Enrique Zuazua  for his comments, suggestions and for fruitful discussions.
\bibliographystyle{smfplain} 
\bibliography{biblio_2}

\providecommand{\bysame}{\leavevmode ---\ }
\providecommand{\og}{``}
\providecommand{\fg}{''}
\providecommand{\smfandname}{et}
\providecommand{\smfedsname}{\'eds.}
\providecommand{\smfedname}{\'ed.}
\providecommand{\smfmastersthesisname}{M\'emoire}
\providecommand{\smfphdthesisname}{Th\`ese}
\begin{thebibliography}{10}

\bibitem{ref26}
{\scshape B.~Ainseba {\normalfont \smfandname} M.~Langlais} -- {\og On a
  population dynamics control problem with age dependence and spatial
  structure\fg}, \emph{Journal of Mathematical Analysis and Applications}
  (2000), p.~455--474.

\bibitem{ref36}
{\scshape P.~Alessio {\normalfont \smfandname} Z.~Enrique} -- {\og Remarks on
  long time versus steady state optimal control\fg}, \emph{Springer
  International Publishing Switzerland} (2016).

\bibitem{ref77}
{\scshape E.~Alvarez} -- {\og Perturbing the boundary conditions of the
  generator of a cosine family\fg}, \emph{Semigroup Forum} (2011).

\bibitem{ref42}
{\scshape S.~Anita} -- \emph{Analysis and control of age-dependent population
  dynamics}, vol.~11, Kluwer Academic Publishers, 2000.

\bibitem{ref25}
{\scshape S.~Anita {\normalfont \smfandname} N.~Hegoburu} -- {\og Null
  controllability via comparison results for nonlinear age-structured
  population dynamics\fg}, \emph{Mathematics of Control, Signals, and Systems}
  (2019).

\bibitem{ref19}
{\scshape V.~Barbu, M.~Iannelli {\normalfont \smfandname} M.~Martcheva} -- {\og
  On the controllability of the lotka-mckendrick model of population
  dynamics\fg}, \emph{Journal of Mathematical Analysis and Applications}
  \textbf{253} (2001), p.~142--165.

\bibitem{ref61}
{\scshape M.~Barreau} -- {\og Stabilité et stabilisation de systèmes
  linéaires à l’aide d’inégalités matricielles linéaires\fg},
  \emph{HAL Id: hal-02111784} (2019).

\bibitem{ref14}
{\scshape A.~Bedr’Eddine} -- {\og Exact and approximate controllability of
  the age and space population dynamics structured model\fg},
  \emph{Mathematical ,analysis and applications} (2001), p.~11.

\bibitem{ref21}
{\scshape A.~Bedr’Eddine {\normalfont \smfandname} S.~Anita} -- {\og Internal
  exact controllability of the linear population dynamics with diffusion\fg},
  \emph{Electronic Journal of Differential Equations} \textbf{2004} (2004),
  no.~112, p.~1--11.

\bibitem{ref23}
{\scshape A.~Bedr'Eddine {\normalfont \smfandname} A.~Sebastian} -- {\og Local
  exact controllability of the age-dependent population dynamics with
  diffusion\fg}, \emph{Abstract and Applied Analysis} (2001), p.~357–368.

\bibitem{ref46}
{\scshape F.~M. Callier, J.~Winkin {\normalfont \smfandname} J.~L. Willems} --
  {\og Convergence of the time-invariant riccati differential equation and
  lq-problem: mechanisms of attraction\fg}, \emph{International Journal of
  Control} (2007).

\bibitem{ref20}
{\scshape M.~Debayan} -- {\og On the null controllability of the
  lotka-mckendrick\fg}, \emph{Mathematical Control and Related Fields} (2019).

\bibitem{ref11}
{\scshape R.~Dorfman, P.~A. Samuelson {\normalfont \smfandname} R.~M. Solow} --
  \emph{Linear programming and economic analysis}, Inc. : The McGraw-Hill Book
  Company, 1958.

\bibitem{ref52}
{\scshape C.~Esteve, H.~Kouhkouh, D.~Pighin {\normalfont \smfandname}
  E.~Zuazua} -- {\og The turnpike property and the long-time behavior of the
  hamilton-jacobi equation\fg}, \emph{hal-02873544v1} (2020).

\bibitem{ref33}
{\scshape C.~Esteve-Yague, B.~Geshkovski, D.~Pighin {\normalfont \smfandname}
  E.~Zuazua} -- {\og Turnpike in lipschitz-nonlinear optimal control\fg},
  \emph{arXiv:2011.11091v3[math.OC]} (2022).

\bibitem{ref32}
{\scshape B.~Geshkovski {\normalfont \smfandname} E.~Zuazua} -- {\og Turnpike
  in optimal control of pdes, resnets, and beyond\fg}, \emph{Nonlinear
  Analysis, Geometry and Applications} (2022).

\bibitem{ref53}
{\scshape L.~Grune, M.~Schaller {\normalfont \smfandname} A.~Schiela} -- {\og
  Sensitivity analysis for mpc and an exponential turnpike theorem for linear
  quadratic optimal control of general evolution equations\fg},  (2018).

\bibitem{ref43}
{\scshape G.~H. Hardy, J.~E. Littlewood {\normalfont \smfandname} G.~Polya} --
  \emph{Inequalities}, Cambridge at the University Press, Juillet 1934.

\bibitem{ref18}
{\scshape N.~Hegoburu, P.~Magal {\normalfont \smfandname} M.~Tucsnak} -- {\og
  Controllability with positivity constraints of the lotka-mckendrick
  system\fg}, \emph{Siam J. Control Optim.} \textbf{56} (2018), no.~2,
  p.~723--750.

\bibitem{ref24}
{\scshape N.~Hegoburu {\normalfont \smfandname} M.~Tucsnak} -- {\og Null
  controllability of the lotka-mckendrick system with spatial diffusion\fg},
  \emph{Mathematical Control and Related Fields} \textbf{8} (2018), no.~(3-4),
  p.~707--720.

\bibitem{ref58}
{\scshape E.~Hille {\normalfont \smfandname} R.~S. Phillips} --
  \emph{Fonctional analysis and semi-groups}, vol.~31, American Mathematical
  Society, 1957.

\bibitem{ref72}
{\scshape W.~Huyer} -- {\og Semigroup formulation and approximation of a linear
  age-dependent population problem with spatial diffusion\fg}, \emph{Semigroup
  Forum} (1994).

\bibitem{ref74}
{\scshape A.~Ibañez} -- {\og Optimal control of the lotka–volterra system:
  turnpike property and numerical simulations\fg}, \emph{JOURNAL OF BIOLOGICAL
  DYNAMICS} (2017).

\bibitem{ref31}
{\scshape C.~Jean-Michel} -- \emph{Control and nonlinearity}, vol. 136,
  American Mathematical Society, www.ams.org/bookpages/surv-136, 2007.

\bibitem{ref37}
{\scshape Z.~Jerzy} -- \emph{Mathematical control theory : An introduction},
  vol. 136, Birkhauser Boston, c/o Springer Science+Business Media LLC, 233
  Spring Street, New York, NY 10013, USA, edition 1995.

\bibitem{ref60}
{\scshape S.~Johansson} -- {\og Tools for control system design-stratification
  of matrix pairs and periodic riccati differential equation solvers\fg},
  \smfphdthesisname, Umea University, Fevrier 2009.

\bibitem{ref71}
{\scshape O.~Kavian {\normalfont \smfandname} O.~Traore} -- {\og Approximate
  controllability by birth control for a nonlinear population dynamics
  model\fg}, \emph{Control, Optimisation and Calculus of Variations} (2011).

\bibitem{ref57}
{\scshape C.~Kenne, G.~Leugering {\normalfont \smfandname} G.~Mophou} -- {\og
  Optimal control of a population dynamics model wuth missing birth rate\fg},
  \emph{Siam J. Control Optim.} \textbf{58} (2020), no.~3, p.~1289--1313.

\bibitem{ref78}
{\scshape R.~N. Klaus-Jochen~Engel} -- \emph{One parameter semigroups for
  linear evolution equations}, Springer Verlag New York, 1991.

\bibitem{ref54}
{\scshape G.~Lance, E.~Trélat {\normalfont \smfandname} E.~Zuazua} -- {\og
  Shape turnpike for linear parabolic pde models\fg}, \emph{Systems and Control
  Letters} (2020).

\bibitem{ref59}
{\scshape X.~Li {\normalfont \smfandname} J.~Yong} -- \emph{Optimal control
  theory for infinite dimensional systems}, 1st edition \smfedname, Library of
  Congress Cataloging-in-Publication Data, 1995.

\bibitem{ref79}
{\scshape D.~Maity, M.~Tucsnak {\normalfont \smfandname} E.~Zuazua} -- {\og
  Controllability of a class of infinite dimensional systems with age
  structure\fg}, \emph{hal-01964612v1} (2018).

\bibitem{ref3}
{\scshape D.~Maity, M.~Tuksnak {\normalfont \smfandname} E.~Zuazua} -- {\og
  Controllability and positivity constraints in population dynamics with age
  structuring and diffusion\fg}, \emph{Journal des Mathématiques Pures et
  Appliquées} \textbf{129} (2019), no.~24, p.~153--179.

\bibitem{ref39}
{\scshape L.~W. Mckenzie} -- {\og Turnpike theorems for a generalized leontief
  mode\fg}, \emph{The Econometric Society} \textbf{31} (1963), no.~1/2,
  p.~165--180.

\bibitem{ref64}
{\scshape D.~B. Nath} -- \emph{Numerical methods for linear control systems
  design and analysis}, Elsevier Academic Press, 2004.

\bibitem{ref49}
{\scshape A.~Ouedraogo {\normalfont \smfandname} O.~Traore} -- {\og Optimal
  control for a nonlinear population dynamics problem\fg}, \emph{Portugallae
  Mathematica} \textbf{62} (2005).

\bibitem{ref44}
{\scshape D.~Pighin {\normalfont \smfandname} N.~Sakamoto} -- {\og The turnpike
  with lack of observability\fg}, \emph{hal-02908073} (2020).

\bibitem{ref5}
{\scshape A.~Porretta {\normalfont \smfandname} E.~Zuazua} -- {\og Long time
  versus steady state optimal control\fg}, \emph{Siam J. Control optim.}
  \textbf{51} (2013), no.~6, p.~4242–4273.

\bibitem{ref63}
{\scshape A.~J. Pritchard {\normalfont \smfandname} J.~Zabczyk} -- {\og
  Stability and stabilizability of infinite dimensional systems\fg}, \emph{SIAM
  REVIEW} (1981).

\bibitem{ref76}
{\scshape A.~R. Said~Hadd, Rosanna~Manzo} -- {\og Unbounded perturbations of
  the generator domain\fg}, \emph{Discrete and continuous dynamical systems}
  (2015).

\bibitem{ref27}
{\scshape Y.~Simporé, B.~M. Ndiaye, O.~Traoré {\normalfont \smfandname}
  D.~Seck} -- {\og Null controllability by birth control for a population
  dynamics model\fg}, \emph{Nonlinear Analysis, Geometry and Applications}
  (2022).

\bibitem{ref38}
{\scshape E.~Trelat, C.~Zhang {\normalfont \smfandname} E.~Zuazua} -- {\og
  Steady-state and periodic exponential turnpike property for optimal control
  problems in hilbert spaces\fg}, \emph{Control Optim.} \textbf{56} (2018),
  no.~2, p.~1222--1252.

\bibitem{ref56}
{\scshape E.~Trélat {\normalfont \smfandname} C.~Zhang} -- {\og Integral and
  measure-turnpike properties for infinite-dimensional optimal control
  systems\fg}, \emph{Mathematics of Control, Signals, and Systems, Springer
  Verlag} (2018), p.~30:3.

\bibitem{ref55}
{\scshape E.~Trélat {\normalfont \smfandname} E.~Zuazua} -- {\og The turnpike
  property in finite-dimensional nonlinear optimal control\fg}, \emph{Journal
  of Differential Equations} \textbf{258} (2015), p.~81--114.

\bibitem{ref48}
{\scshape M.~Tucsnak {\normalfont \smfandname} G.~Weiss} -- \emph{Observation
  and control for operator semigroups}, ResearchGate, December 12, 2009.

\bibitem{ref75}
{\scshape J.~C. Willems} -- {\og Dissipative dynamical systems, part i ;
  general theory\fg}, \emph{IEEE CONTROL SYSTEMS} (2022).

\end{thebibliography}
\end{document}